\documentclass[10pt, pagebackref, a4paper]{amsart}
\usepackage{hyperref}

\makeatletter
\providecommand{\SG@adddot}{.}
\@ifundefined{href}{%
   \providecommand{\href}[2]{}%
   \renewcommand{\SG@adddot}{}}{}
\def\tospace#1{\@tospace#1 \tospace@delimiter}
\def\@tospace#1 #2\tospace@delimiter{#1}
\def\MR#1{\edef\MR@help{{http://www.ams.org/mathscinet-getitem?mr=\tospace{#1}}{\tospace{#1}}}%
\expandafter\href\MR@help\SG@adddot}
\makeatother

\providecommand*{\backref}{}
\providecommand*{\backrefalt}{}
\renewcommand*{\backref}[1]{}
\renewcommand*{\backrefalt}[4]{%
	\ifcase #1 %
	\or
	  Cited page #2.
	\else
	  Cited pages #2.
	\fi
}
\providecommand{\texorpdfstring}[2]{#1}

\newcommand{\BB}{\mathcal{B}}
\newcommand{\CC}{\mathcal{C}}
\newcommand{\DD}{\mathcal{D}}
\newcommand{\EE}{\mathcal{E}}
\newcommand{\FF}{\mathcal{F}}
\newcommand{\FFF}{\mathbf{F}}
\newcommand{\HH}{\mathcal{H}}
\newcommand{\HHH}{\mathbf{H}}
\newcommand{\JJ}{\mathcal{M}}
\newcommand{\KK}{\mathcal{K}}
\newcommand{\LL}{\mathcal{L}}
\newcommand{\MM}{\mathcal{M}}

\newcommand{\TT}{\mathcal{T}}
\newcommand{\TTT}{\mathbb{T}}
\newcommand{\UU}{\mathcal{U}}
\DeclareMathOperator{\Leb}{Leb}
\DeclareMathOperator{\dLeb}{dLeb}
\DeclareMathOperator{\Disc}{Disc}
\newcommand{\ii}{\mathbf{i}}

\newcommand{\id}{{\rm id}}

\newcommand{\real}{\mathbb{R}}
\newcommand{\R}{\real}
\newcommand{\complex}{\mathbb{C}}
\newcommand{\integer}{\mathbb{Z}}
\newcommand{\Z}{\integer}

\newcommand{\Cs}{C_{\#}}
\newcommand{\epsilons}{\epsilon_{\#}}
\newcommand{\st}{\;|\;}

\DeclareMathOperator{\Jac}{Jac}
\newcommand{\ZZ}{\mathcal{Z}}
\newcommand{\VV}{\mathcal{V}}
\renewcommand{\tilde}{\widetilde}
\newcommand{\leqc}{\leq_c{}}

\newcommand{\coloneqq}{\mathrel{\mathop:}=}
\newcommand{\dd}{\, {d}}
\newcommand{\ress}{r_{\rm ess}}
\newcommand{\bm}{{\zeta}}
\newcommand{\bphi}{{\Phi}}
\newcommand{\brho}{{\rho}}
\newcommand{\bell}{{\zeta'}}
\newcommand{\regg}{\gamma}
\newcommand{\LSRB}{\LL}

\newcommand{\norm}[2]{\left\| #1 \right\|_{#2}}
\newcommand{\nor}[1]{\left\| #1 \right\|}
\DeclareMathOperator{\Card}{Card}
\DeclareMathOperator{\supp}{supp}

\newtheorem{proposition}{Proposition}[section]
\newtheorem{lemma}[proposition]{Lemma}

\newtheorem{theorem}[proposition]{Theorem}
\newtheorem{corollary}[proposition]{Corollary}

\newtheorem{definition}[proposition]{Definition}

\theoremstyle{remark}

\newtheorem{remark}[proposition]{Remark}

\numberwithin{equation}{section}

\begin{document}

\title[Spaces for piecewise cone hyperbolic maps]{Banach
spaces for  piecewise cone hyperbolic maps}
\author{Viviane Baladi and S\'{e}bastien Gou\"{e}zel}

\address{D.M.A., UMR 8553, \'{E}cole Normale Sup\'{e}rieure,  75005 Paris, France}
\email{viviane.baladi@ens.fr}
\address{IRMAR, CNRS UMR 6625,
Universit\'{e} de Rennes 1, 35042 Rennes, France}
\email{sebastien.gouezel@univ-rennes1.fr}
\thanks{We are very grateful to Carlangelo Liverani for  conversations, encouragements,
and showing us a preliminary version of a manuscript on coupled
Anosov diffeomorphisms. Many thanks to P\'{e}ter B\'alint for
important conversations and patient explanations on billiards,
and to Duncan Sands for his enlightening comments on Lozi
maps. Many thanks also to the anonymous referees for their very
accurate comments.
A crucial part of this work was done during the 2008 Semester
on Hyperbolic Dynamical Systems in the Schr\"{o}dinger Institut in
Vienna: We express our gratitude to the organisers. VB is
partially supported by ANR-05-JCJC-0107-01.}
\date{February 12, 2010}

\begin{abstract}
We consider piecewise cone hyperbolic systems satisfying a
bunching condition and we obtain a bound on the essential
spectral radius of the associated weighted transfer operators
acting on anisotropic Sobolev spaces. The bunching condition is
always satisfied in dimension two, and our results give a
unifying treatment of the work of Demers-Liverani
\cite{demers_liverani} and our previous work
\cite{baladi_gouezel_piecewise}. When the complexity is
subexponential, our bound implies a spectral gap for the
transfer operator corresponding to the physical measures in
many cases (for example if $T$ preserves volume, or if the
stable dimension is equal to $1$ and the unstable dimension is
not zero).
\end{abstract}

\maketitle
\section{Introduction}

The ``spectral'' or ``functional'' approach to study
statistical properties of dynamical systems with enough
hyperbolicity, originally limited to one-dimensional dynamics,
has  greatly expanded its range of applicability  in recent
years. The following spectral gap result of\footnote{Various
improvements of this result have been  obtained since then,
\cite{gouezel_liverani, GL_Anosov2, bt_aniso, bt_zeta}, in
particular in the Axiom A setting.}
 Blank--Keller--Liverani \cite{bkl_spectre_anosov} appeared in 2002:

\begin{theorem}\label{meta}
Let $T : X\to X$ be a $C^3$ Anosov diffeomorphism on a compact
Riemannian manifold, with a dense orbit. Define a bounded
linear operator by
\begin{equation}\label{lastar}
\LL \omega = \frac{\omega \circ T^{-1}}{|\det DT \circ T^{-1}|}\, ,
\quad \omega \in L^\infty(X)\, .
\end{equation}
Then there exist a Banach space $\BB$ of distributions on $X$,
containing $C^\infty(X)$, and a bounded operator on $\BB$,
coinciding with $\LL$ on $\BB \cap L^\infty(X)$ and  denoted
also by $\LL$, with the following  properties: The spectral
radius of $\LL$ on $\BB$ is equal to one, the essential
spectral radius of $\LL$ on $\BB$ is strictly smaller than one,
$\LL$ has a fixed point in $\BB$. Finally, $1$ is the only
eigenvalue on the unit circle, and it is simple.
\end{theorem}

It is a remarkable fact that ``Perron-Frobenius-type'' spectral
information as in the above theorem (possibly with a nonsimple
real maximal eigenvalue of finite multiplicity and other
eigenvalues on the unit circle) gives simpler proofs of many
known theorems, but also new information. Among these
consequences, let us just mention: Existence of finitely many
physical measures whose basins have full measure (working with
slightly more general transfer operators, one can treat other
equilibrium states), exponential decay of correlations for
physical measures and H\"{o}lder observables, statistical and
stochastic stability, linear response and the linear response
formula, central and local limit theorems, location of the
poles of dynamical zeta functions and zeroes of dynamical
determinants, smooth Anosov systems with holes, etc.
(We just recall that the dual of $\LL$ preserves Lebesgue measure,
so that the fixed point of $\LL$ corresponds to the physical
measure. See \cite{bt_aniso} and \cite{GL_Anosov2}.)

One of the advantages of this ``functional approach'' is that
it bypasses the construction of Markov partitions and the need
to introduce artificial ``one-sided'' expanding endomorphisms
(such endomorphisms only retain a  small part of the smoothness
of the original hyperbolic diffeomorphism).

Billiards with convex scatterers, also called Sinai billiards,
are among the most natural and interesting dynamical systems.
They are uniformly hyperbolic, preserve Liouville measure, but
they are only piecewise smooth. Analyzing the difficulties
posed by the singularities has been an important challenge for
mathematicians, and it is only in 1998 that L.-S. Young
\cite{lsyoung_annals} proved that the Liouville measure enjoys
exponential decay of correlations for two-dimensional Sinai
billiards (under a finite horizon condition, which was shortly
thereafter removed by Chernov \cite{chernov_decay}). It should
be noted that these results were in fact obtained for a
discrete-time version of the billiard flow. Indeed the question
of whether the original two-dimensional continuous-time Sinai
billiard enjoys decay of correlations is to this day still
open. (Chernov \cite{chernov_stretched_flow} recently obtained
stretched exponential upper bounds.) It is well known that the
continuous-time case is much more difficult, and it seems that
the ideas of Dolgopyat \cite{dolgopyat_decay} which were
exploited in several smooth hyperbolic situations are not
compatible with the tools used in \cite{lsyoung_annals} for
example. We believe that a new, ``functional,'' proof (via a
spectral gap result for the transfer operator \eqref{lastar} on
a suitable anisotropic Banach space of distributions)  of
exponential decay of correlations for {\it discrete-time}
surface Sinai billiards  will be a key stepping stone towards
the  expected proof of exponential decay of correlations for
the {\it continuous-time} Sinai billiards.

The recent paper of Demers-Liverani \cite{demers_liverani} was
a first breakthrough in this direction, as we explain next.
Since none of the spaces of \cite{gouezel_liverani, GL_Anosov2,
bt_aniso, bt_zeta} behave well with respect to multiplication
by characteristic functions of sets, they cannot be used for
systems with singularities. Demers--Liverani
\cite{demers_liverani} therefore introduced some new Banach
spaces, on which transfer operators associated to
two-dimensional piecewise hyperbolic systems admit a spectral
gap. However, the construction and the argument of
\cite{demers_liverani} are quite intricate, in particular,
pieces of stable or unstable manifolds are iterated by the
dynamics, and the way they are cut by the discontinuities has
to be studied in a very careful way, in the spirit of
\cite{lsyoung_annals} and  \cite{chernov_decay}. As a
consequence, adapting the approach in \cite{demers_liverani} to
billiards (which are not piecewise hyperbolic, stricto sensu,
because their derivatives blow up along the singularity lines)
is daunting.

Another progress in the direction of a modern proof of
exponential decay of correlations for discrete-time billiards
is our previous paper \cite{baladi_gouezel_piecewise}. There,
we showed that ideas of Strichartz \cite{strichartz} imply that
classical  anisotropic Sobolev spaces $H_p^{t,s}$  in the
Triebel-Lizorkin class \cite{triebel_III} (Definition ~
\ref{space}, these spaces had been introduced in dynamics in
\cite{baladi_Cinfty}) are suitable for piecewise hyperbolic
systems, under the condition that the system admits a smooth
(at least $C^1$) stable foliation. Unfortunately, although it
holds for several nontrivial examples, this condition is pretty
restrictive: In general, the foliations are only measurable!

In the present paper, we consider piecewise smooth piecewise
hyperbolic dynamics. We are able to remove the assumption of
smoothness of the stable  foliation, whenever the hyperbolicity
exponents of the system satisfy a {\it bunching} condition (see
\eqref{bunch} and \eqref{bunch2} below). This condition is
rather standard in smooth hyperbolic dynamics, where it ensures
that the dynamical foliations are $C^1$ instead of the weaker
H\"{o}lder condition which holds in full generality (see
\cite{hirsch_pugh_shub}, or, e.g., \cite{katok}). The bunching
condition is always satisfied in codimension one (in
particular, it holds in dimension two, so that our results
apply to physical measures of all surface piecewise hyperbolic
systems previously covered in \cite{lsyoung_annals} or
\cite{demers_liverani}, in particular to hyperbolic Lozi maps
possessing a compact invariant domain, see Appendix
\ref{app_physical}). The present paper requires the dynamics to
be $C^{1+\alpha}$ on each (closed) domain of smoothness, and
therefore does not apply directly to discrete-time Sinai
billiard. However, we expect that it will be possible to adapt
the methods here to obtain the desired functional proof of
exponential decay of correlations for two-dimensional Sinai
billiards. We shall use the terminology ``cone-hyperbolic'' to
stress that hyperbolicity is defined in terms of cones and that
there is \emph{a priori} no invariant stable distribution,
contrary to our previous paper \cite{baladi_gouezel_piecewise}.

We use the Triebel spaces $H_p^{t,s}$ as  building blocks in
the construction of our new Banach spaces $\HHH_{p}^{t,s}(R)$
(Definition  ~\ref{defnorm})
and $\HHH$ (see \eqref{normeHnR}). As
a consequence, we may exploit, as we did in
\cite{baladi_gouezel_piecewise}, the rich existing theory (in
particular regarding interpolation), and use again the results
of Strichartz \cite{strichartz}.
%Another  feature
%common to the present work and \cite{BG1} is  working with a scale of Sobolev spaces gives a lot
%of flexibility, and unifies the present piecewise hyperbolic theory and previous
%works on piecewise expanding systems.

The new ingredient with respect to
\cite{baladi_gouezel_piecewise} is that we define our norm by
considering the Triebel norm in $\real^d$ through suitable $C^1$ charts,
taking now the {\it supremum} over {\it all} cone-admissible
charts  $\FF$ (Definition ~ \ref{deffol}). We use the bunching
assumption to show that the family is invariant under iteration
(Lemma ~ \ref{lemcompose}). {\it Indeed, this is how we avoid
the necessity for a smooth stable foliation.} As in
\cite{baladi_gouezel_piecewise}, we do not iterate single
stable or unstable manifolds (contrary to \cite{lsyoung_annals,
chernov_decay, demers_liverani}), and we do not need to match
nearby stable or unstable manifolds: Everything follows from an
appropriate functional analytic framework.

Our main result, Theorem ~ \ref{MainTheorem}, is an upper bound
on the essential spectral radius of weighted transfer operators
associated to cone hyperbolic systems satisfying the bunching
condition and acting on a Banach space $\HHH$ of anisotropic
distributions. If the complexity growth (as measured by
\eqref{cplx}) is subexponential, and if either $\det DT\equiv
1$, or $d_s=1$ and $d_u >0$, then one can always choose the
Banach space so that the transfer operator \eqref{lastar} has
essential spectral radius strictly smaller than $1$, and thus a
spectral gap. This spectral gap property gives finiteness and
exponential mixing (up to a finite period) of the physical
measures (see e.g.~Theorem 33 in
\cite{baladi_gouezel_piecewise}, or its generalization below,
Theorem~ \ref{thm_physical}).

Let us mention here that all existing results on piecewise
hyperbolic systems, including the present one, require some
kind of transversality condition between the discontinuity
hypersurfaces  and the stable or unstable dynamical directions
or cones (see Definition \ref{transs}). (This condition is
satisfied for billiards, modulo Remark ~ \ref{forbilliards}.)

The paper is organized as follows. In Section~ \ref{22} we
define formally the dynamical systems for which our results
hold and the anisotropic spaces $\HHH$ on which the transfer
operator will act: Subsection ~ \ref{setting} contains the
assumptions on the dynamics and the statement of our main
result, Theorem ~ \ref{MainTheorem}. In Subsection~
\ref{sec:def_fol}, we recall the definition of the Triebel
spaces $H^{t,s}_p$ and we define the cone-admissible foliations
$\FF(C_0,C_1)$, depending on two parameters $C_0$ and $C_1$
that should be suitably chosen. In Subsection ~ \ref{spaces},
we combine these two ingredients, together with a ``zoom'' by a
large factor $R>1$, to construct the Banach spaces of
distributions $\HHH_p^{t,s}(R, C_0, C_1)$. Subsection ~
\ref{reducr} contains a technical step which reduces our main
result to a more convenient form, Theorem ~
\ref{MainTheorembis}, constructing along the way the final
Banach spaces $\HHH$ from the $\HHH_{p}^{t,s}(R, C_0, C_1)$.

Section ~ \ref{invv} is devoted to the proof of invariance of
the class $\FF$ of admissible foliations. This is the heart of
our argument, and the main new technical ingredient is Lemma ~
\ref{lemcompose}. Its proof is based on the usual
Hadamard-Perron graph transform ideas (see \eqref{heart}--\eqref{heart''}),  but
requires to be spelt out in full detail in order to discover
the appropriate conditions in Definition ~ \ref{deffol}.

Section ~  \ref{sec:local} contains various results on the
local spaces $H^{t,s}_p$, in particular the corresponding
``Leibniz'' (Lemma ~ \ref{Leib}) and ``chain-rule'' (Lemmas
~\ref{CompositionDure} and \ref{lemcomposeD1alpha}) estimates,
and the fact that characteristic functions of appropriate sets
are bounded multipliers (Lemma ~ \ref{lem:multiplier}). These
results are mostly adapted from
\cite{baladi_gouezel_piecewise}.  Subsection ~ \ref{basicc}
also contains a compactness embedding statement for spaces
$\HHH_{p}^{t,s}(R, C_0, C_1)$ (Lemma ~ \ref{embed}) which is
crucial for our  Lasota-Yorke-type estimate in the proof of our
main result.

Finally, Section ~ \ref{mainsec} contains the proof of Theorem
~ \ref{MainTheorembis}.

Four appendices contain some complements: Appendices \ref{apA}
and \ref{apB} contain useful technical results, Appendix
\ref{app_general} describes some extensions of our main result
(which allow us in particular to sometimes weaken our transversality
assumption),
and Appendix \ref{app_physical} gives consequences  concerning physical measures
of our
main result and its extension.

Note that the methods in this paper do not allow to exploit the
additional smoothness available if $T$ is Anosov or Axiom A and
$C^ r$ for $r> 2$ (even if they satisfy the bunching
condition), contrarily to \cite{gouezel_liverani, GL_Anosov2,
bt_aniso, bt_zeta}. The present work is thus complementary to
the approach of \cite{gouezel_liverani, GL_Anosov2, bt_aniso,
bt_zeta} which gives more information in the smooth case (but
fails when there are singularities).

%%%%%%%%%%%%%%%%%%%%%%%%%%%%%%%%%%%%%%%%%%%%%%%%%%%%%%%%%%%%%%%%%%%%
\section{Definitions and statement of the spectral theorem}

\label{22}

\subsection{The main result}
\label{setting}
Let $X$ be a Riemannian manifold of dimension $d \ge 2$ without
boundary, and let $X_0$ be a compact subset of $X$. We view
$1\le d_s\le d-1$ and $d_u=d-d_s\ge 1$ as being fixed integers,
so constants may depend on these numbers\footnote{Our methods
also work when $d_s=0$ or $d_u=0$, but they do not improve on
the results of \cite{baladi_gouezel_piecewise} since the stable
and unstable manifolds are automatically smooth in this case.}.
We call $C^1$ hypersurface a codimension-one $C^1$ submanifold
of $X$, possibly  with boundary. We say that a function $g$ is
$C^r$ for $r >0$ if $g$ is $C^{[r]}$ and all partial
derivatives of order $[r]$ are $r-[r]$-H\"{o}lder. The norm of a
vector (in the tangent space of $X$, or in $\R^d$) will be
denoted by $|v|$.

Hyperbolicity will be defined in terms of cones, and we shall
need the cones to satisfy some  form of convexity (in \eqref{useconv}).
Even the simplest linear cone $|x|^2 \le |y_1|^2+ |y_2|^2$ in $\real^3$
is not convex in the usual sense (it contains $(1,1,0)$ and $(1,0,1)$ but not
$(1,1/2,1/2)$).
Therefore, we introduce the following definition:

\begin{definition}\label{defconvtr}
A cone of dimension $d'\in [1,d-1]$ in $\real^d$ is a closed
subset $C$ of $\R^d$ with nonempty interior, invariant under
scalar multiplication, such that $d'$ is the maximal dimension
of a vector subspace included in $C$.

A cone $C$ of dimension $d'$ is transverse to a vector subspace
$E$ of $\R^d$ if $E$ contains a subspace of dimension $d-d'$
which intersects $C$ only at $0$.

A cone $C$ is convexly transverse to a vector subspace $E$ if
$C$ is transverse to $E$ and, additionally, for all $z\in
\R^d$, $C\cap (E+z)$ is convex.

Two cones $C_u$ and $C_s$, of respective dimensions $d_{u}$ and
$d_s$, with $d_{u}+d_s=d$, are convexly transverse if $C_u\cap
C_s=\{0\}$, for any vector subspace $E_s \subset C_s$ the cone
$C_{u}$ is convexly transverse to $E_s$, and for any vector
subspace $E_u \subset C_u$ the cone $C_{s}$ is convexly
transverse to $E_u$.
\end{definition}

We claim that if $A:\real^{d_s} \to \real^{k}$ is a nonzero
linear map then the set $C_A=\{(x,y)\in \real^{d_u}\times
\real^{d_s} \st |x|\leq |Ay|\}$ (which obviously contains the
$d_s$-dimensional vector subspace $\{(0,y)\}$) is a
$d_s$-dimensional cone which is convexly transverse to
$\{(x,0)\}$. See Appendix ~\ref{apB} for the easy proof of this
claim. It follows that, if $C_A$ and $C_{A'}$ are cones in
$\real^d$ associated (not necessarily for the same coordinates)
to nonzero linear maps
 $A:\real^{d_s} \to \real^{k_1}$ and $A':\real^{d_u} \to \real^{k_2}$,
then $C_A$ and $C_{A'}$ are convexly transverse if and only if
$C_A\cap C_{A'}=\{0\}$. Definition~\ref{defconvtr} is slightly
more flexible than such linear cones. More importantly, it
sheds light on the essence of the convexity assumption.

\begin{definition}
[Piecewise $C^{1+\alpha}$ cone hyperbolic maps] \label{PiecHyp}
Let $\alpha\in (0,1]$.   A piecewise $C^{1+\alpha}$ (cone)
hyperbolic map is a map $T: X_0 \to X_0$ such that there exist
finitely many pairwise disjoint open subsets $(O_i)_{i\in I}$,
covering Lebesgue almost all $X_0$, so that each $\partial O_i$
is a finite union of $C^1$ hypersurfaces, and so that for each
$i\in I$:

(1) There exists a $C^{1+\alpha}$ map $T_i$ defined on a
neighborhood $\tilde O_i$ of $\overline{O_i}$ in $X$, which is
a diffeomorphism onto its image and such that
$T|_{O_i}=T_i|_{O_i}$.

(2) There exist two families of convexly transverse cones
$\CC^{(u)}_i(q)$ and $\CC^{(s)}_i(q)$ in the tangent space
$\TT_q X$, depending continuously on $q\in \overline{O_i}$, so
that $\CC^{(u)}_i(q)$ is $d_u$-dimensional and $\CC_i^{(s)}(q)$
is $d_s$-dimensional, and such that:

(2.a) For each  $q\in \overline{ O_i}\cap T_i
^{-1}(\overline{O_j})$, then $DT_i(q) \CC^{(u)}_i(q) \subset
\CC^{(u)}_j(T_i (q))$, and there exists $\lambda_{i,u}(q)>1$
such that
  \begin{equation*}
  |DT_i (q) v|\ge \lambda_{i,u}(q) |v| \, ,\forall v\in \CC^{(u)}_i(q) \, .
  \end{equation*}

(2.b) For each  $q\in \overline{ O_i}\cap T_i
^{-1}(\overline{O_j})$, then $DT_i^{-1}(T_i(q))
\CC^{(s)}_j(T_i(q)) \subset \CC^{(s)}_i(q)$, and there exists
$\lambda_{i,s}(q)\in (0,1)$ such that
  \begin{equation*}
  |DT_i^{-1} (T_i(q))v|\ge \lambda_{i,s}^{-1}(q) |v|\, ,
  \forall v\in \CC^{(s)}_j(T_i(q)) \, .
  \end{equation*}
\end{definition}

Note that we do not assume that $T$ is continuous or injective
on $X_0$.

We introduce some notation. For $n\ge 1$, and $\ii
=(i_0,\dots,i_{n-1})\in I^n$ we let $T_\ii^n=T_{i_{n-1}}\circ
\dots \circ T_{i_0}$, which is defined on a neighborhood of
$\overline{O_\ii}$, where $O_{(i_0)}=O_{i_0}$, and
  \begin{equation}
  O_{(i_0,\dots,i_{n-1})}=\{q\in O_{i_0} \st T_{i_0}(q)\in
  O_{(i_1,\dots,i_{n-1})}\}\, .
  \end{equation}
Denote by $\lambda_{\ii, s}^ {(n)}(q)<1$ and
$\lambda_{\ii,u}^{(n)}(q)>1$ the weakest contraction and
expansion coefficients of $T^n_{\ii}$ at $q$, and by
$\Lambda_{\ii,s}^{(n)}(q)\leq \lambda_{\ii, s}^ {(n)}(q)$ and
$\Lambda_{\ii,u}^{(n)}(q) \geq \lambda_{\ii,u}^{(n)}(q)$ its
strongest contraction and expansion coefficients. We put
  \begin{equation*}
  \lambda_{s,n}(q)=\sup_{\ii} \lambda_{\ii, s}^ {(n)}(q)<1\, ,
  \quad \lambda_{u,n}(q)=\inf_{\ii}\lambda_{\ii,u}^{(n)}(q)>1 \, ,
  \end{equation*}
where the infimum and the supremum are restricted to those
$\ii$ such that $q\in \overline O_\ii$.

As is usual in piecewise hyperbolic settings, we shall require
a transversality assumption on the discontinuity
hypersurfaces\footnote{This condition is unrelated to the
``convex transversality" assumption on the cones!}:

\begin{definition}[Transversality condition]\label{transs}
Let $T$ be a piecewise $C^{1+\alpha}$ hyperbolic map. We say
that $T$ satisfies the transversality condition if  each
$\partial O_i$ is a finite union of $C^1$ hypersurfaces $K_{i,
k}$ which are everywhere transverse to the stable cones, i.e.,
for all $q\in K_{i,k}$, $\TT_q K_{i,q}$ contains a
$d_u$-dimensional subspace that intersects $\CC_i^{(s)}(q)$
only at $0$.
\end{definition}

\begin{remark}[Transversality in the image]\label{forbilliards}
If the cone field is continuous (i.e., it does not really
depend on $i$, as is the case with Sinai billiards), then one
can weaken this requirement, by demanding only that the images
$T(K_{i,q})$ are transverse to the stable cone (see Appendix
\ref{app_general} for details). When the cone fields are not
globally continuous, the stronger requirement in Definition
\ref{transs} is necessary to ensure that $C^1$ functions belong
to the Banach space $\HHH$ we shall construct below (see the
argument after Definition ~\ref{defnorm}).
\end{remark}

To estimate dynamical complexity, we define the
$n$-complexities at the beginning and at the end:
  \begin{equation}
  \label{cplx}
  D^b_n=\max_{q\in X_0} \Card \{ \ii \in I^n \st q \in
  \overline{O_{\ii}} \} \, , \quad
  D^e_n=\max_{q\in X_0} \Card \{ \ii \in I^n \st q \in
  \overline{T^n(O_{\ii})} \} \, .
  \end{equation}
(For a globally invertible map $T$ we have $D^e_n(T,\{ O_i,
i\})=D^b_n(T^{-1}, \{T(O_i), i\})$. For $T(x)=2x$ mod $1$ on
$[0,1]$ we have $D^e_n=2^n$, but fortunately this quantity
plays no role for the transfer operator associated to $g=|\det
DT|^{-1}$ when $d_s=0$, up to taking $p$ close enough to $1$ in
Theorem~\ref{MainTheorem}.)

Our main result can now be stated (all Jacobians in this paper
are relative to Lebesgue measure, and  $|\det DT|$ denotes the
Jacobian of $T$):

\begin{theorem}[Spectral theorem]\label{MainTheorem}
Let  $\alpha \in (0,1]$, and let $T$ be a piecewise
$C^{1+\alpha}$ cone hyperbolic map satisfying the
transversality condition. Assume in addition the following {\it
bunching\footnote{Condition
\eqref{bunch} always holds if $d_u=1$.} condition:} For some $n>0$,
  \begin{equation}
  \label{bunch}
  \sup_{\ii \in I^n,\, q\in \overline O_\ii} \frac{
  \lambda^{(n)}_{\ii,s}(q)^\alpha \Lambda^{(n)}_{\ii,u}(q)}{\lambda^{(n)}_{\ii,u}(q)  }
  <1\, .
  \end{equation}
Let $\beta \in (0,\alpha)$ be small enough so that
   \begin{equation}
  \label{bunch2}
  \sup_{\ii \in I^n,\, q\in \overline O_\ii} \frac{
  \lambda^{(n)}_{\ii,s}(q)^{\alpha-\beta} \Lambda^{(n)}_{\ii,u}(q)^{1+\beta} }{\lambda^{(n)}_{\ii,u}(q)  }  <1
  \, .
\end{equation}
Let $1<p<\infty$ and let $t, s \in \real$ be so that
\begin{equation}\label{pstcond}
1/p-1<s<0<t<1/p\, ,
\quad -\beta<t-|s| <0\, , \quad \alpha t+|s| < \alpha\, .
\end{equation}
Then there exists a space $\HHH=\HHH(p,t,s)$ of distributions
on $X$, containing $C^1$ and in which $L^\infty\cap \HHH$ is
dense, and such that for any function $g:X_0\to \complex$  so
that the restriction of $g$ to each $O_i$ admits a $C^{\regg}$
extension to $\overline{O_i}$ for some $\regg>t+|s|$,  the
operator $\LL_g$ defined on $L^\infty$ by
\[
(\LL_g \omega)(q)=\sum_{T(q')=q}
g(q')\omega(q')
\]
extends continuously to $\HHH$. Moreover, its essential
spectral radius on $\HHH$ is at most
  \begin{equation}\label{thebest}
  \lim_{n\to\infty}  (D_{n}^b)^{1/(pn)} \cdot (D_{n}^e)^{(1/n)(1-1/p)}
  \cdot \norm{g^{(n)}|\det DT^{n}|^{1/p}
  \max(\lambda_{u,n}^{-t},\lambda_{s,n}^{-(t-|s|)})}{L^\infty}^{1/n} \, ,
  \end{equation}
where we set $g^{(n)}(q)=  \prod_{k=0}^{n-1} g(T^k(q))$, for $n
\ge 1$.
\end{theorem}

Our proof does not give good bounds on the spectral radius of
$\LL_g$ on $\HHH$. However, if $g=|\det DT|^{-1}$ and the bound
in \eqref{thebest} is $<1$, then Theorem 33 in
\cite{baladi_gouezel_piecewise} implies that the spectral
radius is equal to $1$, and that $T$ has finitely many physical
measures, attracting Lebesgue almost every point of the
manifold. For details,
we refer the reader to Appendix ~\ref{app_physical},
where we also explain how to iterate the map in the other direction of
time to get different conditions under which this conclusion
holds. (These conditions are satisfied whenever $d_s = 1$, they
apply for instance to any Lozi map with a compact invariant
domain $X_0$, see Corollary ~ \ref{goodforLozi}.)

The limit in \eqref{thebest} exists by submultiplicativity. We
can bound $\lambda_{s,n}$ and $\lambda_{u,n}^{-1}$ by
$\lambda^{-n}$, where $\lambda>1$ is the weakest rate of
contraction/expansion of $T$. Therefore, if $g= |\det DT|^{-1}$
then the essential spectral radius is strictly smaller than $1$
if there exist $s,t$ and $p$ as in Theorem ~ \ref{MainTheorem}
with
\begin{equation*}
  \lim_{n\to\infty}  (D_{n}^b)^{1/(pn)} \cdot (D_{n}^e)^{(1/n)(1-1/p)}
  \cdot \norm{|\det DT^{n}|^{1/p-1}}{L^\infty}^{1/n} < \lambda^{\min (t, -(t-|s|))}\, .
  \end{equation*}
In particular, if $g= |\det DT|^{-1}\equiv 1$, then the
essential spectral radius is strictly smaller than $1$ if
$\lim_{n\to\infty}  (D_{n}^b)^{1/(pn)} (D_{n}^e)^{(1/n)(1-1/p)}
< \lambda^{\min (t, -(t-|s|))} $, that is, if hyperbolicity
dominates complexity.

Subsections~\ref{sec:def_fol} and~\ref {spaces} are devoted to
the definition of spaces $\HHH^{t,s}_p(R, C_0, C_1)$ which will
give the space $\HHH$ of Theorem \ref{MainTheorem}  via
Proposition~ \ref{propinvar} (see (\ref{normeHnR})). Let us now
describe briefly this space $\HHH$, which generalizes the
spaces of \cite{baladi_Cinfty,baladi_gouezel_piecewise}.
Intuitively, an element of $\HHH$ is a distribution which has
$t$ derivatives in $L^p$ in all directions together with $s$
derivatives in $L^p$ in the stable direction. This amounts to
$s+t$ derivatives in $L^p$ in the stable direction, and $t$
derivatives in the transverse ``unstable'' direction. Since
$t>0$ and $t+s=t-|s|<0$, the transfer operator increases
regularity in this space. The restriction $1/p-1<s<0<t<1/p$ is
designed so that this space is stable under multiplication by
characteristic functions of nice sets (see \cite[Lemma~
23]{baladi_gouezel_piecewise}) --- this makes it possible to
deal with discontinuous maps. If one assumes that there exists
a $C^1$ stable direction, the above rough description can be
made precise, using anisotropic Sobolev spaces: This was done
in \cite{baladi_gouezel_piecewise}\footnote{The local spaces in
Definition~\ref{space} are the same as those in
\cite{baladi_gouezel_piecewise}.}. In our setting, there is in
general not even a continuous stable direction, so we shall
instead use a {\it class} of local foliations (with uniformly
bounded $C^{1+\beta}$ norms) compatible with the stable cones,
and define our norm as the supremum of the anisotropic Sobolev
norms over all local foliations in this class (Definition~
\ref{defnorm}). To ensure that the space so defined is
invariant under the action of the transfer operator, one should
make sure that the preimage under iterates of $T$ of a
foliation in our class remains in our class: This is the
content of our key Lemma ~\ref{lemcompose}. Since we want those
foliations to have bounded $C^1$ norm (otherwise, the argument
for anisotropic Sobolev norms  fails), we need the bunching
condition \eqref{bunch} to prove this invariance. (In the
smooth, i.e., Axiom A case, \eqref{bunch} would ensure that the
stable foliation is $C^1$ --- see e.g. \cite[\S 19.1]{katok} in
the case $\alpha=1$ --- and the strengthening \eqref{bunch2}
would even ensure that the stable foliation is $C^{1+\beta}$.
In the general piecewise smooth case, the foliation is only
measurable, even if \eqref{bunch} holds.)

\subsection{Anisotropic spaces \texorpdfstring{$H_p^{t,s}$}{Hpts} in
\texorpdfstring{$\real^d$}{Rd} and the class
\texorpdfstring{$\FF(z_0,\CC^s, C_0, C_1)$}{F(z0,Cs, C0, C1)}
of local foliations}
\label{sec:def_fol}

In this subsection, we  recall the anisotropic  spaces
$H^{t,s}_p$  in $\real^d$ (which were used in
\cite{baladi_gouezel_piecewise}), and we define a class $\FF$
of cone-admissible local foliations in $\real^d$  with
uniformly bounded $C^ {1+\beta}$ norms (in Lemma ~\ref{lemcompose} we
shall show that this class is invariant under iterations of the
dynamics). These are the two building blocks that we shall use
in Section ~ \ref{spaces} to define our spaces of
distributions.

We write $z\in \real^d$ as $z=(x,y)$ where
$x=(z_1,\dots,z_{d_u})$ and $y=(z_{d_u+1},\dots, z_d)$. The
subspaces $\{x \} \times \real^{d_s}$ of $\real^d$ will  be
referred to as the {\it stable leaves} in $\real^d$. We say
that a diffeomorphism of $\real^ d$ {\it preserves stable
leaves} if its derivative has this property. For $r>0$ and
$z=(x,y)\in \R^d$, let us write $B(x,r)=\{ x'\in \R^{d_u} \st
|x'-x|\leq r\}$, $B(y,r)=\{ y'\in \R^{d_s} \st |y'-y|\leq r\}$
and $B(z,r)=B(x,r)\times B(y,r)$. We denote the Fourier
transform in $\real^d$ by $\FFF$. An element of the dual space
of $\real^d$ will be written as $(\xi,\eta)$ with $\xi \in
\real^{d_u}$ and $\eta\in \real^{d_s}$.

The local anisotropic Sobolev  spaces $H_p^{t,s}$ belong to a
class of spaces first studied by Triebel \cite{triebel_III}:

\begin{definition}[Sobolev spaces $H^{t,s}_p$ and $H^{t}_p$ in $\real^d$]
\label{space} For $1<p<\infty$, and $t$, $s \in \real$,
let $H_p^{t,s}$ be the set of (tempered) distributions $w$ in
$\real^d$ such that
  \begin{equation}\label{norm}
  \norm{w}{H_p^{t,s}}\coloneqq
  \norm{\FFF^{-1}(a_{t,s}\FFF w)}{ L^p} < \infty\, ,
  \end{equation}
where
  \begin{equation}
  a_{t,s}(\xi,\eta)=(1+|\xi|^2+|\eta|^2)^{t/2} (1+|\eta|^2)^{s/2} .
  \end{equation}
For $1 < p < \infty$, $t \in \real$,  the set
$H_p^{t}=H_p^{t,0}$ is the standard (generalized) Sobolev
space.
\end{definition}

Triebel proved that rapidly decaying $C^\infty$ functions are
dense in each $H^{t,s}_p$ (see e.g. ~ \cite[Lemma
18]{baladi_gouezel_piecewise}). In particular, we could
equivalently define $H_p^{t,s}$ to be the closure of rapidly
decaying $C^\infty$ functions for the norm \eqref{norm}.

\medskip

We shall work with local foliations indexed by points $m$ in
appropriate finite subsets of $\real^d$ (defined in
\eqref{fset} below). The following definition of the class of
foliations is the key new ingredient of the present work. We
view $\alpha \in (0,1]$ and $\beta \in (0,\alpha]$ as fixed
(like in the statement of Theorem ~ \ref{MainTheorem}) while
the constants $C_0>1$ and $C_1>2C_0$ will be chosen later.
These constants play the following role: if $C_0$ is large,
then the admissible foliation covers a large domain; if $C_1$
is large, then the leaves of the foliation are almost parallel.

\begin{definition}[Sets $\FF(m,\CC^s, C_0, C_1)$ of cone-admissible foliations at $m\in \R^d$]
\label{deffol}
Let $\CC^s$ be a $d_s$-dimensional cone in $\R^d$, transverse
to $\R^{d_u}\times \{0\}$, let $m=(x_m,y_m)\in \R^d$, and let
$1<C_0<C_1/2$. The set $\FF(m,\CC^s,C_0,C_1)$ of
$\CC^s$-admissible local foliations at $m$ is the set of maps
  \begin{equation*}
  \phi=\phi_F: B(m,C_0)\to \R^d
  \, ,
  \quad \phi_F(x,y)=(F(x, y), y)
  \, ,
  \end{equation*}
where  $F:B(m,C_0) \to \R^{d_u}$ is $C^1$ and satisfies
  \begin{equation*}
  (\partial_y F(z)w,w)\in \CC^s\, ,  \forall w\in\R^{d_s}\, ,
  \forall z\in B(m, C_0)\, ;
  \quad
  F(x,y_m)= x\, ,  \, \forall x\in B(x_m, C_0)\, ,
  \end{equation*}
and, for all $(x,y)$ and $(x',y')$ in $B(m,C_0)$,
  \begin{align}
  \label{aalpha}
  | DF(x,y) -DF(x,y') | &\leq |y-y'|^\alpha/C_1\, ,\\
  |DF(x,y)- DF(x',y)|&\leq  |x-x'|^\beta /C_1\, ,
  \label{beeta}
  \end{align}
and
  \begin{equation}
  | D F(x,y)- DF (x,y')-DF(x',y)+DF(x',y') | \leq |x-x'|^\beta |y-y'|^{\alpha-\beta}/C_1  \, .
  \label{beeta'}
  \end{equation}
\end{definition}

The set $\FF(m,\CC^s,C_0,C_1)$ is large, as we explain next: If
the cone $\CC^s$ is $d_s$-dimensional and transverse to
$\real^{d_u}$, then it contains a $d_s$-dimensional vector
subspace $E$ which is transverse to $\real^{d_u}$. Therefore,
there exists a (possibly zero) linear map $\mathbb E :
\real^{d_s}\to \real^{d_u}$ so that $E=\{ (\mathbb E w, w) \, ,
w \in \real^{d_s} \}$. It follows that the affine map
$F_{\mathbb E}(x,y)=x+\mathbb E (y-y_m)$ is such that
$\phi_{F_{\mathbb E}}\in \FF(m,\CC^s, C_0, C_1)$. Then, it is
easy to see that if $F$ is $C^{1+\alpha}$, with
$F(x,y_m)=F_{\mathbb E}(x,y_m)=x$, and $F$ is close enough to
$F_{\mathbb E}$,  then $\phi_F \in \FF(m,\CC^s,C_0,C_1)$. (To
check \eqref{beeta'}, consider separately the cases $|x-x'|\le
|y-y'|$ and $|x-x'|>|y-y'|$.)

We now collect easy but important consequences of the above
definition. (See also the remarks at the end of this subsection
about the technical conditions \eqref{aalpha}--\eqref{beeta'}.)
We shall see in Lemma \ref{lempropphi} that the  graphs
$\{(F(x,y),y)\mid |y-y_m|<C_0\}$ for $|x-x_m|<C_0$ form a
partition of a neighborhood of $m$ of size proportional to
$C_0$ (through the $R$-zoomed charts to be introduced in
Section~\ref{spaces}, this will correspond to a neighborhood of
size of the order of $C_0/R$ in the manifold), and their
tangent space is everywhere contained in $\CC^s$. The map $F$
thus defines a local foliation (justifying the terminology),
and the map $\phi_F$ is a diffeomorphism straightening this
foliation, i.e., the leaves of the foliation are the images of
the stable leaves of $\R^d$ under the map $\phi_F$. (The maps
$y \mapsto (F(x,y),y)$ for fixed $x$ are sometimes called  {\it
plaques,} while $x \mapsto F(x,y)$ for $y$ fixed is the {\it
holonomy} between the transversals of respective heights $y_m$
and $y$.) Moreover, if $C_1$ is very large, then $DF$ is close
to constant, i.e., $\phi_F$ is very close to an affine map. The
conditions in the definition up to \eqref{aalpha} imply that
the local foliation defined by $F$ is $C^{1+\alpha}$ along the
leaves. Moreover, the next lemma shows that these conditions
imply uniform bounds on $F$ (independent of $C_0$).

\begin{lemma}[Admissible foliations are $C^{1+\beta}$ foliations]
\label{lempropphi}
For any $d_s$-dimensional cone $\CC^s$ transverse to
$\R^{d_u}\times \{0\}$, there exists a constant $\Cs$ depending
only on $\CC^s$ such that, for any $1<C_0<C_1/2$, and any
$\phi_F \in \FF(m,\CC^s, C_0, C_1)$, the map $\phi_F$ is a
diffeomorphism onto its image with
$\norm{D\phi_F}{C^{\beta}}\leq \Cs$ and
$\norm{D\phi_F^{-1}}{C^{\beta}} \leq \Cs$. Moreover, $\phi_F(
B(m,C_0))$ contains $B(m,\Cs^{-1} C_0)$.
\end{lemma}

The proof of these claims does not require \eqref{beeta'}.

\begin{proof} Let $\phi=\phi_F \in \FF(m,\CC^s, C_0, C_1)$.
We first check that $\norm{DF}{C^0}\le \Cs$. Observe first that
$\partial_y F$ is bounded since the cone $\CC^s$ is transverse
to $\R^{d_u}\times\{0\}$. Since $F(x,y_m)=x$, we have
$\partial_x F(x,y_m)=\id$, hence \eqref{aalpha} gives
  \begin{equation}
  \label{eqpartialxF}
  |\partial_x F(x,y)-\id|=|\partial_x F(x,y) -\partial_x F(x,y_m)|
  \leq |y-y_m|^\alpha/C_1
  \leq C_0^{\alpha}/C_1 < 1/2\, .
  \end{equation}
In particular, $|\partial_x F|$ is uniformly bounded. This
shows that $\norm{DF}{C^0}\le \Cs$. We next observe that
condition \eqref{beeta} together with \eqref{aalpha} imply that
$DF$ is $\beta$-H\"{o}lder: There exists a constant $\Cs$
(independent of $C_0$) such that, for all pairs $(x,y)$ and
$(x',y')$ in $B(m,C_0)$,
  \begin{equation}\label{willuse'}
  |DF(x,y)-DF(x',y')|\le \Cs d((x,y), (x',y'))^\beta \, .
  \end{equation}
Indeed, \eqref{aalpha} gives $|DF(x,y)-DF(x,y')|\leq
|y-y'|^\alpha$, and \eqref{beeta} gives
$|DF(x,y')-DF(x',y')|\leq |x-x'|^\beta$. Since $\beta \le
\alpha$, \eqref{willuse'} follows. We have shown that
$\norm{D\phi}{C^{\beta}}\leq \Cs$.

For any vector $v$, \eqref{eqpartialxF} shows that $\langle
\partial_x F v,v\rangle \geq |v|^2/2$. Integrating this
inequality on the segment between $x$ and $x'$, for $v=x'-x$,
we get $\langle F(x,y)-F(x',y), x-x'\rangle \geq |x-x'|^2/2$.
In particular,
  \begin{equation}
  |F(x,y)-F(x',y)|\geq |x-x'|/2\, .
  \end{equation}

By Lemma \ref{belongsDD}, this implies that the map $\phi$
belongs to the class $\DD(\Cs)$ defined in Subsection~
\ref{subsectDD}, for some $\Cs>0$ independent of $C_0, C_1$. In
particular, $\phi$ is a diffeomorphism onto its image, and
$|D\phi^{-1}|\leq \Cs$. Since $D\phi$ is $\beta$-H\"{o}lder, it
follows that $D\phi^{-1}$ is also $\beta$-H\"{o}lder, and
$\norm{D\phi^{-1}}{C^{\beta}}\leq \Cs$.

Finally, Lemma ~\ref{DDbigimage} shows that $\phi( B(m,C_0))$
contains $B(\phi(m),\Cs^{-1} C_0)$.
\end{proof}

We end this subsection with the promised remarks on the
conditions in Definition ~ \ref{deffol}  involving $\alpha$ and
$\beta$.

\begin{remark}[Condition \eqref{aalpha}]
Condition \eqref{aalpha} is  used in the proof of Lemma~
\ref{lempropphi} to ensure that $|DF|$ is uniformly bounded. It
would seem more natural to replace \eqref{aalpha} by the weaker
condition $|DF|\leq C$. However, it turns out that this weaker
condition is never invariant under the graph transform, while
\eqref{aalpha} is invariant if \eqref{bunch} is satisfied (see
\eqref{heart}). If $T$ is piecewise $C^2$ one can take
$\alpha=1$, and this is what is usually done in the literature
(\cite[\S 19]{katok}, \cite[App.~A]{liverani_contact}). In
addition, because of the extra $C^{1+\alpha}$ smoothness in the
$y$-direction given by \eqref{aalpha}, Lemma ~ \ref{lemcompose}
produces diffeomorphisms $\Psi$ and $\Psi_m$ which belong to
the space $D^1_{1+\alpha}$ from Definition \ref{3.1}. This is
useful in view of the composition Lemma ~
\ref{lemcomposeD1alpha}.
\end{remark}

\begin{remark}[Conditions \eqref{beeta} and \eqref{beeta'}: H\"{o}lder Jacobian]
Lemma ~ \ref{embed} about compact embeddings requires the
foliations $\phi_F$ and their inverses $\phi_F^{-1}$ to have
$C^\beta$ Jacobians for some $\beta >0$. (Beware that, even if
$T$ is volume-preserving, the class of foliations satisfying
$|\det D\phi| \equiv 1$ is not invariant under the dynamics,
because of the necessary reparametrizations in the proof of
Lemma ~\ref{lemcompose}.) Lemma ~ \ref{lempropphi} shows that
the conditions \eqref{aalpha} and \eqref{beeta} imply that the
Jacobians $J(x,y)= |\det D \phi_F|(x,y)= |\det
\partial_x F|(x,y)$ and $\tilde J(x,y)=|\det
D\phi_F^{-1}|(x,y)$ are $\beta$-H\"{o}lder (with a $C^\beta$ norm
bounded independently of $C_0$).

Condition \eqref{beeta} will only be used to ensure that $J$
and $\tilde J$ are  $C^\beta$. It turns out that the H\"{o}lder
condition on the Jacobians, by itself, is not preserved when
the foliation is iterated under hyperbolic maps, and neither is
the condition \eqref{beeta} alone. However, the pair
 \eqref{beeta}--\eqref{beeta'} is invariant if
\eqref{bunch2} is satisfied (see in particular Step~ 3 in the
proof of Lemma ~\ref{lemcompose}).
\end{remark}

%%%%%%%%%%%%%%%%%%%%%%%%%%%%%%%%%%%%%%%%%%%%%%%%%%%

\subsection{Extended cones, suitable charts and spaces of distributions}
\label{spaces}

In this subsection, we introduce appropriate cones
$\CC^s_{i,j}$ and coordinate patches $\kappa_{i,j}$ on the
manifold in order to glue together (via a partition of unity)
the local spaces $H^{t,s}_p$ and define a space
$\HHH_{p}^{t,s}(R)$ of distributions\footnote{This is a
modification of the space denoted $\tilde \HH^{t,s}_p$ in
\cite{baladi_gouezel_piecewise}.}  by using the  charts in
$\FF(m, \CC^s_{i,j}, C_0, C_1)$.

\begin{definition}\label{extc}
An {\it extended cone} $\CC$ is a set of four cones $(\CC^s,
\CC_0^s, \CC^u, \CC^u_0)$ such that $\CC^s$ and $\CC^u$ are
convexly transverse, $\CC^s_0$ contains $\{0\}\times \R^{d_s}$,
$\CC^u_0$ contains $\R^{d_u}\times \{0\}$ and $\CC^s_0-\{0\}$
is contained in the interior of $\CC^s$, $\CC^u_0-\{0\}$ is
contained in the interior of $\CC^u$. Given two extended cones
$\CC$ and $\tilde\CC$, we say that an invertible matrix
$M:\R^d\to \R^d$ {\it sends $\CC$ to $\tilde\CC$ compactly} if
$M \CC^u$ is contained in $\tilde{\CC}^u_0$, and
$M^{-1}\tilde\CC^s$ is contained in $\CC^s_0$.
\end{definition}

For all $i\in I$, we fix once and for all a finite number of
open sets $U_{i,j,0}$ of $X_0$, for $1\leq j\leq N_i$, covering
$\overline{O_i}$, and included in the fixed neighborhood
$\tilde O_i$ of $\overline{O_i}$ where the extension $T_i$ of
$T_{|O_i}$ is defined. Let also $\kappa_{i,j}:U_{i,j,0} \to
\R^d$, for $i\in I$ and $1\leq j\leq N_i$, be a finite family
of $C^\infty$ charts, and let $\CC_{i,j}$ be extended cones in
$\R^d$ such that, wherever $\kappa_{i',j'}\circ T_i \circ
\kappa_{i,j}^{-1}$ is defined, its differential sends
$\CC_{i,j}$ to $\CC_{i',j'}$ compactly. Such charts and cones
exist, as we explain now. Since the map is hyperbolic and the
image of the unstable cone is included in the unstable cone,
small enlargements of the unstable cones are sent strictly into
themselves by the map. Therefore, if one considers charts with
small enough supports, and locally constant cones
$\CC_{i,j}^s$, $\CC_{i,j}^u$ slightly larger than the cones
$D\kappa_{i,j}(q) \CC_i^{(s)}(q)$,
$D\kappa_{i,j}(q)\CC_i^{(u)}(q)$, and slightly smaller cones
$\CC^s_{i,j,0}, \CC^u_{i,j,0}$, they satisfy the previous
requirements. (Convex transversality in the extended cone
follows from our convex transversality assumption on
$\CC_i^{(s)}$ and $\CC_i^{(u)}$.) We also fix open sets
$U_{i,j,1}$ covering $X_0$ such that
$\overline{U_{i,j,1}}\subset U_{i,j,0}$, and we let
$V_{i,j,k}=\kappa_{i,j}(U_{i,j,k})$, $k=0, 1$.

The spaces of distributions will depend on a large parameter $R
\ge 1$ which will play the part of a ``zoom:'' If $R\geq 1$ and
$W$ is a subset of $\R^d$, denote by $W^R$ the set $\{R\cdot z
\st z\in W\}$. Let also $\kappa_{i,j}^R(q)=R \kappa_{i,j}(q)$,
so that $\kappa_{i,j}^R(U_{i,j,k})=V_{i,j,k}^R$. Let
  \begin{equation}
  \label{fset}
  \ZZ_{i,j}(R)
  =\{ m\in V_{i,j,0}^R \cap \Z^d\mid
  B(m,C_0)\cap V_{i,j,1}^R\ne \emptyset\}\, ,
  \end{equation}
and
  \begin{equation}
  \label{fset2}
  \ZZ(R)=\{ (i,j,m) \st i\in I, 1\leq j\leq N_i, m\in \ZZ_{i,j}(R)\}\, .
  \end{equation}

To $\bm=(i,j,m)\in \ZZ(R)$ is  associated the point
$q_\bm\coloneqq(\kappa_{i,j}^R)^{-1}(m)$ of $X$. These are the
points around which we shall construct local foliations, as
follows. Let us first introduce useful notations: We write
  \begin{equation*}
  O_\bm=O_i \, , \quad
  \kappa^R_\bm=\kappa^R_{i,j}
  \text{ and }
  \CC_\bm=\CC_{i,j}\quad \text{ for } \zeta =(i,j,m) \in \ZZ(R)
  \, .
  \end{equation*}
These are respectively the partition set, the chart and the
extended cone that we use around $q_\bm$. Let us fix some
constants $C_0>1$ and $C_1>2C_0$. If $R$ is large enough, say
$R\geq R_0(C_0, C_1)$, then, for any $\bm=(i,j,m)\in \ZZ(R)$
and any chart $\phi_\bm\in \FF(m,\CC^s_\bm, C_0,C_1)$, we have
$\phi_\bm(B(m,C_0)) \subset V_{i,j,0}^R$. For $\bm=(i,j,m) \in
\ZZ(R)$, we can therefore consider the set of charts ($R$,
$C_0$ and $C_1$ do not appear in the notation for the sake of
brevity)
\begin{equation}\label{ourcharts}
\FF(\bm)\coloneqq\{\bphi_\bm=(\kappa_\bm^R)^{-1} \circ \phi_\bm:B(m,C_0) \to X \, ,\,
\phi_\bm\in \FF(m,\CC^s_\bm, C_0,C_1)\}\, .
\end{equation}
The image under  a chart $\bphi_\bm \in \FF(\bm)$ of the stable
foliation in $\R^d$ is a local foliation around the point
$q_\bm$, whose tangent space is everywhere contained in
$(D\kappa_\bm^R)^{-1}(\CC^s_\bm)$. This set is almost contained
in the stable cone $\CC^{(s)}_i(q_\bm)$, by our choice of
charts $\kappa_{i,j}$ and extended cones $\CC_{i,j}$.

Let us fix once and for all a $C^\infty$ function\footnote{Such
a function exists since the balls of radius $d$ centered at
points in $\Z^d$ cover $\R^d$.} $\rho:\R^d\to [0,1]$  such that
  \begin{equation*}
  \rho(z)=0 \text{ if } |z| \ge d\qquad  \text{ and } \qquad \sum_{m\in \Z^d} \rho(z-m)=1\, .
  \end{equation*}
For $\bm=(i,j,m)\in \ZZ(R)$, let $\rho_m(z)=\rho(z-m)$, and
\[
\brho_{\bm}\coloneqq\brho_{\bm}(R)=\rho_m\circ \kappa_\bm^R : X \to [0,1]\, .
\]
Since $\rho_m$ is compactly supported in
$\kappa_{i,j}^R(U_{i,j,0})$ if $m\in \ZZ_{i,j}(R)$ (and $R$ is
large enough, depending on $d$), the above expression is
well-defined. This gives a partition of unity in the following
sense:
\[
\sum_{m\in \ZZ_{i,j}(R)} \brho_{i,j,m}(q)= 1 \,,  \forall q \in U_{i,j,1}\, ,
\quad \brho_{i,j,m}(q)=0 \, , \forall q \notin  U_{i,j,0}\, .
\]
Our choices ensure that the intersection multiplicity of this
partition of unity is bounded, uniformly in $R$, i.e., for any
point $q$, the number of functions such that
$\brho_{\bm}(q)\not=0$ is bounded independently of $R$.

\medskip

The space we shall consider depends in an essential way on the
parameters $p$, $t$, and $s$. It will also depend, in an
inessential way, on  the choices we have made  (i.e.,  the
reference charts $\kappa_{i,j}$, the extended cones
$\CC_{i,j}$, the constants $C_0$ and $C_1$, the function
$\rho$, and $R\geq R_0(C_0, C_1)$): Different choices would
lead to different spaces, but all such spaces share the same
features. We emphasize the dependence on $R$, $C_0$ and $C_1$
in the notations, since all the other choices will be fixed
once and for all.

\begin{definition}[Spaces  $\HHH_{p}^{t,s}(R, C_0, C_1)$ of distributions on $X$]
\label{defnorm}
Let $1<p < \infty$, $s,t\in \R$, let $1<C_0<C_1/2$ and let
$R\geq R_0(C_0, C_1)$. For any system of charts $\Phi=\{
\bphi_{\bm}\in \FF(\bm) \st \bm\in \ZZ(R)\} $, let for
$\omega\in L^\infty(X_0)$
  \begin{equation}
  \label{defnormen}
  \norm{\omega}{\Phi}=\left(\sum_{\bm\in \ZZ(R)}
  \norm
  {(\brho_\bm(R) \cdot 1_{O_\bm} \omega)\circ \bphi_\bm}
  {H_p^{t,s}}^p\right)^{1/p}\, ,
  \end{equation}
and put $\norm{\omega}{\HHH_{p}^{t,s}(R,C_0,C_1)}
=\sup_{\Phi}\norm{\omega}{\Phi}$, the supremum ranging over all
such systems of charts $\Phi$.

The space $\HHH_{p}^{t,s}(R,C_0,C_1)$ is the closure of $\{
\omega \in L^\infty(X_0) \mid
\norm{\omega}{\HHH_{p}^{t,s}(R,C_0,C_1)}<\infty\}$ for the norm
$\norm{\omega}{\HHH_{p}^{t,s}(R,C_0,C_1)}$.
\end{definition}

For fixed $R$, the sum in \eqref{defnormen} involves a
uniformly bounded number of terms. Since the charts $\bphi_\bm$
have a uniformly bounded $C^1$ norm, the functions
$(\brho_\bm(R) \cdot \omega)\circ \bphi_\bm$ are uniformly
bounded in $C^1$ if $\omega$ is $C^1$. Moreover, $H_p^{t,s}$
contains the space of compactly supported $C^1$ functions on
$\real^d$ when $|t|+|s|\leq 1$. Therefore, if there were no
multiplication by $1_{O_\bm}$ in \eqref{defnormen}, then
$\norm{\omega}{\HHH_{p}^{t,s}(R,C_0,C_1)}$ would be finite for
any $C^1$ function $\omega$. When $s, t\in (1/p-1, 1/p)$,
multiplication by $1_{O_\bm} \circ \bphi_\bm$ leaves the space
$H_p^{t,s}$ invariant (see Lemma \ref{lem:multiplier} below).
Therefore, all $C^1$ functions belong to $\HHH_p^{t,s}(R, C_0,
C_1)$ in this case.

\begin{remark}\label{natural2}
A priori, the space $\HHH_{p}^{t,s}(R,C_0,C_1)$ is not
isomorphic to a Triebel space $H^{t,s}_p(X_0)$. However, our
assumptions ensure that $\HHH_{p}^{t,0}(R,C_0,C_1)$ is
isomorphic to the Sobolev-Triebel space $H^{t,0}_p(X_0)$
(whatever the value of $R$, $C_0$, $C_1$) when $-\beta <
t<1+\beta$. See Lemma~\ref{embed} for various embedding claims
on the spaces $\HHH_{p}^{t,s}(R,C_0,C_1)$.
\end{remark}

\subsection{Reduction of the main result}

\label{reducr}

In this subsection, we shall deduce Theorem~\ref{MainTheorem}
from  the following result about the spaces introduced in
Subsection~\ref{spaces}.

To simplify the statements, we will use the following
convention throughout this article: the sentence ``for all
large enough $x, y, z, \dots$'' means that, if $x$ is large
enough, then, if $y$ is large enough (possibly depending on
$x$), then if $z$ is large enough (possibly depending on $x$
and $y$), $\dots$.

\begin{theorem}
\label{MainTheorembis}
Let $T$, $g$, and  $p$, $t$, $s$ satisfy the assumptions of
Theorem~\ref{MainTheorem}. There exist $C_0>1$ and $\Cs>0$ such
that, for any $N>0$, any large enough $C_1>2C_0$, any large
enough integer $n$ which is a multiple of $N$, and any large
enough $R$, the operator $\LL_g^n$ is bounded on
$\HHH_p^{t,s}(R, C_0, C_1)$, and its essential spectral radius
is at most
  \begin{equation}\label{thebest2}
  (\Cs N)^{n/N} (D_{n}^b)^{1/p} \cdot (D_{n}^e)^{1-1/p}
  \cdot \norm{g^{(n)}|\det DT^{n}|^{1/p}
  \max(\lambda_{u,n}^{-t},\lambda_{s,n}^{-(t-|s|)})}{L^\infty} \, .
  \end{equation}
\end{theorem}

The above theorem will be proved in Section~\ref{mainsec}.
Below, we deduce Theorem~\ref{MainTheorem} from Theorem~
\ref{MainTheorembis}, using the following proposition (which
will be proved at the end of Section ~ \ref{mainsec}).

\begin{proposition}
\label{propinvar}
Let $T$, $g$, and  $p$, $t$, $s$ satisfy the assumptions of
Theorem~\ref{MainTheorem}, and let $C_0$ be given by Theorem
\ref{MainTheorembis}. For any large enough $C_1>0$ and $R>0$,
and any large enough $C'_1>0$ and $R'>0$, then for any large
enough $N$, $\LL_g^N$ is continuous from $\HHH_p^{t,s}(R, C_0,
C_1)$ to $\HHH_p^{t,s}(R', C_0, C'_1)$.
\end{proposition}

\begin{proof}[Proof that Theorem ~ \ref{MainTheorembis}
implies Theorem~\ref{MainTheorem}]

Theorem \ref{MainTheorembis} does not claim that the space
$\HHH_p^{t,s}(R, C_0, C_1)$ is invariant under $\LL_g$. This
issue is easy to deal with: Consider $C_1$, $n$ and $R$ such
that Theorem ~\ref{MainTheorembis} applies to $\LL_g^n$ acting
on $\HHH_p^{t,s}(R, C_0, C_1)$, and let $H(n,R, C_0,
C_1)=H(p,t,s,n,R, C_0, C_1)$ be the closure of $L^\infty(X_0)$
for the norm
  \begin{equation}
  \label{normeHnR}
  \norm{\omega}{H(p,t,s, n,R, C_0, C_1)}=\sum_{j=0}^{n-1}\norm{\LL_g^j \omega}{\HHH_p^{t,s}(R, C_0, C_1)}\, .
  \end{equation}
Since $\norm{\LL_g^n \omega}{\HHH_p^{t,s}(R, C_0, C_1)}\leq C
\norm{\omega}{\HHH_p^{t,s}(R, C_0, C_1)}$ by Theorem
\ref{MainTheorembis}, it follows that the operator $\LL_g$ is
continuous on $H(n,R, C_0, C_1)$.

Moreover, for any $C^1$ function $\omega$ and any $j$, the
function $\LL_g^j \omega=\sum_{\ii} 1_{T_\ii O_\ii} (g^{(j)}
\omega) \circ T_\ii^{-j}$ is a sum of $C^\gamma$ functions
multiplied by characteristic functions of nice sets.  The
discussion following Definition \ref{defnorm} (with $C^1$
replaced by $C^\gamma$) implies that $\LL_g^j \omega$ belongs
to $\HHH_p^{t,s}(R, C_0, C_1)$. Hence, $H(n,R, C_0, C_1)$
contains $C^1$ (in particular, it is not reduced to $\{0\}$).

To finish, we shall prove that the claim on the essential
spectral radius of $\LL_g$ holds on $\HHH=H(n,R, C_0, C_1)$, if
$C_1$, $n$ and $R$ are large enough. If $\MM$ is an operator
acting on a Banach space $E$, we denote by $\ress(\MM,E)$ its
essential spectral radius.

\emph{First claim: $\ress(\LL_g, H(n,R, C_0, C_1))\leq
\ress(\LL_g^n, \HHH_p^{t,s}(R, C_0, C_1))^{1/n}$.}

Let us admit this claim for the moment. Then, by
\eqref{thebest2}, the essential spectral radius of $\LL_g$ on
$H(n,R, C_0, C_1)$ is at most
\begin{equation*}
  (\Cs N)^{1/N}(D_{n}^b)^{1/(pn)} \cdot (D_{n}^e)^{(1/n)(1-1/p)}
  \cdot \norm{g^{(n)}|\det DT^{n}|^{1/p}
  \max(\lambda_{u,n}^{-t},\lambda_{s,n}^{-(t-|s|)})}{L^\infty}^{1/n} \, .
\end{equation*}
Since $(\Cs N)^{1/N}$ tends to $1$ when $N\to\infty$, this
factor is not troublesome. However, we do not have
Theorem~\ref{MainTheorem} yet: In \eqref{thebest}, there is a
limit in $n$, while our last bound is for a fixed $n$. This is
why we need to show the following statement:

\emph{Second claim: Let $r$ be the limit in \eqref{thebest}. If
$C_1$, $n$ and $R$ are large enough, we have
$\ress(\LL_g^{n},\HHH_p^{t,s}(R, C_0, C_1))\le r^n$.}

Putting together the first and second claims we deduce that the
space $\HHH=H(n,R, C_0, C_1)$ satisfies the conclusion of
Theorem~\ref{MainTheorem} if $C_1$, $n$ and $R$ are large
enough.

\smallskip
It remains to prove the two above claims. For this, we recall a
characterization of the essential spectral radius of an
operator $\MM$ acting on a Banach space $E$.
\begin{enumerate}
\item Let $\tau>0$, assume that there exist a sequence $j(n)\to
\infty$ and a sequence of compact operators $K_n: E\to E_n$
(for some Banach spaces $E_n$) such that
$\norm{\MM^{j(n)}w}{E}\leq \tau^{j(n)} \norm{w}{E} +
\norm{K_n w}{E_n}$ for any $w\in E$ (or, equivalently, in a
dense subset of $E$) and any large enough $n$. Then
$\ress(\MM,E)\leq \tau$.
\item Conversely, if $\tau>\ress(\MM,E)$, there exists a sequence
of compact operators $K_n:E\to E$ such that, if $n$ is
large enough, $\norm{\MM^n w}{E}\leq \tau^n
\norm{w}{E}+\norm{K_n w}{E}$ for any $w\in E$.
\end{enumerate}
The first assertion was proved by Hennion \cite{hennion} using
a formula of Nussbaum. The second assertion follows from the
spectral decomposition $\MM=K+A$ where $KA=AK=0$, $K$ has
finite rank (and corresponds to the eigenvalues of $\MM$ of
modulus $\geq \tau$), and the spectral radius of $A$ is smaller
than $\tau$ (just take $K_n=K^n$).

We prove the first claim. Let $\tau>\ress(\LL_g^n,
\HHH_p^{t,s}(R, C_0, C_1))^{1/n}$. By Item 2, there exists a
sequence of compact operators $K_{kn}:\HHH_p^{t,s}(R, C_0, C_1)
\to \HHH_p^{t,s}(R, C_0, C_1)$ such that, for large enough $k$,
  \begin{equation*}
  \norm{\LL_g^{kn} \omega}{\HHH_p^{t,s}(R, C_0, C_1)}\leq \tau^{kn}
  \norm{\omega}{\HHH_p^{t,s}(R, C_0, C_1)}+\norm{K_{kn} \omega}{\HHH_p^{t,s}(R, C_0, C_1)}.
	\end{equation*}
Therefore, for $\omega\in H(n,R, C_0, C_1)$,
  \begin{align*}
  \norm{\LL_g^{kn} \omega}{H(n,R, C_0, C_1)}
  &=\sum_{j=0}^{n-1} \norm{\LL_g^{kn} \LL_g ^j \omega}{\HHH_p^{t,s}(R, C_0, C_1)}
  \\&
  \leq \sum_{j=0}^{n-1} \tau^{kn}\norm{\LL_g ^j \omega}{\HHH_p^{t,s}(R, C_0, C_1)}
  + \norm{K_{kn}\LL_g^j\omega}{\HHH_p^{t,s}(R, C_0, C_1)}
  \\&=
  \tau^{kn}\norm{\omega}{H(n,R, C_0, C_1)}+\norm{\tilde K_{kn}\omega}{\HHH_p^{t,s}(R, C_0, C_1)^n},
  \end{align*}
where the operator $\tilde K_{kn}$ from $H(n,R, C_0, C_1)$ to
$\HHH_p^{t,s}(R, C_0, C_1)^n$ is given by
  \begin{equation*}
  \tilde K_{kn}\omega=( K_{kn}\omega, K_{kn} \LL_g\omega,\dots,
  K_{kn}\LL_g^{n-1}\omega) \, .
  \end{equation*}
Since this operator is compact, Item 1 above gives that
$\ress(\LL_g, H(n,R, C_0, C_1))\leq \tau$, and thus the first
claim.

\smallskip
Finally, we prove the second claim. The idea is to use
Proposition~\ref{propinvar} to go from $\HHH_p^{t,s}(R, C_0,
C_1)$ to $\HHH_p^{t,s}(R', C_0, C'_1)$ for large $C'_1$ and
$R'$, use the good control on the essential spectral radius on
$\HHH_p^{t,s}(R', C_0, C'_1)$, and then return to
$\HHH_p^{t,s}(R, C_0, C_1)$. Let $C_1$, $n$ and $R$ be as in
the statement of the second claim. Consider $\tau>r$, and let
us fix $C'_1>2C_0$, $k$ and $R'$ large enough so that
$\ress(\LL_g^{kn}, \HHH_p^{t,s}(R', C_0, C'_1)) <\tau^{kn}$:
This is possible by Theorem~\ref{MainTheorembis}. Therefore, by
Item 2, for large $j$, there exists a compact operator
$K_{jkn}: \HHH_p^{t,s}(R', C_0, C'_1) \to \HHH_p^{t,s}(R', C_0,
C'_1)$ such that $\norm{\LL_g^{jkn} \omega}{\HHH_p^{t,s}(R',
C_0, C'_1)}\leq \tau^{jkn} \norm{\omega}{\HHH_p^{t,s}(R', C_0,
C'_1)}+\norm{K_{jkn} \omega}{\HHH_p^{t,s}(R', C_0, C'_1)}$. By
Proposition~\ref{propinvar}, we can choose $m$ such that the
operator $\LL_g^{mn}$ sends $\HHH_p^{t,s}(R, C_0, C_1)$ to
$\HHH_p^{t,s}(R', C_0, C'_1)$ and $\HHH_p^{t,s}(R', C_0, C'_1)$
to $\HHH_p^{t,s}(R, C_0, C_1)$ continuously, with a norm
bounded by a constant that we denote by $C$. Then, for any
$\omega \in \HHH_p^{t,s}(R, C_0, C_1)$,
  \begin{align*}
  \norm{ \LL_g^{(jk+2m)n}\omega}{\HHH_p^{t,s}(R, C_0, C_1)}
  \leq C \norm{ \LL_g^{(jk+m)n}\omega}{\HHH_p^{t,s}(R', C_0, C'_1)}
  \hspace{-6cm}&
  \\&
  \leq C \tau^{jkn}\norm{\LL_g^{mn}\omega}{\HHH_p^{t,s}(R',C_0,C'_1)}+\norm{K_{jkn}
  \LL_g^{mn}\omega}{\HHH_p^{t,s}(R', C_0, C'_1)}
  \\&
  \leq C^2 \tau^{jkn} \norm{\omega}{\HHH_p^{t,s}(R,C_0,C_1)}
  +\norm{K_{jkn}  \LL_g^{mn}\omega}{\HHH_p^{t,s}(R',C_0,C'_1)} \, .
  \end{align*}
The operator $\tilde K_{jkn}\coloneqq K_{jkn}\LL_g^{mn}$ is
compact from $\HHH_p^{t,s}(R,C_0,C_1)$ to
$\HHH_p^{t,s}(R',C_0,C'_1)$. Therefore, Item 1 ensures that
  \begin{equation*}
  \ress(\LL_g^n,\HHH_p^{t,s}(R,C_0,C_1))
  \leq \liminf_{j\to\infty} (C^2 \tau^{jkn})^{1/(jk+2m)}
  =\tau^n \, .
  \end{equation*}
This ends the proof of the second claim and of the theorem.
\end{proof}

%%%%%%%%%%%%%%%%%%%%%%%%%%%%%%%%%
\section{Invariance of the class of cone admissible local foliations}
\label{invv}

In order to prove the bounds necessary for
Theorem~\ref{MainTheorembis}, we need to check that the class
of admissible foliations defined in Subsection
\ref{sec:def_fol} is invariant under the iteration of the map
$T^{-1}$ (viewed in charts). This is the purpose of the key
Lemma~\ref{lemcompose} below, which says that if $\phi_m\in
\FF(m,  \CC^s, C_0, C_1)$ is an admissible foliation, then the
chart $\phi'$ obtained by pulling it back by a diffeomorphism
$\TT^{-1}$ of $\real^d$, and  reparameterizing to put it in
standard form is still admissible {\it if the map $\TT$ is
sufficiently hyperbolic, $C^{1+\alpha}$, and satisfies a
bunching condition} (see \eqref{suffhyp}). This fact is not
surprising: It is well known (see e.g. the Hadamard-Perron
arguments in \cite[\S 6.2, \S 19]{katok}) that $C^1$ foliations
remain $C^1$ after a graph transform if the transformation
satisfies a bunching condition. However, the statement of
Lemma~\ref{lemcompose} is a little involved because (in order
to avoid exponential proliferation of the number of charts) we
need to ``glue together'' all pulled back charts $\phi_m$
associated to a set $\JJ$ of ``well-separated'' points $m$.
This must be done carefully, controlling  the size of the
domains of definition of the new chart $\phi'$ thus produced.
%Note also that our arguments
%cannot be purely local (relying, say, on the implicit function
%theorem).

If the pullback of a foliation $\phi(x,y)=(F(x,y),y)$ under a
map $\TT$ is given in standard form by a map
$\phi'(x,y)=(F'(x,y),y)$, this means that $\TT^{-1}\circ
\phi=\phi'\circ \TTT$ for some map $\TTT$ defined on a subset
of $\R^d$, and sending stable leaves to stable leaves. This map
$\TTT$ is needed to straighten $\TT^{-1}\circ \phi$, which
typically does not have the form $(x,y)\mapsto (F'(x,y),y)$.
The map $\TTT$ corresponds to $\TT^{-1}$ in the charts
$\phi,\phi'$, and it will be important to control well its
smoothness and hyperbolicity. In particular, the following
definition will be useful.

\begin{definition}\label{3.1} For $C>0$ let $D^1_{1+\alpha}(C)$ denote  the set of $C^1$
diffeomorphisms $\Psi$ defined on a subset of $\R^d$, sending
stable leaves to stable leaves, and such that
  \begin{equation*}
  \max\bigl ( \sup |D\Psi(x,y)|, \sup |D\Psi^{-1}(x,y)|,
  \sup_{x,y,y'} \frac{|D\Psi(x,y)-D\Psi(x,y')|}{|y-y'|^\alpha}\bigr )\leq C \, .
  \end{equation*}
\end{definition}

Before we state Lemma~\ref{lemcompose}, we need one more
notation:

\begin{definition}
Let $\CC$ and $\tilde \CC$ be extended cones (Definition ~
\ref{extc}). If an invertible matrix $M:\R^d\to \R^d$ sends
$\CC$ to $\tilde\CC$ compactly, let
$\lambda_u(M)=\lambda_u(M,\CC,\tilde\CC)$ be the least
expansion under $M$ of vectors in $\CC^u$, and
$\lambda_s(M)=\lambda_s(M,\CC,\tilde\CC)$ be the inverse of the
least expansion under $M^{-1}$ of vectors in $\tilde\CC^s$.
Denote by $\Lambda_u(M)=\Lambda_u(M,\CC,\tilde\CC)$ and $
\Lambda_s(M)=\Lambda_s(M,\CC,\tilde\CC)$ the strongest
expansion and contraction coefficients of $M$ on the same
cones.
\end{definition}

The key lemma can now be stated:

\begin{lemma}
\label{lemcompose}
Let $\CC$ and $\tilde\CC$ be extended cones, let $\alpha\in
(0,1]$ and let $\beta \in (0,\alpha)$. For any large enough
$C_0$ (depending on $\CC$ and $\tilde\CC$) and any $C_1 >
2C_0$, there exist constants $C$ (depending on $\CC$,
$\tilde\CC$ and $C_0$) and $\epsilon$ (depending on $\CC$,
$\tilde \CC$, $C_0$ and $C_1$) satisfying the following
properties:

Let $\TT$ be a $C^{1+\alpha}$ diffeomorphism of $\R^d$ with
$\TT(0)=0$ and, setting $M\coloneqq D\TT(0)$, so that
  \begin{align}
  \nonumber &
  \norm{ \TT^{-1}  \circ M -\id }{C^{1+\alpha}} \leq \epsilon\, , \quad
  M \text{  sends } \CC \text{ to } \tilde\CC \text{ compactly, } \\
   \label{suffhyp}
  &
  \lambda_{s}(M)^{\alpha-\beta} \Lambda_{u}(M)^{1+\beta} \lambda_{u}(M) ^{-1}  <\epsilon \, ,
  \qquad  \lambda_u(M)>\epsilon^{-1}\, ,\,\,  \lambda_s(M)^{-1}>\epsilon^{-1} \, .
\end{align}
Let $\JJ \subset \R^d$ be a finite set such that $|m-m'|\geq C$
for all $m\neq m'\in \JJ$, and consider any family of charts
$\{\phi_m\in \FF(m,\tilde\CC^s, C_0, C_1)\mid m\in \JJ\}$.

Then, defining
  \begin{equation*}
  \JJ'\coloneqq \{ m \in \JJ \mid  B(m,d) \cap \TT(B(0,d))\ne \emptyset \} \, ,
  \end{equation*}
and setting  $\Pi(x,y)=(x,0)$, we have:

(a) $|\Pi m -\Pi m'|\geq C_0$ for all $m\not=m'$ in $\JJ'$.

(b) There exist $\phi'\in \FF(0,\CC^s, C_0, C_1)$,  and
diffeomorphisms $\TTT_m$, for $m\in \JJ'$, such that
  \begin{equation}\label{newphi}
  \TT^{-1}\circ \phi_m=\phi'\circ \TTT_m
  \quad \text{ on }
  \phi_m^{-1}( B(m,d) \cap \TT(B(0,d))) \, ,\, \, \forall m \in \JJ'\, .
  \end{equation}

(c) For each $m \in \JJ'$, we can write $\TTT_m=\Psi \circ
D^{-1}\circ \Psi_m$, where
\begin{itemize}
\item The diffeomorphism $\Psi_m$ is in
$D^1_{1+\alpha}(C)$, its  range contains  $B(\Pi
m,C_0^{1/2})$, and $\Psi_m( \phi_m^{-1}(B(m,d))) \subset
B(\Pi m, C_0^{1/2}/2)$.
\item The matrix $D$ is block diagonal, of the form $D=\left(\begin{smallmatrix} A&0\\0&B
\end{smallmatrix}\right)$ with
  \begin{equation*}
  |Av|\geq C^{-1}\lambda_u(M) |v| \text{ and } |Bv|\leq C\lambda_s(M) |v|\, .
  \end{equation*}
\item The diffeomorphism $\Psi$ is in $D^1_{1+\alpha}(C)$,
its range contains $B(0, C_0^{1/2})$.
\end{itemize}
\end{lemma}

Note that  (c) implies in particular that each $\TTT_m$ sends
stable leaves to stable leaves. Note also that if $C_0$ is
large enough, then $\phi' \in \FF(0, \CC^s, C_0, C_1)$ implies
$(\phi')^{-1}(B(0,d))\subset B(0, C_0^{1/2}/2)$ (because
$\|(\phi')^{-1}\|_{C^1} \le \Cs$ by Lemma ~ \ref{lempropphi}).

Statements (b) and (c) are the main result of the lemma: (b)
shows that the pullback of all the relevant charts $\phi_m$ can
be glued together to form an admissible chart $\phi'$, while
(c) gives an expression of $\TTT_m$, that is, $\TT^{-1}$ in the
charts $\phi_m$, $\phi'$, as the composition of two well
controlled diffeomorphisms $\Psi$, $\Psi_m$, and a matrix $D$
with good hyperbolic properties. Statement (a), although an
essential consequence of hyperbolicity, has a more technical
nature: It is used in Step 2 of the proof of the lemma (when
gluing foliations), and also later in the proof of Theorem ~
\ref{MainTheorembis}. At the first reading, the reader can
ignore the  information on the ranges of $\Psi$ and $\Psi_m$
(but beware that they will be important in the proof of Theorem
~ \ref{MainTheorembis}).

\begin{remark}\label{shorter}
Composing with translations, we deduce a more general result
from Lemma ~\ref{lemcompose}, replacing $0$ by $\ell \in
\real^d$, and allowing $\TT(\ell)\ne \ell$: Just replace $M$ by
$D\TT(\ell)$, the projection $\Pi$ by
$\Pi(x,y)=(x,y_{\TT(\ell)})$, where $\TT(\ell)=(x_{\TT(\ell)},
y_{\TT(\ell)})$, and assume that
  \begin{equation*}
  \norm{ (\TT^{-1} [\cdot + \TT(\ell)]-\ell) \circ D\TT(\ell)
  -\id }{C^{1+\alpha}} \leq \epsilon
  \end{equation*}
and that $D\TT(\ell)$ sends  $\CC $ to $\tilde\CC$ compactly.
One then uses the condition $B(m,d) \cap \TT(B(\ell,d))\ne
\emptyset$ to define $\JJ'$. Of course, $\phi'$ is then in
$\FF(\ell ,\CC^s, C_0, C_1)$, equality \eqref{newphi} holds on
$\phi_m^{-1}( B(m,d) \cap \TT(B(\ell,d)))$, and the range of
$\Psi$  contains $B(\ell, C_0^{1/2})$. Finally, we have
$(\phi')^{-1}(B(\ell,d))\subset B(\ell, C_0^{1/2}/2)$.
\end{remark}

\medskip

\begin{proof}
[Proof of Lemma~\ref{lemcompose}]
We shall write $\pi_1$ and $\pi_2$ for, respectively, the first
and the second projection in $\R^d=\R^{d_u}\times \R^{d_s}$.

\emph{Step zero: Preparations.} We shall write $\Cs$ and
$\epsilons$ for a large, respectively small, constant,
depending only on $\CC,\tilde\CC$, that may vary from line to
line. For the other parameters, we will always specify if they
depend on $C_0$ or $C_1$.

The set $M(\R^{d_u}\times \{0\})$ is contained in
$\tilde\CC^u$, hence uniformly transverse to $\{0\}\times
\R^{d_s}$. Therefore, it can be written as a graph $\{(x,Px)\}$
for some matrix $P$ with norm depending only on $\tilde\CC$.
Let $Q(x,y)=(x,y-Px)$, so that $QM$ sends $\R^{d_u}\times
\{0\}$ to itself. In the same way, $M^{-1}(\{0\}\times
\R^{d_s})$ is contained in $\CC^s$, hence it is a graph
$\{(P'y,y)\}$. Letting $Q'(x,y)=(x-P'y,y)$, the matrix
$D=QM(Q')^{-1}$ leaves $\R^{d_u}\times \{0\}$ and $\{0\}\times
\R^{d_s}$ invariant, i.e., it is block-diagonal, of the form $
\left(\begin{smallmatrix} A&0\\0&B
\end{smallmatrix}\right)$, and moreover
$|Av|\geq \Cs^{-1}\lambda_u |v|$ and $|Bv|\leq \Cs\lambda_s
|v|$ (since the matrices $Q$ and $Q'$, as well as their
inverses,  are uniformly bounded in terms of $\CC$ and
$\tilde\CC$).

We can readily prove  assertion (a) of the lemma. Let $m\in
\JJ'$, there exists $z \in B(m,d) \cap \TT(B(0,d))$. The set
$Q\TT(B(0,d))=DQ'(\TT^{-1}M)^{-1}(B(0,d))$ is included in
$\{(x,y)\st |y|\leq \Cs\}$ for some constant $\Cs$ (the role of
$Q$ is important here). Since $Qz\in Q\TT(B(0,d))$, we obtain
$|\pi_2(Q z)|\leq \Cs$. Since $|z-m|\leq d$, we also have
$|Qz-Qm|\leq \Cs$, hence $|\pi_2(Qm)|\leq \Cs$ (for a different
constant $\Cs$). Since $Qm-\Pi m=(x_m, \pi_2(Qm)) -
(x_m,0)=(0,\pi_2(Qm))$, we obtain
  \begin{equation}
  \label{piQm}
  |Qm-\Pi m|\leq \Cs\, .
  \end{equation}
Since the points $m\in \JJ'$ are far apart by assumption, the
points $Qm$ for $m\in \JJ'$ are also far apart, and it follows
that the points $\Pi m$ are also far apart. Increasing the
distance between points in $\JJ'$, we can in particular ensure
that $|\Pi m-\Pi m'|\geq C_0$ for any $m\not=m'\in \JJ'$,
proving (a).

The strategy of the proof of the rest of the lemma is the
following: We write
  \begin{equation}
  \TT^{-1}=\TT^{-1}M \cdot (Q')^{-1}\cdot D^{-1} \cdot Q \, .
  \end{equation}
We shall start from the partial foliation given by the maps
$\phi_m$ for $m\in \JJ$, apply $Q$ (Step ~1) to obtain a new
partial foliation at $Q m$, modify it via gluing (Step ~2) to
obtain a global foliation, and then push this foliation
successively with $D^{-1}$ (Step~ 3), $(Q')^{-1}$ (Step ~4), and
$\TT^{-1}M$ (last step).

We shall use in this proof the spaces of local diffeomorphisms
$\DD(\Cs)$ and of matrix-valued functions
$\KK(\Cs)=\KK^{\alpha,\beta}(\Cs)$ introduced in Appendix
\ref{app}. As in Remark~\ref{domdef} of this appendix, we will
write $\KK(\Cs,A)$ for the functions defined on a set $A$ and
satisfying the inequalities defining $\KK(\Cs)$ ($A$ will
sometimes be omitted when the domain of definition is obvious).
The map $\phi_m$ belongs to $\DD(\Cs)$ (see the proof of Lemma
~ \ref{lempropphi}), and the matrix-valued function $D\phi_m$
belongs to $\KK(\Cs, B(m, C_0))$ (boundedness of $D\phi_m$ is
proved in Lemma ~ \ref{lempropphi}, while the H\"{o}lder-like
properties are given by \eqref{aalpha}--\eqref{beeta'}).

\medskip

\emph{First step: Pushing the foliations with $Q$.} We
formulate in detail  the construction in this first step (a
version of Lemma~ \ref{lemcomposeQ} will be used also in the
last step, replacing $Q$ by $\TT^{-1} M$, while steps 2-3-4 are
much simpler).

\begin{lemma}
\label{lemcomposeQ} (Notation as in Lemma ~\ref{lemcompose} and Step 0 of its proof.)
There exists a constant $\Cs$ such that, if $C_0$ is large
enough and $C_1>2C_0$, for any $m=(x_m, y_m)\in \JJ'$ there
exist two maps $\phi_m^{(1)}: B(\Pi m, C_0^{1/2}) \to \R^d$ and
$\Psi_m:B(m,C_0^{2/3})\to \R^d$ such that
  \begin{equation*}
  \phi_m^{(1)}\circ \Psi_m=Q\circ \phi_m \text{ on } \phi_m^{-1}(B(m,d)) \, .
  \end{equation*}
Moreover, $\Psi_m$ is a diffeomorphism in $D^1_{1+\alpha}(\Cs)$
whose  range contains $B(\Pi m, C_0^{1/2})$, and
$\Psi_m(\phi_m^{-1}(B(m,d))) \subset B(\Pi m, C_0^{1/2}/2)$.
Finally, $\phi_m^{(1)}(x,y)=(F^{(1)}_{m}(x,y),y)$ on $B(\Pi m,
C_0^{1/2})$, with $F^{(1)}_m$ a $C^1$ map so that
$F^{(1)}_{m}(x,0)=x$ and $DF^{(1)}_m$ belongs to $\KK(\Cs,B(\Pi
m, C_0^{1/2}) )$.
\end{lemma}

Note that if $\EE$ is the foliation given by
$\phi_m(x,y)=(F_m(x,y),y)$, then by definition $\phi_m^{(1)}$
sends the stable leaves of $\R^d$ to the foliation $Q(\EE)$,
i.e., $\phi_m^{(1)}$ is the standard parametrization of the
foliation $Q(\EE)$.

\begin{proof}[Proof of Lemma ~\ref{lemcomposeQ}]
Fix $m=(x_m,y_m) \in \JJ'$. The map $Q\circ \phi_m$ does not
qualify as $\phi_m^{(1)}$ for two reasons. First, $\pi_2 \circ
Q \circ \phi_m(x,y)$ is generally not equal to $y$. Second,
$\pi_1 \circ Q\circ \phi_m(x,0)$ is generally not equal to $x$.
We shall use two maps $\Gamma^{(0)}$ and $\Gamma^{(1)}$ (sending
stable leaves to stable leaves) to compensate for these two
problems. The map $\Gamma^{(0)}$ will have the form
$\Gamma^{(0)}(x,y)=(x,G(x,y))$ where for fixed $x$, the map
$y\mapsto G(x,y)$ will be a diffeomorphism of the vertical leaf
$\{x\}\times \R^{d_s}$, so that $\pi_2 \circ Q \circ \phi_m
\circ \Gamma^{(0)} (x,y)=y$. In particular, $Q \circ
\phi_m\circ \Gamma^{(0)}(x,0)$ is of the form $(L^{(1)}(x),0)$,
for some map $L^{(1)}$. Choosing
$\Gamma^{(1)}(x,y)=((L^{(1)})^{-1}(x),y)$ solves our second
problem: the map
$$\phi_m^{(1)}\coloneqq Q \circ \phi_m
\circ \Gamma^{(0)} \circ \Gamma^{(1)}
$$ satisfies both
$\pi_2\circ \phi_m^{(1)}(x,y)=y$ and $\pi_1 \circ
\phi_m^{(1)}(x,0)=x$, as desired. Then, the map
$\Psi_m=(\Gamma^{(0)}\circ \Gamma^{(1)})^{-1}$
sends
stable leaves to stable leaves and $Q\circ \phi_m =
\phi_m^{(1)} \circ \Psi_m$.

We shall now be more precise, justifying the existence of the
maps mentioned above, and estimating their domain of
definition, their range and their smoothness.

\emph{The map $\Gamma^{(0)}$.} For fixed $x$, the map $y\mapsto
G(x,y)$ should satisfy $\pi_2 \circ Q \circ
\phi_m(x,G(x,y))=y$, i.e., it should be the inverse to the map
  \begin{equation}
  \label{def_Lx}
  L_x : y\mapsto \pi_2 \circ Q \circ \phi_m(x,y)=y-PF_m(x,y)\, ,
  \end{equation}
where we denote $\phi_m(x,y)=(F_m(x,y),y)$. We claim that this
map is invertible onto its image, and that there exists
$\epsilons^0>0$  such that
  \begin{equation}
  \label{ellx_diffeo}
  |L_x(y')-L_x(y)|\geq \epsilons^0 |y'-y|\, ,
  \quad \forall x\in B(x_m,C_0) \, , \quad \forall y,y'\in B(y_m,C_0)
  \, .
  \end{equation}
Indeed, fix $x \in B(x_m, C_0)$ and let $w=y'-y$. Writing
$F(y)=F_m(x,y)$, we have
\begin{equation}
  \label{jlqsmfdlkjqmdslf}
  L_x(y')-L_x(y)  =w -P \int_{t=0}^1 \partial_y F(y+tw) w \dd t \, .
  \end{equation}
Each vector $(\partial_y F(y+tw)w,w)$ belongs to $\tilde\CC^s$.
Since this cone is convexly transverse to $\R^{d_u}\times
\{0\}$, the set $\tilde\CC^s \cap (\R^{d_u} \times \{w\})$ is
convex, hence
  \begin{equation}
  \label{useconv}
  v_1\coloneqq
  \left(\int_{t=0}^1 \partial_y F(y+tw) w \dd t, w\right)
  \in \tilde \CC^s\, .
  \end{equation}
On the other hand, since the graph of $P$ is included in
$\tilde \CC^u$, $v_2\coloneqq (\int_{t=0}^1 \partial_y F(y+tw)
w \dd t, P \int_{t=0}^1 \partial_y F(y+tw) w \dd t)$ belongs to
$\tilde\CC^u$. Let $\epsilons^0>0$ be such that
$B(v,\epsilons^0|v|) \cap \tilde\CC^u =\emptyset$ for any $v\in
\tilde\CC^s -\{0\}$. Since $v_1 \in \tilde \CC^s$ and $v_2 \in
\tilde \CC^u$, we get $|v_1-v_2|\geq \epsilons^0 |v_1|$. As
$v_1$ and $v_2$ have the same first component, this gives
$|\pi_2(v_1)-\pi_2(v_2)|\geq \epsilons^0|v_1|$, i.e.,
  \begin{equation*}
  \left|w - P \int_{t=0}^1 \partial_y F(y+tw) w \dd t\right|\geq \epsilons^0 |w|\, ,
  \end{equation*}
which implies \eqref{ellx_diffeo} by \eqref{jlqsmfdlkjqmdslf}.

The map $\Lambda^{(0)}:(x,y)\mapsto (x,L_x(y))$ is well defined
on $B( m, C_0)$, its derivative is bounded by a constant $\Cs$,
and its second component satisfies \eqref{ellx_diffeo}. Lemma
~\ref{belongsDD} (with $x$ and $y$ exchanged) shows that
$\Lambda^{(0)} \in \DD(\Cs)$ for some constant $\Cs$. In
particular, $\Lambda^{(0)}$ admits an inverse $\Gamma^{(0)}$,
which also belongs to $\DD(\Cs)$.

By Lemma \ref{DDbigimage}, the range of $\Lambda^{(0)}$ (which
coincides with the domain of definition of $\Gamma^{(0)}$)
contains the ball $B( \Lambda^{(0)}(m), C_0/\Cs)$. Moreover,
$\Lambda^{(0)}(m)=Q m$. By \eqref{piQm}, we have $|Qm - \Pi
m|\leq \Cs$, hence the domain of definition of $\Gamma^{(0)}$
contains $B( \Pi m, C_0/\Cs -\Cs)$. If $C_0$ is large enough,
this contains $B(\Pi m, C_0^{2/3})$.

\emph{The map $\Gamma^{(1)}$.} Consider $\phi_m^{(0)}\coloneqq
Q\circ \phi_m \circ \Gamma^{(0)}$. It is a composition of maps
in $\DD(\Cs)$, hence it also belongs to $\DD(\Cs)$. Moreover,
its restriction to $\R^{d_u}\times \{0\}$ has the form
$(x,0)\mapsto (L^{(1)}(x),0)$. It follows that the map
$L^{(1)}$ (defined on a subset of $\R^{d_u}$) also satisfies
the inequalities defining $\DD(\Cs)$. In particular, it is
invertible, and we may define
$\Gamma^{(1)}(x,y)=((L^{(1)})^{-1}(x),y)$. This map belongs to
$\DD(\Cs)$. By construction, $\phi_m^{(1)} \coloneqq Q \circ
\phi_m \circ \Gamma^{(0)} \circ \Gamma^{(1)}$ can be written as
$(F_m^{(1)}(x,y),y)$ with $F_m^{(1)}(x,0)=x$.

We have $\phi_m^{(0)}(Q m)=Q m$. Since $|\Pi m-Qm| \leq \Cs$ by
\eqref{piQm}, and $\phi_m^{(0)}$ is Lipschitz, we obtain
$|\phi^{(0)}_m(\Pi m)-\Pi m|\leq \Cs$, i.e.,
$|L^{(1)}(x_m)-x_m|\leq \Cs$. Since $L^{(1)}\in \DD(\Cs)$,
Lemma~\ref{DDbigimage} shows that $L^{(1)}(B(x_m,C_0^{2/3}))$
contains the ball $B(x_m,C_0^{2/3}/\Cs-\Cs)$. Therefore, it
contains the ball $B(x_m,C_0^{1/2})$ if $C_0$ is large enough.
Hence, the domain of definition of the map $\Gamma^{(1)}$
contains $B(\Pi m, C_0^{1/2})$. This shows that $\phi_m^{(1)}$
is defined on $B(\Pi m, C_0^{1/2})$.

\emph{The map $\Psi_m$.} We can now define
$\Psi_m=(\Gamma^{(0)}\circ
\Gamma^{(1)})^{-1}=(L^{(1)}(x),L_x(y))$, so that $Q\circ
\phi_m=\phi_m^{(1)} \circ \Psi_m$. We have seen that $\Psi_m\in
\DD(\Cs)$, hence $D\Psi_m$ and $D\Psi_m^{-1}$ are uniformly
bounded. To show that $\Psi_m \in D^1_{1+\alpha}(\Cs)$, we
should check that $|D\Psi_m(x,y)-D\Psi_m(x,y')| \leq
\Cs|y-y'|^\alpha$. This follows directly from the construction
and the corresponding inequality \eqref{aalpha} for $DF_m$.
Finally, since $\Psi_m\in \DD(\Cs)$,
  \begin{equation*}
  \Psi_m( \phi_m^{-1}(B(m,d))) \subset \Psi_m(B(m, \Cs))
  \subset B( \Psi_m(m), \Cs)\, .
  \end{equation*}
Since $Q m=\phi_m^{(1)}(\Psi_m(m))$ and $\Pi m=\phi_m^{(1)}(\Pi
m)$, we get $|\Psi_m(m) -\Pi m|\leq \Cs |Qm-\Pi m|\leq \Cs$ by
\eqref{piQm}. Therefore, $\Psi_m( \phi_m^{-1}(B(m,d))) \subset
B( \Pi m, \Cs)$, and this last set is included in $B(\Pi m,
C_0^{1/2}/2)$ if $C_0$ is large enough.

\emph{The regularity of $DF_m^{(1)}$.} To finish the proof, we
should prove that $DF_m^{(1)}$ satisfies the
bounds defining $\KK(\Cs)$, for some constant
$\Cs$ independent of $C_0$. Since $\phi_m^{(1)}=Q\circ \phi_m
\circ \Gamma^{(0)}\circ \Gamma^{(1)}$, we have
  \begin{equation}
  \label{eqdphim1}
  D \phi_m^{(1)}=
  ( DQ \circ \phi_m \circ \Gamma^{(0)}\circ \Gamma^{(1)})
  \cdot (D\phi_m\circ \Gamma^{(0)}\circ \Gamma^{(1)})
  \cdot( D\Gamma^{(0)} \circ \Gamma^{(1)})
  \cdot D\Gamma^{(1)}\, .
  \end{equation}
Since $\KK$ is invariant under multiplication (Proposition~
\ref{KKinvmult}), and under composition by Lipschitz maps
sending stable leaves to stable leaves (Proposition~
\ref{KK_invcomp}), it is sufficient to show that $D\phi_m$,
$D\Gamma^{(0)}$, and $D\Gamma^{(1)}$ all satisfy the bounds
defining $\KK(\Cs)$.
For $D\phi_m$, this follows from our assumptions
(note that this is where \eqref{beeta}--\eqref{beeta'} are used).

Since $\Gamma^{(0)}=(\Lambda^{(0)})^{-1}$, we have
$D\Gamma^{(0)}=(D\Lambda^{(0)})^{-1}\circ \Gamma^{(0)}$. Since
$D\Lambda^{(0)}$ is expressed in terms of $DF_m$, it belongs to
$\KK$. As $\KK$ is invariant under inversion (Proposition
\ref{KKinvmult}) and composition, we obtain $D\Gamma^{(0)} \in
\KK(\Cs)$.

Since $D\phi_m^{(1)}(x,0)=\id$, it follows from
\eqref{eqdphim1} that, on the set $\{(x,0)\}$, $D\Gamma^{(1)}$
is the inverse of the restriction of a function in $\KK$, and in
particular $D\Gamma^{(1)}(x,0)$ is a $\beta$-H\"{o}lder continuous
function of $x$, by \eqref{KK_holder}. Since
$D\Gamma^{(1)}(x,y)$ only depends on $x$, it follows that
$D\Gamma^{(1)}$ belongs to $\KK$. This concludes the proof of
Lemma ~\ref{lemcomposeQ}.
\end{proof}

\smallskip

We return to the proof of Lemma~\ref{lemcompose}:

\emph{Second step: Gluing the foliations $\phi^{(1)}_{m}$
together.}

Let $\gamma(x,y)$ be a $C^\infty$ function equal to $1$ on the
ball $B(C_0^{1/2}/2)$, vanishing outside of $B(C_0^{1/2})$. Let
$\phi^{(1)}_{m}(x,y)=(F^{(1)}_m(x,y),y)$ be a foliation defined
by Lemma~ \ref{lemcomposeQ}, and put
  \begin{equation}
  \phi^{(2)}_{m}(x,y)=(\gamma(x-x_m,y)(F^{(1)}_m(x,y)-x)+x, y) \, .
  \end{equation}
Then $\phi^{(2)}_{m}$ defines a foliation on the ball of radius
$C_0^{1/2}$ around $\Pi m$, coinciding with $\phi^{(1)}_{m}$ on
$B(\Pi m,C_0^{1/2}/2)$, and $\phi^{(2)}_{m}$ is equal to the
identity on the boundary of $B(\Pi m,C_0^{1/2})$. By
construction, $\phi^{(2)}_m(x,y)=(F^{(2)}_m(x,y),y)$ with
$F^{(2)}_m(x,0)=x$. Moreover, $DF^{(2)}_m$ is expressed in
terms of $\gamma$, $D\gamma$, $F_m^{(1)}$ and $DF_m^{(1)}$. All
those functions belong to $\KK(\Cs)$ (the first three functions
are Lipschitz and bounded, hence in $\KK(\Cs)$, while we proved
in Lemma ~\ref{lemcomposeQ} that $DF_m^{(1)}\in \KK(\Cs)$).
Therefore, $DF^{(2)}_m \in \KK(\Cs)$ by Proposition
~\ref{KKinvmult}.

We proved in (a) that the balls $B(\Pi m, C_0^{1/2})$ for
$m\in\JJ'$ are disjoint, therefore all those foliations can be
glued together (with the trivial vertical foliation outside of
$\bigcup_{m\in \JJ'} B(\Pi m,C_0^{1/2})$), to get a single
foliation parameterized by $\phi^{(2)}:\R^d \to \R^d$. We
emphasize that this new foliation is not necessarily contained
in the cone $Q( \tilde\CC^s)$, since the function $\gamma$
contributes to the derivative of $\phi^{(2)}$. Nevertheless, it
is uniformly transverse to the direction $\R^{d_u}\times\{0\}$,
and this will be sufficient for our purposes. Let us write
$\phi^{(2)}(x,y)=(F^{(2)}(x,y),y)$, where $F^{(2)}$ coincides
everywhere with a function $F_m^{(2)}$ or with the function
$(x,y)\mapsto x$. Since all the derivatives of those functions
belong to $\KK(\Cs)$, it follows that $DF^{(2)}\in \KK(\Cs)$
(for some other constant $\Cs$, worse than the previous one due
to the gluing). Since we will need to reuse this last constant,
let us denote it by $\Cs^{(0)}$.

\medskip

\emph{Third step: Pushing the  foliation $\phi^{(2)}$ with
$D^{-1}$.} This step is very simple, although this is where
\eqref{suffhyp} is needed: Define a new foliation by
  \begin{equation}
  \label{heart}
  F^{(3)}(x,y)=A^{-1}F^{(2)}(Ax, By), \quad \phi^{(3)}(x,y)=(F^{(3)}(x,y),y) \, ,
  \end{equation}
so that $D^{-1}\phi^{(2)}=\phi^{(3)}D^{-1}$. The map $F^{(3)}$
satisfies $F^{(3)}(x,0)=x$. Moreover
  \begin{align*}
  \partial_x F^{(3)}(x,y)=A^{-1}(\partial_x F^{(2)})(Ax, By) A\,,\quad
  \partial_y F^{(3)}(x,y)=A^{-1}(\partial_y F^{(2)})(Ax, By) B\, .
  \end{align*}
In particular, if $|A^{-1}|$ and $|B|$ are small enough (which
can be ensured by decreasing $\epsilon$ in \eqref{suffhyp}), we
can make $\partial_y F^{(3)}$ arbitrarily small. Since $|B|\leq
1 \leq |A|$, it also follows that
  \begin{equation}
  \label{heart'}
  \begin{split}
  \raisetag{-30pt}
  |D F^{(3)}(x,y)-D F^{(3)}(x,y')| &\leq |A^{-1}| |A| |DF^{(2)}(Ax,By)-DF^{(2)}(Ax,By')|
  \\
  &\!\!\!\!\!\leq |A^{-1}||A| \Cs^{(0)} |By-By'|^\alpha
  \leq |A^{-1}| |A| \Cs^{(0)} |B|^\alpha |y-y'|^\alpha \, .
  \end{split}
  \end{equation}
In the same way,
  \begin{equation}
  \label{heart''}
  \begin{split}
  |DF^{(3)}(x,y)-D&F^{(3)}(x,y')-DF^{(3)}(x',y)+DF^{(3)}(x',y')|
  \\&
  \leq |A^{-1}| |A| |DF^{(2)}(Ax,By)-DF^{(3)}(Ax,By')
  \\&\quad\quad \quad\quad\quad\quad\quad
  -DF^{(3)}(Ax',By)+DF^{(3)}(Ax',By')|
  \\&
  \leq |A^{-1}| |A| \Cs^{(0)} |Ax-Ax'|^\beta |By-By'|^{\alpha-\beta}
  \\&
  \leq |A^{-1}| |A| \Cs^{(0)} |A|^\beta |B|^{\alpha-\beta} |x-x'|^\beta |y-y'|^{\alpha-\beta}
  \, .
  \end{split}
  \end{equation}
If the bunching constant $\epsilon$ in \eqref{suffhyp} is small
enough (depending on $C_1$), we can ensure that the two last
equations are bounded, respectively, by $|y-y'|^\alpha/(2C_1)$
and $|x-x'|^\beta |y-y'|^{\alpha-\beta} / (4C_0^2C_1)$, i.e.,
the map $F^{(3)}$ satisfies the requirements \eqref{aalpha} and
\eqref{beeta'} for admissible foliations, with better constants
that will be useful below.

Taking $y'=0$ in \eqref{heart''}, we obtain
  \begin{equation*}
  |DF^{(3)}(x,y) - DF^{(3)}(x',y)|
  \leq |x-x'|^\beta |y|^{\alpha-\beta} / (4C_0^2C_1)+ |DF^{(3)}(x,0) - DF^{(3)}(x',0)|.
  \end{equation*}
Moreover, $\partial_x F^{(3)}(x,0)=\partial_x
F^{(3)}(x',0)=\id$, so that
  \begin{align*}
  |DF^{(3)}(x,0) - DF^{(3)}(x',0)|
  &
  =|\partial_y F^{(3)}(x,0)-\partial_y F^{(3)}(x',0)|
  \\&
  \leq |A^{-1}| |B| |\partial_y F^{(2)}(Ax,0) -\partial_y F^{(2)}(Ax',0)|
  \\&
  \leq |A^{-1}| |B| \Cs^{(0)} |Ax-Ax'|^\beta
  \\&
  \leq |A^{-1}| |B| \Cs^{(0)} |A|^\beta |x-x'|^\beta\, .
  \end{align*}
The quantity $|A^{-1}| |B| |A|^\beta$ is bounded by $\Cs
\lambda_u^{-1} \lambda_s \Lambda_u^\beta$. Choosing $\epsilon$
small enough in \eqref{suffhyp}, it can be made arbitrarily
small. For $|y|\leq C_0^2$, this yields
  \begin{equation}
  |DF^{(3)}(x,y) - DF^{(3)}(x',y)| \leq |x-x'|^\beta/(2C_1)\, ,
  \end{equation}
which is a small reinforcement of \eqref{beeta}.

\medskip

\emph{Fourth step: Pushing the foliation $\phi^{(3)}$ with
$(Q')^{-1}$.} Define a map $F^{(4)}(x,y)=F^{(3)}(x,y)+P'y$, and
let $\phi^{(4)}(x,y)=(F^{(4)}(x,y),y)$. The corresponding
foliation is the image of $\phi^{(3)}$ under $(Q')^{-1}$. Let
us fix a cone $\CC^s_1$ which sits compactly  between $\CC^s_0$
and $\CC^s$. Since the graph $\{(P'y,y)\}$ is contained in
$\CC^s_0$, the foliation $F^{(4)}$ is contained in $\CC^s_1$ if
$\partial_y F^{(3)}$ is everywhere small enough. Moreover, the
bounds of the previous step concerning $DF^{(3)}$ directly
translate into the following bounds for $DF^{(4)}$, for all
$x,x'\in \R^{d_u}$ and all $y,y'\in B(0,C_0^2)$:
  \begin{align}
  \label{controleF4}
  &| DF^{(4)}(x,y) -DF^{(4)}(x,y') | \leq |y-y'| ^\alpha/ (2C_1)\, ,
  \\ \label{controleF4'}
  &|DF^{(4)}(x,y)- DF^{(4)}(x',y)|\leq  |x-x'|^\beta /(4C_0^2C_1)\, ,
  \\ \label{controleF4''}
  \begin{split}
  &| D F^{(4)}(x,y)- DF^{(4)} (x,y')-DF^{(4)}(x',y)+DF^{(4)}(x',y') |
  \\& \hphantom{| DF^{(4)}(x,y) -DF^{(4)}(x,y') |}
  \leq |x-x'|^\beta |y-y'|^{\alpha-\beta}/(2C_1)  \, .
  \end{split}
  \end{align}
In particular, since $\partial_x F^{(4)}(x,0)=\id$, the bound
\eqref{controleF4} implies that $\partial_x F^{(4)}$ is bounded
and has a bounded inverse on a ball of radius $C_1\geq 2C_0$.

\medskip

\emph{Last step: Pushing the foliation $\phi^{(4)}$ with
$\TT^{-1}M$.} Let $\UU = \TT^{-1}M$, and consider $\phi'$ the
foliation obtained by pushing the foliation $\phi^{(4)}$ with
$\UU$. We claim that $\phi'$ belongs to $\FF(0,\CC^s, C_0,
C_1)$, and that we can write $\UU\circ \phi = \phi' \circ
\Psi'$ for some $\Psi'\in D_1^{1+\alpha}(\Cs)$.

To prove this, we follow the arguments in the proof of Lemma
~\ref{lemcomposeQ} (with simplifications here since $\UU$ is
close to the identity). First, fix $x$ and consider the map
$L_x : y\mapsto \pi_2\circ \UU \circ \phi^{(4)}(x,y)$. Writing
$\UU=\id+\VV$ where $\norm{\VV}{C^{1+\alpha}}\leq \epsilon$, we
have $L_x(y)= y + \pi_2\circ V (F^{(4)}(x,y),y)$. Since
$F^{(4)}$ is bounded in $C^1$ on the ball $B(0,2C_1)$, it
follows that, if $\epsilon$ is small enough, then the
restriction of $L_x$ to the ball $B(0,2C_1)$ (in $\R^{d_s}$) is
arbitrarily close to the identity. Therefore, its inverse is
well defined, and we can set $\Gamma^{(0)}(x,y)=(x,
L_x^{-1}(y))$. By construction, the map $\UU\circ
\phi^{(4)}\circ \Gamma^{(0)}(x,y)$ has the form $(L^{(1)}(x),
y)$ for some function $L^{(1)}$, which is bounded in
$C^{1+\alpha}$ and arbitrarily close to the identity in $C^{1}$
if $\epsilon$ is small. Let
$\Gamma^{(1)}(x,y)=((L^{(1)})^{-1}(x), y)$, then the map
$\phi'\coloneqq \UU \circ \phi^{(4)}\circ \Gamma^{(0)}\circ
\Gamma^{(1)}$ is defined on the set $\{(x,y) \st |y|\leq C_1\}$
(which contains $B(0,C_0)$), and it takes the form
$\phi'(x,y)=(F'(x,y), y)$ for some function $F'$ with
$F'(x,0)=0$.

Since $\phi'$ is obtained by composing $\phi^{(4)}$ with
diffeomorphisms arbitrarily close to the identity, it follows
from \eqref{controleF4}--\eqref{controleF4''} that $F'$
satisfies \eqref{aalpha}--\eqref{beeta'}. Moreover, since
$(\partial_y F^{(4)}(z)w,w)$ takes its values in the cone
$\CC^s_1$, it follows that $(\partial_y F'(z)w,w)$ lies in the
cone $\CC^s$ if $\UU$ is close enough to the identity. Hence,
the foliation defined by $\phi'$ is contained in $\CC^s$. This
shows that $\phi'$ belongs to $\FF(0,\CC^s, C_0, C_1)$.

Finally, the function $\Psi = (\Gamma^{(0)}\circ
\Gamma^{(1)})^{-1}$ belongs to $D_1^{1+\alpha}(\Cs)$. This
concludes the proof of Lemma~\ref{lemcompose}.
\end{proof}

\begin{remark}
\label{remOKgrand}
An inspection of the proof of Lemma~\ref{lemcompose} shows that
one can obtain stronger conclusions: For any $C'>0$, one can
ensure that the final chart $\phi'$ is defined on a ball of
radius $C'$, and satisfies $|D\phi'(x,y)-D\phi'(x,y')|\leq
|y-y'|^\alpha/C'$, as follows. If the bunching and
hyperbolicity conditions in \eqref{suffhyp} are large enough,
the third step of the proof yields a chart $\phi^{(3)}$ with $
|D F^{(3)}(x,y)-D F^{(3)}(x,y')| \leq \delta |y-y'|^\alpha$ for
arbitrarily small $\delta>0$. Hence, in the inequality
\eqref{controleF4} regarding the map $F^{(4)}$ (which is
defined on $\real^d$), the constant $2 C_1$ can be replaced
with an arbitrarily large constant, allowing an arbitrarily
large domain of definition for $\phi'$. The same observation
holds for \eqref{controleF4'} and \eqref{controleF4''}. This
remark is the key to the proof of Proposition~\ref{propinvar}.
\end{remark}

%%%%%%%%%%%%%%%%%%%%%%%%%%%%%%%%%%%%%%%%%%%%%%%%%%%%%%%%%%
\section{Results on the local spaces \texorpdfstring{$H_p^{t,s}$}{Hpts}.}
\label{sec:local}

\subsection{Basic facts on the local spaces \texorpdfstring{$H_p^{t,s}$}{Hpts}.}

\label{basicc}

We start with  reminders from \cite{baladi_gouezel_piecewise}.

The proof of Lemma 22 from \cite{baladi_gouezel_piecewise}
implies the following:

\begin{lemma}
\label{Leib} Let $t>0$, $s<0$ and $\tilde \alpha>0$ be real numbers
with $t+|s|<\tilde \alpha$. For any $p \in (1,\infty)$, there exists a
constant $\Cs$ such that for any $C^{\tilde \alpha}$ function $g :
\real^d \to \complex$,
  \begin{equation*}
  \norm{ g \cdot \omega}{H_p^{t,s}}\le \Cs
  \|g\|_{C^{\tilde \alpha}} \norm{\omega}{H_p^{t,s}}\, .
  \end{equation*}
\end{lemma}

The following extension of a classical result of Strichartz
(see \cite[Lemma 23]{baladi_gouezel_piecewise}) is the key to
our results. It follows from Lemma 23 of
\cite{baladi_gouezel_piecewise} and a linear change of
coordinates.

\begin{lemma}
\label{lem:multiplier} Let $1<p<\infty$ and $1/p-1<s\leq 0
\leq t <1/p$. Let $e_1,\dots,e_d$ be a basis of $\R^d$, such
that $e_{d_u+1},\dots,e_d$ form a basis of $\{0\}\times
\R^{d_s}$. There exists a constant $\Cs$ (depending only on
$p,s,t$ and the norm of the matrix change of coordinate between
$e_1,\dots,e_d$ and the canonical basis of $\R^d$) so that, for
any subset $U$ of $\real^d$ whose intersection with almost
every line directed by a vector $e_i$ has at most $M$ connected
components,
  \begin{equation*}
  \norm{1_{U} \omega}{H_p^{t,s}} \leq \Cs M
  \norm{\omega}{H_p^{t,s}}\, .
  \end{equation*}
\end{lemma}

The following is essentially Lemma 28 in
\cite{baladi_gouezel_piecewise}.

\begin{lemma}[Localization principle]
\label{lem:localization}
Let $\mathbb K$ be a compact subset of $\R^d$. For each $m\in
\Z^d$, consider a function $\eta_m$ supported in $m+\mathbb K$,
with uniformly bounded $C^1$ norm. For any $p\in (1,\infty)$
and $t$, $s\in \real$ with $|t|+|s|<1$, there exists $\Cs >0$
so that for each $\omega \in H_p^{t,s}$
  \begin{equation*}
  \left(\sum_{m\in \integer^d} \norm{\eta_m \omega }{H_p^{t,s}}^p\right)^{1/p}
  \le \Cs  \norm{\omega}{H_p^{t,s}} \, .
  \end{equation*}
\end{lemma}
\begin{proof}
Consider  a  compactly supported $C^\infty$ function $\gamma$,
equal to $1$ on $\mathbb K$, and write $\gamma_m(z)=\gamma(z-m)$.
Then $\eta_m=\eta_m \gamma_m$, and
  \begin{equation*}
  \norm{\eta_m \omega }{H_p^{t,s}}
  =\norm{\eta_m \gamma_m \omega }{H_p^{t,s}}
  \leq \Cs \norm{\gamma_m \omega}{H_p^{t,s}}
  \end{equation*}
by Lemma \ref{Leib}. The result follows by applying \cite[Lemma
28]{baladi_gouezel_piecewise}.
\end{proof}

For any real number $t$ and any $1<p<\infty$, we set
$H_p^{t}(X_0)$ to be  the Sobolev-Triebel space defined as the
distributions that have finite $H_p^{t}(\R^d)$ norm in any
(fixed) smooth coordinate system.

As usual, a compact imbedding statement \`{a} la Arzel\`{a}-Ascoli will
be used (recall that $X_0$ is compact):

\begin{lemma}\label{embed}
Let $s<0<t$ with $t+|s|<1$, and let $1 < p < \infty$. Assume
that $t-|s|> - \beta$. Then, for any $R,C_0,C_1$, the space
$\HHH_p^{t,s}(R, C_0, C_1)$ is continuously embedded in
$H_p^{t-|s|}(X_0)$. In addition, we have the continuous
embeddings
  \begin{equation}
  \label{eq_embedding}
  \HHH_p^{t,s}(R,C_0,C_1)\subset \HHH_p^{t',s'}(R,C_0,C_1) \text{ if }t' \le t \text{ and }s' \le s\, .
  \end{equation}
Moreover, this inclusion is compact if $t'<t$.
\end{lemma}

\begin{proof}
Before proving the lemma, we start with a functional analytic
preliminary, required because $t-|s|$ will be strictly negative
in our application of the lemma. If $1/p+1/p'=1$ for $1<p,\ p'
< \infty$, and $r > 0$, then classical duality results (see
e.g. \cite[Lemma 20]{baladi_gouezel_piecewise} and references
therein) yield $(H_{ p}^r)^*=H_{p'}^{-r}$. If $G$ is a
diffeomorphism of $\real^d$ then the  dual operator $L^*$ on
$H_p^r$ to $ L(w')= w' \circ G$ is $w\mapsto |\det DG^{-1}| w
\circ G^{-1}$. For $r\in [0,1]$, $H_p^{r}$ is invariant under
the composition by a $C^1$ diffeomorphism $G$ (since this is
the case of $H_p^0=L^p$, and $H_p^1$). By duality, $H_p^{-r}$
is invariant by $w\mapsto |\det DG^{-1}|\cdot w \circ G^{-1}$.
Therefore, Lemma \ref{Leib} shows that $H_p^{-r}$ is invariant
under the composition with diffeomorphisms whose jacobian is
$C^\beta$ for some $\beta >r$.

We now turn to the proof of the lemma. In any admissible chart,
the continuous embedding claim \eqref{eq_embedding} follows
from the definitions and properties of Triebel spaces, taking
the supremum over all admissible charts. For the rest of the
proof, let us fix $R,C_0,C_1$. To simplify notations, we will
write $\HHH^{t,s}_p$ for $\HHH^{t,s}_p(R,C_0,C_1)$.

Consider now $s'\leq s$ and $t'<t$. Fix also $t_0<t$ with
$t_0-|s|>-\beta$. Since $H_p^{t,s}$ is included in
$H_p^{t-|s|,0}$, it follows by taking the supremum over the
admissible charts that $\HHH_p^{t,s}$ is included in
$\HHH_p^{t-|s|,0}$. Moreover, for any admissible charts
$\phi_1, \phi_2 \in \FF(\bm)$ for some $\bm$ (recall
\eqref{ourcharts}), the change of coordinates $\phi_2\circ
\phi_1^{-1}$ is $C^1$ and has a (uniformly) $C^\beta$ Jacobian.
It follows from the functional analytic preliminary that
changing the system $\Phi$ of charts in the definition of the
$\HHH_p^{t-|s|,0}$-norm gives equivalent norms. Therefore,
$\HHH_p^{t-|s|,0}$ is isomorphic to the Triebel space
$H_p^{t-|s|}(X_0)$. Since the inclusion of $H_p^{t-|s|}(X_0)$
in $H_p^{t_0-|s|}(X_0)$ is compact, it follows that the
inclusion $\HHH_p^{t,s} \to \HHH_p^{t_0-|s|,0}$ is also
compact.

Consider now a sequence $\omega_n\in \HHH_p^{t,s}$, with norms
bounded by $1$. To prove that the inclusion of $\HHH_p^{t,s}$
in $\HHH_p^{t',s'}$ is compact, it is sufficient to show that,
for any $\epsilon$, there exists a subsequence of $\omega_n$
along which
  \begin{equation}
  \label{lqmkdsjfposdf}
  \limsup \norm{\omega_m-\omega_n}{\HHH_p^{t',s'}} \leq 2\epsilon.
  \end{equation}
We can assume without loss of generality that $\omega_n$
converges in $\HHH_p^{t_0-|s|,0}$. Let $C(\epsilon)$ be such
that any distribution $\omega$ on $\R^d$ satisfies
  \begin{equation}
  \label{goodcharts}
  \norm{\omega}{H_p^{t',s'}}\leq \epsilon \norm{\omega}{H_p^{t,s}}+C(\epsilon)
  \norm{\omega}{H_p^{t_0-|s|,0}}.
  \end{equation}
To prove that such a constant $C(\epsilon)$ exists, let us note
that the kernel $a_{t',s'}$ defining the $H_p^{t',s'}$-norm is
bounded by $\epsilon a_{t,s}$ outside of a compact set, where
it is bounded by $C(\epsilon) a_{t_0-|s|,0}$ if $C(\epsilon)$
is large enough. Therefore, \eqref{goodcharts} follows from the
Marcinkiewicz multiplier theorem (see e.g. \cite[Theorem
21]{baladi_gouezel_piecewise} or \cite[Theorem
2.4/2]{triebel_III}).

Taking the supremum of the equation \eqref{goodcharts} over the
admissible charts, we obtain
  \begin{equation}
  \norm{\omega_n-\omega_m}{\HHH_p^{t',s'}}
  \leq \epsilon \norm{\omega_n-\omega_m}{\HHH_p^{t,s}}
  +C(\epsilon) \norm{\omega_n-\omega_m}{\HHH_p^{t_0-|s|,0}}.
  \end{equation}
Since the quantity
$\norm{\omega_n-\omega_m}{\HHH_p^{t_0-|s|,0}}$ converges to $0$
when $n,m\to\infty$, this proves \eqref{lqmkdsjfposdf}.
\end{proof}

The following lemma on partitions of unity is Lemma 32 from
\cite{baladi_gouezel_piecewise}:
\begin{lemma}
\label{lem:sum} Let $t$ and $s$ be
arbitrary real numbers. There exists a constant $\Cs$ such
that, for any distributions $v_1,\dots, v_l$ with compact
support in $\R^d$, belonging to $H_p^{t,s}$, there exists a
constant $C$ depending only on the supports of the
distributions $v_i$ with
  \begin{equation}
  \norm{ \sum_{i=1}^l v_i}{H_p^{t,s}}^p \leq \Cs m^{p-1} \sum_{i=1}^l
  \norm{v_i}{H_p^{t,s}}^p + C \sum_{i=1}^l
  \norm{v_i}{H_p^{t-1,s}}^p,
  \end{equation}
where $m$ is the intersection multiplicity of the supports of
the $v_i$'s, i.e., $m=\sup_{x\in \R^d} \Card\{i \st x\in
\supp(v_i)\}$.
\end{lemma}

%%%%%%%%%%%%%%%%%
\subsection{The effect of composition on the local space \texorpdfstring{$H_p^{t,s}$}{Hpts}.}

In view of Theorem \ref{MainTheorem}, we describe how the local
spaces $H_p ^{t,s}$  behave under  composition with hyperbolic
matrices and appropriate maps preserving the stable leaves.

The following lemma is a particular case of \cite[Lemma
25]{baladi_gouezel_piecewise}.

\begin{lemma}
\label{CompositionDure}
For all $s<0<t$ and $t-|s|<0$, for all $p\in (1,\infty)$,
and every $t'<t$ there exists a constant $\Cs$ (depending only on
$t$, $s$, $p$, $t'$) so that the following holds:
Let $D=\left(\begin{smallmatrix} A&0\\0&B
\end{smallmatrix}\right)$
be a block diagonal matrix  such that $|Av|\geq \lambda_u |v|$
and $|Bv|\leq \lambda_s |v|$ for $\lambda_u>1$ and
$\lambda_s<1$. Then there exists a constant $C$ such that, for
all $\omega\in H_p^{t,s}$,
  \begin{align*}
  \norm{\omega\circ D^{-1}}{H_p^{t,s}}&\leq \Cs
  |\det D|^{1/p} \max(\lambda_u^{-t}, \lambda_s^{-(t+s)})\norm{\omega}{H_p^{t,s}}
   + C \norm{\omega}{H_p^{t',s}} \, .
  \end{align*}
\end{lemma}

Adapting the second part of the proof of \cite[Lemma
25]{baladi_gouezel_piecewise} gives:

\begin{lemma}
\label{lemcomposeD1alpha}
Let $C>0$, and let $-\alpha<s<0<t<1$ with $\alpha t+|s|<
\alpha$. There exists a constant $C'>0$  so that for any
$\Psi\in D^1_{1+\alpha}(C)$  whose range contains a ball
$B(z,C_0^{1/2})$, and for any distribution $\omega\in
H_p^{t,s}$ supported in $B(z,C_0^{1/2}/2)$, the composition
$\omega\circ \Psi$ is well defined, and
  \begin{equation}
  \norm{\omega \circ \Psi}{H_p^{t,s}}\leq C' \norm{\omega}{H_p^{t,s}}\, .
  \end{equation}
\end{lemma}

\begin{proof}
Without loss of generality, we may assume $z=\Psi^{-1}(z)=0$.
Let $\gamma$ be a $C^\infty$ function equal to $1$ on
$B(0,C_0^{1/2}/2)$ and vanishing outside of $B(0,C_0^{1/2})$.
We want to show that the operator $\MM : \omega\mapsto (\gamma
\omega) \circ \Psi$ is bounded by $C'$ as an operator from
$H_p^{t,s}$ to itself. By interpolation, it is sufficient to
prove this statement for $H_p^{1,0}$, for $L^p$, and for
$H_p^{0,-\alpha}$. This is done in the second step of the proof
of Lemma 25 in \cite{baladi_gouezel_piecewise} -- the result
there is formulated for $C^{1+\alpha}$ diffeomorphisms, but a
glance at the proof there indicates that the $C^\alpha$
regularity of the jacobian is only used along the stable
leaves, in the argument for $H_p^{0,-\alpha}$, and the
definition of $D^1_{1+\alpha}(C)$ ensures that the jacobian is
indeed regular along stable leaves.
\end{proof}

%%%%%%%%%%%%%%%%%%%%%%%%%%%%%%%%%%%%%%%%%%%%

\section{Proof of the main theorem on piecewise cone hyperbolic maps}
\label{mainsec}

In this section, we prove Theorem \ref{MainTheorembis} and
Proposition~ \ref{propinvar}.

We may fix once and for all a constant $C_0>1$ large enough so
that the assumptions of Lemma~\ref{lemcompose} are satisfied
for the finite set $\CC_{i,j}$ of extended cones chosen in
Section~ \ref{spaces}.

The following lemma implies Theorem \ref{MainTheorembis} since
the inclusion of $\HHH_p^{t,s}$ into $\HHH_p^{t',s}$ is compact
for $s<0<t$ if $t'<t$, and $t+|s|<1$, $t-|s|>-\beta$, by
Lemma~\ref{embed}.

\begin{lemma}
\label{maintechlemma}
Let $\alpha$, $T$, $g$, $p$ be as in Theorem ~
\ref{MainTheorem} and let $1/p-1<s<0<t<1/p$, with
$\alpha|s|+t<\alpha$. For any $t'<t$ there is $\Cs$ so that,
for any $N$, if $C_1$ is large enough, then for any large
enough $n$ which is a multiple of $N$, and for any large enough
$R$, there exists $D_n$ so that
 \begin{multline}
  \label{eq:main}
  \norm{\LL_g^n \omega}{\HHH^{t,s}_p(R,C_0,C_1)}^p\leq D_n \norm{\omega}{\HHH_p^{t',s}(R,C_0,C_1)}^p
  + \Cs (\Cs N^p)^{n/N} D_n^b (D_n^e)^{p-1} \times
  \\ \times\norm{ |\det DT^n|
  \max(\lambda_{u,n}^{-t}, \lambda_{s,n}^{-(s+t)})^p
  |g^{(n)}|^p}{L^\infty}\norm{\omega}{\HHH^{t,s}_p(R,C_0,C_1)}^p \, .
\end{multline}
\end{lemma}

\begin{proof}[Proof of Lemma ~ \ref{maintechlemma}]
To simplify notation, we  write $x\leqc y$ if $x\leq y$ up to
compact terms, i.e., terms which are controlled by
$\norm{\omega}{\HHH_p^{t',s}(R,C_0,C_1)}$ for some $t'<t$. Note
that if $t''<t$ is such that $t''<t'$, then an upper bound in
terms of $t''$ trivially implies the upper bound for $t'$
because $\norm{\omega}{H^{t'',s}_p}\le C
\norm{\omega}{H^{t',s}_p}$. Conversely, an upper bound in terms
of $t'$ implies the upper bound for $t''$ because, for any
$\epsilon>0$, there exists a constant $C(\epsilon)$ so that for
all $v$
  \begin{equation*}
  \norm{\omega}{H^{t',s}_p}\le \epsilon \norm{\omega}{H^{t,s}_p}+ C(\epsilon)\norm {\omega}{H^{ t'',s}_p} \, .
  \end{equation*}
(The above bound is proved just like \eqref{goodcharts}.) We
shall apply the above remark implicitly whenever we have a
bound $x\leqc y$. This allows us to replace $t''<0$ by $0<t'
<t$ when invoking Lemmas~ \ref{Leib} or
~\ref{lemcomposeD1alpha}.

Before starting the proof, let us describe the order in which
we choose the constants. First, $N$ is fixed in the statement
(it will be used in the second step of the proof in order to
apply Lemma ~\ref{lem:multiplier}).
Then, we choose $C_1$ very large, in the second step below, so
that the admissible charts $\phi_\bm$ are close enough to
linear maps ($C_1$ depends on $N$). Then, we fix $n$
to be some very
large multiple of $N$, depending on $C_1$ (it should be large enough so that
every branch of $T^n$ is hyperbolic enough so that
Lemma~\ref{lemcompose} applies). Finally, we choose $R$ very large so
that, at scale $1/R$, all the iterates of $T$ up to time $n$
look like linear maps, and all the boundaries of the sets we
are interested in look like hyperplanes. For the presentation
of the argument, we will start the proof with some values of
$C_1, n,R$, and increase them whenever necessary, checking each
time that $C_1$ does not depend on $n,R$, and that $n$ does not
depend on $R$, to avoid bootstrapping issues. We will denote by
$\Cs$ a constant that does not depend on $N,C_1,n,R$, and may
vary from line to line.

\medskip

For every $\ii \in I^n$, we fix a small neighborhood $\tilde
O_\ii$ of $\overline{O_\ii}$ such that $T_\ii$ admits an
extension to $\tilde O_\ii$ with the same hyperbolicity
properties as the original $T_\ii$. Reducing these sets if
necessary, we can ensure that their intersection multiplicity
is bounded by $D_n^b$, and that the intersection multiplicity
of the sets $T_\ii \tilde{O_\ii}$ is bounded by $D_n^e$.

For $\bm=(i,j,m)\in \ZZ(R)$, let us write
$$A(\bm)=A(\bm,R)=(\kappa_\bm^R)^{-1}(B(m,d)) \subset X\, .
$$
The set $A(\bm)$ is a neighborhood of $q_\bm$, of diameter
bounded by $\Cs R^{-1}$, and containing the support of
$\brho_{\bm}$.

Let us fix some system of charts $\Phi$ as in the Definition ~
\ref{defnorm} of the ${\HHH_p^{t,s}(R, C_0, C_1)}$-norm. We
want to estimate $\norm{\LL_g^n \omega}{\Phi}$.

\emph{First step.} The  sets $\{T_\ii^n(O_\ii)\st \ii \in I^n\}$
have intersection multiplicity at most $D_n^e$.
Writing\footnote{Elements of $L^\infty$ are defined almost everywhere,
and the transfer operator is defined initially on $L^\infty$,
so the fact that $\bigcup_i O_i=X_0$ only modulo a zero
Lebesgue measure set is irrelevant.} $\LL_g^n \omega=\sum_{\ii}
1_{T_\ii^n O_\ii}(g^{(n)} \omega)\circ T_\ii^{-n}$, we get by
Lemma \ref{lem:sum} that for each $\bm \in \ZZ(R)$
  \begin{multline*}
  \norm{(\brho_\bm\cdot 1_{O_\bm} \LL_g^n \omega)\circ
  \bphi_\bm}{H_p^{t,s}}^p
  \\ \leqc \Cs (D_n^e)^{p-1} \sum_{\ii \in I^n}
  \norm{(\brho_\bm 1_{O_\bm} 1_{T_\ii^n O_\ii}(g^{(n)} \omega)\circ T_\ii^{-n})\circ \bphi_\bm}{H_p^{t,s}}^p \, .
  \end{multline*}
Summing over $\bm \in \ZZ(R)$, we obtain
  \begin{equation*}
  \norm{\LL_g^n \omega}{\Phi}^p \leqc
   \Cs (D_n^e)^{p-1}
  \sum_{\bm \in \ZZ(R), \ii\in I^n}
  \norm{(\brho_\bm 1_{O_\bm} 1_{T_\ii^n O_\ii}(g^{(n)} \omega)\circ
  T_\ii^{-n})
  \circ \bphi_\bm}
  {H_p^{t,s}}^p \, .
  \end{equation*}

For $i\in I$, let  $U_{i,j,2}$, $1\leq j\leq N_i$, be arbitrary
open sets covering a fixed neighborhood $\tilde O_i^0$ of
$\overline{O_i}$, such that $\overline{U_{i,j,2}}\subset
U_{i,j,1}$ (they do not depend on $n$, $R$, or any other
choice). For each $\bm\in \ZZ(R)$, and
$\ii=(i_0,\dots,i_{n-1})\in I^n$ such that $T_\ii^n O_\ii$
intersects $A(\bm)$, the point $T_\ii^{-n}(q_\bm)$ belongs to
$\tilde O_{i_0}^0$ if $R$ is large enough, we can therefore
consider $k$ such that it belongs to $U_{i_0,k,2}$. Then
$\sum_{\ell\in \ZZ_{i_0,k}(R)}\brho_{i_0,k,\ell}$ is equal to
$1$ on a neighborhood of fixed size of $T_\ii^{-n}(q_\bm)$, so
that $\sum_{\ell\in \ZZ_{i_0,k}(R)} \brho_{i_0,k,\ell} \circ
T_\ii^{-n}$ is equal to $1$ on $A(\bm)$ if $R$ is large enough
(depending on $n$ but not on $\Phi$ or $\bm$). Since the
intersection multiplicity of the supports of the
$\brho_{i_0,k,\ell}\circ T_\ii^{-n}$ is uniformly bounded,
Lemma \ref{lem:sum} gives, if $R$ is large enough (uniformly in
$\Phi$, $\bm$, $k$, $\ii$)
  \begin{multline*}
  \norm{
	(\brho_{\bm}1_{O_\bm} 1_{T_\ii^n O_\ii}(g^{(n)} \omega)
	\circ T_\ii^{-n}) \circ \bphi_\bm
  }{H_p^{t,s}}^p
  \\
  \leqc
  \Cs \sum_{\ell\in \ZZ_{i_0,k}(R)}
  \norm{
	(\brho_\bm 1_{O_\bm} 1_{T_\ii^n O_\ii}(\brho_{i_0,k,\ell} \cdot g^{(n)} \omega)
	\circ T_\ii^{-n}) \circ \bphi_\bm
  }{H_p^{t,s}}^p \, .
  \end{multline*}

Taking $R$ large enough and summing over $\bm\in \ZZ(R)$,
$\ii\in I^n$ and $k$ in $\{1, \ldots, N_{i_0}\}$ such that
$T_\ii^{-n}(q_\bm) \in U_{i_0,k,2}$,  we get (writing
$\bell=(i_0,k,\ell)\in \ZZ(R)$)
  \begin{equation}
  \label{klqjsdfml}
  \norm{\LL_g^n \omega}{\Phi}^p
  \leqc \Cs (D_n^e)^{p-1} \sum_{\bm, \ii, \bell}
  \norm{(\brho_\bm 1_{O_\bm} 1_{T_\ii^n O_\ii}( \brho_\bell \cdot g^{(n)} \omega)\circ T_\ii^{-n})
  \circ \bphi_\bm}
  {H_p^{t,s}}^p \, ,
  \end{equation}
where the sum is restricted to those $(\bm,\ii,\bell)$ such
that the support of $\brho_{\bell}$ is included in $\tilde
O_\ii$, the support of $\brho_\bm$ is included in $T_\ii \tilde
O_\ii$, and $O_\bell=O_{i_0}$ (this restriction will be
implicit in the rest of the proof).

\emph{Second step: Getting rid of the characteristic function.}
We claim that, if $R$ is large enough, then for any $\bm$,
$\ii$, $\bell$ as in the right-hand-side of (\ref{klqjsdfml})
  \begin{multline}
  \label{multipl}
  \norm{(\brho_\bm 1_{O_\bm} 1_{T_\ii^n O_\ii}( \brho_\bell \cdot g^{(n)} \omega)\circ T_\ii^{-n})
  \circ \bphi_\bm}
  {H_p^{t,s}}^p
  \\ \leq
  \Cs (\Cs N^p)^{n/N} \norm{(\brho_\bm ( 1_{O_\bell} \brho_\bell \cdot g^{(n)} \omega)\circ T_\ii^{-n})
  \circ \bphi_\bm}
  {H_p^{t,s}}^p \, .
  \end{multline}
Note that $1_{T_\ii^n O_\ii}=1_{T_\ii^n O_\ii}\cdot
(1_{O_\bell}\circ T_\ii^{-n})$. Hence, to prove this
inequality, it is sufficient to show that the multiplications
by $1_{O_\bm}\circ \bphi_\bm$ and by $1_{T_\ii^n O_\ii}\circ
\bphi_\bm$ act boundedly on $H_p^{t,s}$, with norms bounded
respectively by $\Cs$ and $(\Cs N^p)^{n/N}$. We shall show the
latter, the former is similar. Let $\kappa = n/N$, we decompose
$\ii=(i_0,\dots,i_{n-1})$ into subsequences of length $N$, as
$(\ii_0,\dots,\ii_{\kappa-1})$. Then $1_{T_\ii^n O_\ii} =
\prod_{j=0}^{\kappa-1} 1_{O_{\ii_j}} \circ
T_{\ii_j\ii_{j+1}\dots \ii_{\kappa-1}}^{-(\kappa-j)N}$. Define
a set $P_j=T_{\ii_j\ii_{j+1}\dots
\ii_{\kappa-1}}^{(\kappa-j)N}(O_{\ii_j})$, it is therefore
sufficient to show that each multiplication by $1_{P_j} \circ
\bphi_\bm$ acts boundedly on $H_p^{t,s}$, with norm at most $\Cs
N^p$. Let us fix such a set $P=P_j$. Locally, its boundary is
contained in the images of the boundaries of the sets $O_i$
under iterates of the map $T$. Let $L>0$ be such that the
boundary of each $O_i$, $i\in I$, is made of at most $L$
hypersurfaces, it follows that the boundary of $P$ is made of
at most $LN$ hypersurfaces $Q_{h}$ (which are all uniformly
transverse to the stable cone).

We wish to use our transversality assumption to apply Lemma~
\ref{lem:multiplier}. Write $\bm=(i,j,m)$. Since the support of
$\brho_\bm \circ (\kappa_\bm^R)^{-1}=\rho_m$ is contained in
the ball $B(m,d)$, it is sufficient to prove the bounded
multiplier property for distributions supported in
$\phi_\bm^{-1}(B(m,d))$. In $B(m,d)$, the boundary of the set
$\kappa_\bm^R(P)$ is contained in $\bigcup \kappa_\bm^R(Q_h)$.
If $R$ is large enough, all the hypersurfaces
$\kappa_\bm^R(Q_h)$ look like hyperplanes in $\R^d$.

\medskip

We will need the following easy geometrical lemma.
\begin{lemma}
For any $\delta>0$, $\delta'>0$ and $M>0$, there exists
$\epsilon>0$ satisfying the following property. Consider $M$
hyperplanes $H_1,\dots, H_M$ in $\R^d$, such that every $H_j$
contains a $d_u$-dimensional subspace $E_j$ making an angle at
least $\delta$ with $\{0\}\times \R^{d_s}$. Then
\begin{itemize}
\item For any unit vector $f\in \R^d$, there exists a
vector $e\in \R^d$ with $|e-f| \leq \delta'$ making an
angle at least $\epsilon$ with every $H_j$.
\item For any unit vector $f\in \{0\}\times \R^{d_s}$, there exists a
vector $e\in \{0\}\times\R^{d_s}$ with $|e-f| \leq \delta'$
making an angle at least $\epsilon$ with every $H_j$.
\end{itemize}
\end{lemma}
The first point is proved by arguing that the measure of the
$\epsilon$--neighborhood of $H_j$ in the ball $B(f,\delta')$
tends to $0$ when $\epsilon$ tends to $0$. Therefore, if
$\epsilon$ is small enough, there exists a vector $e$ in
$B(f,\delta')$ avoiding all those neighborhoods, hence
satisfying the required conclusion. For the second point, we
obtain in the same way a vector $e\in \{0\}\times\R^{d_s}$ with
$|e-f| \leq \delta'$ which is $\epsilon$-transverse to $H_j
\cap( \{0\}\times \R^{d_s})$ for $1\leq j\leq M$. Since $E_j$
in the assumptions is uniformly transverse to $e$, the result
follows.

\medskip
Let us fix $\delta'>0$ so that any family $e_1,\dots,e_d$ which
is $\delta'$-close to the canonical orthonormal basis
$(f_1,\dots,f_d)$ of $\R^d$ is still a basis, and the matrices
of the coordinate changes are bounded by a constant $\Cs$.

The pullback of every hypersurface $\kappa_\bm^R(Q_h)$ under
the differential $D\phi_\bm(m)$ is very close to an hyperplane
in $\R^d$. Applying the lemma with $M=LN$, we therefore obtain
vectors $e_1,\dots,e_d$ which are $\delta'$-close to an
orthonormal basis of $\R^d$,  such that $e_{d_u+1},\dots,e_d$
form a basis of $\{0\}\times \R^{d_s}$, and which make
everywhere an angle at least $\epsilon$ with the hypersurfaces
$\kappa_\bm^R(Q_h)$, for some $\epsilon>0$ depending solely on
$N$.

Consider now a straight line directed by one of the vectors
$e_l$. Its image under $\phi_\bm$ is not anymore a straight
line. However, if $\phi_\bm$ is very close to a linear map
(which is true if $C_1$ is large enough), then it will almost
be a straight line. In particular, its direction will deviate
by at most $\epsilon/2$, hence it will be transverse to the
hypersurface $\kappa_\bm^R(Q_h)$, and it will intersect it in
at most one point.

We have proved that, if $C_1$ is large enough, then any line
$S$ directed by one of the vectors $e_l$ intersects each
boundary hypersurface of $\bphi_\bm^{-1}(P)$ in at most one
point. Since $\bphi_\bm^{-1}(P)$ has at most $NL$ boundary
hypersurfaces, $S$ intersects $\bphi_\bm^{-1}(P)$ along at most
$NL$ connected components. Therefore,
Lemma~\ref{lem:multiplier} (together with our assumption that
$1/p-1<s<0<t<1/p$) implies that the multiplication by the
characteristic function of this set acts boundedly on
$H_p^{t,s}$, with a norm bounded by $\Cs NL$. This proves
~\eqref{multipl}.

Combining~\eqref{multipl} with~\eqref{klqjsdfml}, we get
  \begin{equation}
  \label{2step}
  \norm{\LL_g^n \omega}{\Phi}^p
  \leqc \Cs (\Cs N^p)^{n/N} (D_n^e)^{p-1} \sum_{\bm, \ii, \bell}
  \norm{(\brho_\bm( 1_{O_\bell}\brho_\bell \cdot g^{(n)} \omega)\circ T_\ii^{-n})
  \circ \bphi_\bm}
  {H_p^{t,s}}^p \, .
  \end{equation}

\emph{Third step: Using the composition lemma.} The right hand
side of \eqref{2step} involves a sum over $\bell$ and $\bm$,
and has therefore too many terms. In this step, we shall use
Lemma~\ref{lemcompose}, to pull the charts $\Phi_\bm$ back at
time $-n$, and  glue some of the pulled-back charts together to
get rid of the summation over $\bm$.

Let us partition $\ZZ(R)$ into finitely many subsets
$\ZZ^1,\dots,\ZZ^E$ such that $\ZZ^e$ is included in one of the
sets $\ZZ_{i,j}(R)$, and $|m-m'|\geq C(C_0)$ whenever
$(i,j,m)\not=(i,j,m')\in \ZZ^e$, where $C(C_0)$ is the constant
$C$ constructed in Lemma~\ref{lemcompose} (it only depends on
$C_0$). The number $E$ may be chosen independently of $n$.

We shall prove the following: For any $\bell \in \ZZ(R)$, any
$\ii\in I^n$ (such that the support of $\brho_{\bell}$ is
included in $\tilde O_\ii$ and $O_\bell=O_{i_0}$) and any
$1\leq e\leq E$, there exists an admissible chart
$\bphi'=\bphi'_{\bell,\ii,e} \in \FF(\bell)$ such that
  \begin{equation}
  \label{eqmainstep}
  \sum_{\bm\in \ZZ^e} \norm{(\brho_\bm( 1_{O_\bell}\brho_\bell \cdot g^{(n)} \omega)\circ T_\ii^{-n})
  \circ \bphi_\bm}
  {H_p^{t,s}}^p
  \leqc \Cs \chi_n \norm{(1_{O_\bell} \brho_\bell \cdot\omega) \circ \bphi'_{\bell,\ii,e}}{H_p^{t,s}}^p\, ,
  \end{equation}
where
  \begin{equation}
  \chi_n=\norm{ |\det DT^n|
  \max(\lambda_{u,n}^{-t}, \lambda_{s,n}^{-(s+t)})^p
  |g^{(n)}|^p}{L^\infty}.
  \end{equation}
As always, the sum on the left hand side of \eqref{eqmainstep}
is restricted to those values of $\bm$ such that the support of
$\brho_\bm$ is included in $T_\ii \tilde O_\ii$

Let us fix $\bell$, $\ii$ and $e$ as above, until the end of
the proof of \eqref{eqmainstep}. All the objects we shall now
introduce shall depend on these choices, although we shall not
make this dependence explicit to simplify the notations. Let
$i,j$ be such that $\ZZ^e\subset \ZZ_{i,j}(R)$, and let
$\JJ=\{m\st (i,j,m)\in \ZZ^e$\}. Since the points in $\JJ$ are
distant of at least $C(C_0)$, Lemma~\ref{lemcompose} will
apply.

Increasing $R$, we can ensure that the map
  \begin{equation*}
  \TT\coloneqq \kappa_{i,j}^R\circ T^n_\ii \circ
  (\kappa_\bell^R)^{-1}
  \end{equation*}
is arbitrarily close to its differential $M=D\TT(\ell)$ at
$\ell\coloneqq \kappa_{\bm'}(q_{\bm'})$, i.e., the map
$(\TT^{-1} [\cdot + \TT(\ell)]-\ell) \circ M$ is close to the
identity in $C^{1+\alpha}$, say on the ball $B(0,2d)$.
Moreover, recalling the notation from the beginning of Section
~ \ref{invv}, the matrix $M$ sends $\CC_\bell$ to $\CC_{i,j}$
compactly, and
  \begin{equation}
  \Cs \geq \lambda_u(M,\CC_\bell,\CC_{i,j})/ \lambda^{(n)}_u( q_\bell) \geq \Cs^{-1} \, ,
  \end{equation}
with similar inequalities for $\lambda_s$ and $\Lambda_u$.
Since $T$ is uniformly hyperbolic and satisfies the bunching
conditions \eqref{bunch} and \eqref{bunch2}, we can ensure by
taking $n$ large enough that $M$ satisfies   \eqref{suffhyp}
for the constant $\epsilon=\epsilon(C_0,C_1)$ constructed in
Lemma~\ref{lemcompose}.
%We wish to apply Lemma~\ref{lemcompose}, but to do so we need to
%extend $\TT$ to a diffeomorphism of $\R^d$, as follows.
By Lemma~\ref{lemextend}, since the map $(\TT^{-1} [\cdot +
\TT(\ell)]-\ell) \circ M$ is close to the identity on
$B(0,2d)$, there exists a diffeomorphism of $\R^d$, close to
the identity and coinciding with this map on $B(0,d)$.
Composing with $M^{-1}$ and translating, we obtain an extension
of $\TT^{-1}$, coinciding with $\TT^{-1}$ on $B(\TT(\ell),d)$,
and still denoted by $\TT^{-1}$. Taking  $R$ large enough, we
can ensure that $\norm{ (\TT^{-1} [\cdot + \TT(\ell)]-\ell)
\circ M -\id }{C^{1+\alpha}} \leq \epsilon(C_0,C_1)$.

We may therefore apply Lemma \ref{lemcompose} (see also Remark~
\ref{shorter}), and we obtain a block diagonal matrix $D$, a
chart $\phi'$ around $\ell$, and diffeomorphisms $\Psi_m$,
$\Psi$ such that, for any $m$ in the set $\JJ'$ of those
elements in $\JJ$ for which $\brho_\bm\cdot\brho_\bell\circ
T_\ii^{-n}$ is nonzero,
  \begin{equation}
  \label{newchart}
  \TT^{-1}\circ \phi_\bm=\phi' \circ \Psi \circ D^{-1}\circ \Psi_m
  \end{equation}
on the set where $(\brho_\bm\cdot\brho_\bell\circ T_\ii^n)
\circ \Phi_\bm$ is nonzero.

Writing $\omega'= (1_{O_\bell}\brho_\bell \cdot g^{(n)}
\omega)\circ (\kappa_{\bell}^R)^{-1}$, we have (recall that
$(i,j)$ is fixed so that $\ZZ^e\subset \ZZ_{i,j}(R)$)
  \begin{align*}
  \sum_{\bm\in \ZZ^e} \bigl\|(\brho_\bm( 1_{O_\bell}\brho_\bell \cdot g^{(n)} &\omega)\circ T_\ii^{-n})
  \circ \bphi_\bm\bigr\|_{H_p^{t,s}}^p\\ &
  =
  \sum_{m\in \JJ'} \norm{\rho_m\circ \phi_{i,j,m} \cdot \omega'\circ \TT^{-1} \circ \phi_{i,j,m}}
  {H_p^{t,s}}^p
%  \\&
%  =\sum_{m\in \JJ'} \norm{\rho_m\circ \phi_{i,j,m} \cdot
%  \omega'\circ \phi' \circ \Psi \circ D^{-1}\circ \Psi_m}
%  {H_p^{t,s}}^p
  \\&
  =
  \sum_{m\in \JJ'} \norm{ (\rho_m\circ \phi_{i,j,m}\circ \Psi_m^{-1} \cdot
  \omega'\circ \phi' \circ \Psi \circ D^{-1})\circ \Psi_m}{H_p^{t,s}}^p\, .
  \end{align*}
Using the notations and results of Lemma ~\ref{lemcompose}, the
terms in this last equation are of the form $v\circ \Psi_m$,
where $v$ is a distribution supported in
$\Psi_m(\phi_{i,j,m}^{-1}(B(m,d)))\subset B(\Pi m,
C_0^{1/2}/2)$. Since the range of $\Psi_m$ contains $B(\Pi m,
C_0^{1/2})$, and since $\alpha t + |s|< \alpha$,
Lemma ~ \ref{lemcomposeD1alpha} gives
$\norm{v\circ \Psi_m}{H_p^{t,s}}\leq \Cs \norm{v}{H_p^{t,s}}$,
yielding a bound
  \begin{equation*}
  \Cs \sum_{m\in \JJ'} \norm{ \rho_m\circ \phi_{i,j,m}\circ \Psi_m^{-1} \cdot
  \omega'\circ \phi' \circ \Psi \circ D^{-1}}{H_p^{t,s}}^p\, .
  \end{equation*}
The functions $\rho_m \circ \phi_{i,j,m}\circ \Psi_m^{-1}$ have
a bounded $C^1$ norm, and are supported in the balls $B(\Pi
m,C_0^{1/2}/2)$, whose centers are distant by at least $C_0$,
by Lemma~\ref{lemcompose}~(a). Therefore, by
Lemma~\ref{lem:localization}, the last expression is bounded by
  \begin{equation*}
  \Cs \norm{\omega'\circ \phi' \circ \Psi \circ D^{-1}}{H_p^{t,s}}^p\, .
  \end{equation*}
We may apply Lemma~\ref{CompositionDure} to the composition
with $D^{-1}$ (to obtain an improvement in the $H_p^{t,s}$
norm, up to compact terms). Since $\omega'$ is supported in
$B(\ell, C_0^{1/2}/2)$ while the range of $\Psi$ contains
$B(\ell, C_0^{1/2})$ (by Lemma~\ref{lemcompose}),
Lemma~\ref{lemcomposeD1alpha} implies that the composition with
$\Psi$  is bounded. Summing up, we obtain
  \begin{multline}
  \label{qklsjflmqsfd}
  \sum_{\bm\in \ZZ^e} \norm{(\brho_\bm( 1_{O_\bell} \brho_\bell \cdot g^{(n)} \omega)\circ T_\ii^{-n})
  \circ \bphi_\bm}
  {H_p^{t,s}}^p
  \\
  \leqc \Cs \chi^{(0)}_n(q_\bell)\norm{
	(1_{O_\bell} \brho_\bell \cdot g^{(n)} \omega)
	\circ (\kappa_\bell^R)^{-1}\circ \phi'
  }{H_p^{t,s}}^p \, ,
  \end{multline}
where
  \begin{equation*}
  \chi^{(0)}_n(q_\bell)=(|\det DT^n|
  \max(\lambda_{u,n}^{-t}, \lambda_{s,n}^{-(s+t)})^p)(q_\bell) \, .
  \end{equation*}
Let $\nu>0$. Since  $(\kappa_\bell^R)^{-1}$ contracts by a
factor $1/R$, we can ensure by increasing $R$ that the
$C^{\regg}$ norm of $g^{(n)}\circ (\kappa_\bell^R)^{-1}$ on
$B(\ell,d)$ is bounded by $\Cs |g^{(n)}(q_\bell)|+\nu$ (recall
that, by assumption, $g$ belongs to $C^\regg$ for some
$\regg>t+|s|$). The term $\nu$ here is necessary when $|g|$ is
not bounded away from $0$. Choosing $\nu$ small enough, we can
ensure that $(|g^{(n)}(q_\bell)|+\nu)^p \chi^{(0)}_n(q_\bell)
\leq 2 \chi_n$. Hence, \eqref{qklsjflmqsfd} and
Lemma~\ref{Leib} yield
  \begin{multline*}
  \sum_{\bm\in \ZZ^e} \norm{
	(\brho_\bm( 1_{O_\bell}\brho_\bell \cdot g^{(n)} \omega)
	\circ T_\ii^{-n}) \circ \bphi_\bm
  }{H_p^{t,s}}^p
  \\
  \leqc
  \Cs \chi_n\norm{
	(1_{O_\bell}\brho_\bell \omega)\circ (\kappa_\bell^R)^{-1}
	\circ \phi'
  }{H_p^{t,s}}^p \, .
  \end{multline*}

This concludes the proof of \eqref{eqmainstep}. Summing over
all possible values of $\bell$, $\ii$ and, $e$, we obtain
  \begin{equation}\label{concl}
  \norm{\LL_g^n \omega}{\Phi}^p
  \leqc \Cs (\Cs N^p)^{n/N} (D_n^e)^{p-1}\chi_n\sum_{\bell, \ii} \sum_{e=1}^E
  \norm{(1_{O_\bell} \brho_\bell \omega)\circ \bphi'_{\bell,\ii,e}}{H_p^{t,s}}^p \, .
  \end{equation}

\emph{Fourth step: Conclusion.} The right hand side of
(\ref{concl}) is essentially of the form
$\norm{\omega}{\Phi'}^p$ for some family of admissible charts
$\Phi'$, with the difference that to a point $q_\bell$ for
$\bell\in \ZZ(R)$ correspond several admissible charts around
it. Since $E$ is independent of $n$, the number of those charts
around $q_\bell$ is at most $\Cs \cdot \Card\{\ii \st \tilde
O_\ii \cap A(\bell)\not=\emptyset\}$. If $R$ is large enough,
we can ensure that this quantity is bounded by the intersection
multiplicity of the sets $\tilde O_\ii$, which is at most
$D_n^b$ by construction. Therefore, we obtain
  \begin{equation*}
  \norm{\LL_g^n \omega}{\Phi}^p
  \leqc \Cs (\Cs N^p)^{n/N} (D_n^e)^{p-1} D_n^b \chi_n \norm{\omega}{\HHH^{t,s}_p(R, C_0, C_1)}^p \, .
  \qedhere
  \end{equation*}
\end{proof}

\begin{proof}[Proof of Proposition~\ref{propinvar}]
Remark~\ref{remOKgrand} shows that the charts $\phi'$ we
constructed in the third step of the proof of Lemma~
\ref{maintechlemma} can be defined on larger balls, and with
better bounds. In particular, these new charts will be
admissible when looked at a scale $R'$ and with a smoothness
constant $C'_1$, for any $R/2 \leq R'\leq 2R$ and $C_1/2 \leq
C'_1 \leq 2C_1$. The proof of Lemma~ \ref{maintechlemma}
therefore gives the following statement:

For any large enough $C_1$ (say $C_1\geq C_1^{(0)}$), for any
large enough $n$ (say $n\geq n^{(0)}(C_1)$), and for any large
enough $R$ (say $R\geq R^{(0)}(n,C_1)$), then for any $R'\in
[R/2, 2R]$ and $C'_1 \in [C_1/2, 2C_1]$, the operator $\LL_g^n$
maps continuously $\HHH_p^{t,s}(R, C_0, C_1)$ to
$\HHH_p^{t,s}(R', C_0, C'_1)$.

It follows that, for any $C_1 \geq C^{(0)}_1$ and $R\geq
R^{(0)}(n^{(0)}(C_1), C_1)$, and for any $C'_1\geq C^{(0)}_1$
and $R'\geq R^{(0)}(n^{(0)}(C'_1), C'_1)$, there exists an
integer $n$ such that $\LL_g^n$ maps $\HHH_p^{t,s}(R, C_0,C_1)$
to $\HHH_p^{t,s}(R',C_0,C'_1)$. Moreover, if $n'$ is large
enough, $\LL_g^{n'}$ maps $\HHH_p^{t,s}(R, C_0,C_1)$ to itself.
Writing a large enough integer $N$ as $n'+n$, we get that
$\LL_g^N$ maps $\HHH_p^{t,s}(R, C_0,C_1)$ to
$\HHH_p^{t,s}(R',C_0,C'_1)$.
\end{proof}

%%%%%%%%%%%%%%%%%%%%%%%%%%%%%

\appendix

\section{Calculus for some classes of maps}
\label{apA}

This appendix groups some straightforward results about classes of maps
$\DD$ and $\KK$ which appear in the proofs of Lemma ~\ref{lempropphi}
and  Lemmas~\ref{lemcompose}--\ref{lemcomposeQ}
(together with an easy result, which is useful
for the proof of Lemma ~ \ref{maintechlemma}).

\label{app}

\subsection{The class \texorpdfstring{$\DD$}{D}}

\label{subsectDD}

For $\Cs>0$, let us denote by $\DD(\Cs)$ the class of $C^1$
maps $f$ defined on an open subset of $\R^d$, satisfying
  \begin{equation}
  \label{defD}
  \Cs^{-1} |z-z'|\leq |f(z)-f(z') | \leq \Cs |z-z'|,
  \end{equation}
for any $z,z'$ in the domain of definition of $f$. It follows
that $f$ is a local diffeomorphism, and that $\nor{Df}\leq
\Cs$, $\nor{(Df)^{-1}}\leq \Cs$.

\begin{lemma}
\label{belongsDD}
Assume that $f(x,y)=(g(x,y),y)$ is defined on a set $A_1\times
A_2$ where $A_1$ and $A_2$ are convex, that $|Dg|\leq C$, and
that $|g(x,y)-g(x',y)|\geq C^{-1} |x-x'|$ for some $C>0$. Then
$f\in \DD(\Cs)$, for some constant $\Cs$ depending only on $C$.
\end{lemma}
\begin{proof}
Since the second coordinate of $f(x,y)$ is equal to $y$, while
the derivative of $f$ is bounded by $C$, we have
  \begin{equation}
  \label{mqlsdkjf}
  |y-y'|\leq |f(x,y)-f(x',y')|\leq \Cs^1(|x-x'|+|y-y'|) \, ,
  \end{equation}
for some constant $\Cs^1$ depending only on $C$. This proves
the (trivial) upper bound in \eqref{defD}.

Consider now two points $z=(x,y),z'=(x',y')\in A_1 \times A_2$.
If $|y-y'|\geq C^{-1} |x-x'|/(2\Cs^1)$, we have in particular
$|y-y'|\geq \epsilons^2 |z-z'|$ for some $\epsilons^2$, and we
get from \eqref{mqlsdkjf} that $|f(z)-f(z')|\geq \epsilons^2
|z-z'|$. Otherwise,
  \begin{align*}
  |f(x,y)-f(x',y')|&\geq |f(x,y)-f(x',y)|
  -|f(x',y)-f(x',y')|
  \\&
  \geq C^{-1} |x-x'| - \Cs^1 |y-y'|
  \geq C^{-1} |x-x'|/2 \, .
  \end{align*}
This proves the lower bound in \eqref{defD} in all cases.
\end{proof}

\begin{lemma}
\label{DDbigimage}
Let $f\in \DD(\Cs)$, and assume that the domain of definition
of $f$ contains a ball $B(z,r)$. Then the range of $f$ contains
$B(f(z), r/\Cs)$.
\end{lemma}
\begin{proof}
Let $r'<r$, and consider $A= f (B(z,r')) \cap B(f(z), r'/\Cs)$.
Since $f$ is a local diffeomorphism, this is an open subset of
$B(f(z), r'/\Cs)$. Moreover, if $|z'-z|=r'$, then $f(z')$ does
not belong to $B(f(z), r'/\Cs)$, since $|f(z')-f(z)|\geq
|z'-z|/\Cs=r'/\Cs$. Therefore, $A$ is also equal to
$f(\overline{B(z',r)}) \cap B(f(z), r'/\Cs)$. This is a closed
subset of $B(f(z), r'/\Cs)$, since $f(\overline{B(z',r)})$ is
compact.

Finally, $A$ is open and closed in $B(f(z), r'/\Cs)$. By
connectedness, it coincides with this whole ball. In
particular, the range of $f$ contains $B(f(z), r'/\Cs)$.
Letting $r'$ tend to $r$, we conclude the proof.
\end{proof}

Let us also mention the following easy result, which is useful
for the proof of Lemma ~ \ref{maintechlemma}.
\begin{lemma}
\label{lemextend}
Let $\alpha \in (0,1]$ and let $f:B(0,1)\to \R^d$ be a
diffeomorphism such that $\norm{f-\id}{C^{1+\alpha}}$ is small
enough. Then there exists a diffeomorphism $\tilde f$ of
$\R^d$, coinciding with $f$ on $B(0,1/2)$, and such that
$\norm{\tilde f-\id}{C^{1+\alpha}}\leq \Cs
\norm{f-\id}{C^{1+\alpha}}$, for some universal constant $\Cs$
depending only on the dimension $d$.
\end{lemma}
\begin{proof}
Let us write, for $z\in B(0,1)$, $f(z)=z+\psi(z)$ with
$\norm{\psi}{C^{1+\alpha}}$ small. We may define the required
extension $\tilde f$ of $f$ by $\tilde f(z)=z+\gamma(z)\psi(z)$
where $\gamma$ is $C^\infty$, equal to $1$ on $B(0,1/2)$ and
supported in $B(0,1)$. If $\norm{\psi}{C^{1+\alpha}}$ is small
enough, then $\langle D\tilde f(z) v,v\rangle \geq |v|^2/2$ for
any point $z$ and any vector $v$. Integrating this inequality,
it follows that $|\tilde f(z)-\tilde f(z')|\geq |z-z'|/2$.
Therefore, $\tilde f$ belongs to the class $\DD(2)$. By
Lemma~\ref{DDbigimage}, it is surjective, hence it is a
diffeomorphism of $\R^d$.
\end{proof}

\subsection{The class \texorpdfstring{$\KK$}{K}}

\label{subsecK}

Let us fix $\alpha\in (0,1]$ and $\beta\in (0,\alpha)$. We
denote by $\KK=\KK^{\alpha,\beta}$ the class of matrix-valued
functions $K$ on $\R^d$ such that, for some constant $C$ and
for all $x,x'\in \R^{d_u}$ and all $y,y'\in \R^{d_s}$,
  \begin{align}
  \label{KK_sup}&|K(x,y)|\leq C\, ,\\
  \label{KK_x}&|K(x,y)-K(x',y)|\leq C |x-x'|^\beta \, ,\\
  \label{KK_y}&|K(x,y)-K(x,y')|\leq C |y-y'|^\alpha\, ,\\
  \label{KK_xy}&|K(x,y)-K(x',y)-K(x,y')+K(x',y')| \leq C |x-x'|^\beta |y-y'|^{\alpha-\beta}\, .
  \end{align}
  If $K \in \KK$, we write $\nor{K}$ for the
the smallest $C$ satisfying the inequalities above. We write
$\KK(C)$ for the functions in $\KK$ with $\nor{K}\leq C$.

For instance, any bounded $\alpha$-H\"{o}lder continuous function
$K$ belongs to $\KK$ (to obtain \eqref{KK_xy}, treat separately
the cases $|x-x'|\leq |y-y'|$ and $|x-x'|>|y-y'|$). Note also
that if K is $C^1$ then the left-hand-side of \eqref{KK_xy} can
be rewritten as $|\int_y^{y'} \partial_{y'-y}
K(x,t)-\partial_{y'-y} K(x',t)\, dt|$, i.e., it is a
finite-difference-type expression for $\partial_x \partial_y
K$.

\begin{proposition}
\label{KKinvmult}
A function in $\KK$ satisfies
  \begin{equation}
  \label{KK_holder}
  |K(x,y)-K(x',y')|\leq 3\nor{K} ( |x-x'|+ |y-y'|)^\beta.
  \end{equation}

If $K,K'\in \KK$, then $K+K'\in \KK$, with $\nor{K+K'}\leq
\nor{K}+\nor{K'}$. Moreover, $KK' \in \KK$, with $\nor{KK'}\leq
6 \nor{K}\nor{K'}$. Finally, if $K$ is everywhere invertible
and $|K^{-1}| \leq h$ for some finite number $h$, then
$K^{-1}\in \KK$ and $\nor{K^{-1}} \leq 5\max(1,h^3)
\max(1,\nor{K}^3)$.
\end{proposition}
\begin{proof}
Notice first that we have
  \begin{equation}
  \label{KK_y2}
  |K(x,y)-K(x,y')|\leq 2\nor{K} |y-y'|^{\alpha-\beta}\, .
  \end{equation}
Indeed, this follows from \eqref{KK_y} if $|y-y'|\leq 1$, and
from \eqref{KK_sup} if $|y-y'|>1$.  This inequality also holds
if $|y-y'|^{\alpha-\beta}$ is replaced with $|y-y'|^\beta$
(with the same proof). Therefore, by \eqref{KK_x},
  \begin{multline*}
  |K(x,y)-K(x',y')|
  \leq |K(x,y)-K(x',y)|+|K(x',y)-K(x',y')|
  \\
  \leq \nor{K} |x-x'|^\beta + 2 \nor{K} |y-y'|^\beta
  \leq 3 \nor{K} \max (|x-x'|, |y-y'|)^\beta \, .
  \end{multline*}
\eqref{KK_holder} follows.

Consider now $K,K'\in \KK$. It is trivial that $\nor{K+K'}\leq
\nor{K}+\nor{K'}$. We turn to $KK'$. Let us write $a,b,c,d$ for
$K(x,y), K(x',y), K(x,y'), K(x',y')$. Similarly, we use
$a',b',c',d'$ for $K'$. The inequality \eqref{KK_sup} for $KK'$
is trivial, \eqref{KK_x} follows from the equality
$aa'-bb'=a(a'-b')+(a-b)b'$, and \eqref{KK_y} is similar. For
\eqref{KK_xy}, we use the identity
  \begin{equation*}
  aa'-bb'-cc'+dd'=c(a'-b'-c'+d')+(a-b-c+d) d'+(a-c)(a'-b') + (a-b)(b'-d'),
  \end{equation*}
and the bounds for $a-c$, $a'-b'$, $a-b$ and $b'-d'$ given by
\eqref{KK_x} and \eqref{KK_y2}. This concludes the proof for
$KK'$.

Finally, assume $|K^{-1}|\leq h$. Then \eqref{KK_sup} holds for
$K^{-1}$. Moreover, \eqref{KK_x} follows from the equality
$|a^{-1}-b^{-1}|=|a^{-1}(b-a)b^{-1}| \leq h^2|a-b|$.
\eqref{KK_y} is similar. For \eqref{KK_xy}, we use the identity
  \begin{multline*}
  a^{-1}-b^{-1}-c^{-1}+d^{-1}=a^{-1}(b+c-a-d) b^{-1}
  \\
  + a^{-1}(c-a) c^{-1}(d-c)b^{-1}
  + c^{-1} (d-c) b^{-1} (d-b) d^{-1}\, ,
  \end{multline*}
and the bounds \eqref{KK_x} and \eqref{KK_y2}.
\end{proof}

We recall that the subsets $\{x\}\times \R^{d_s}$ of $\R^d$ are
called ``stable leaves'' of $\R^d$ in this article.

\begin{proposition}
\label{KK_invcomp}
Let $\Psi:\R^d \to \R^d$ send stable leaves to stable leaves,
and assume that its best Lipschitz constant $L$ is finite.
Then, for $K\in \KK$, the function $K\circ \Psi$ also belongs
to $\KK$, and $\nor{K\circ \Psi} \leq 3 \max(1,L) \nor{K}$.
\end{proposition}
\begin{proof}
The inequality \eqref{KK_sup} is trivial for $K\circ \Psi$. For
\eqref{KK_x}, we write using \eqref{KK_holder}
  \begin{multline*}
  |K\circ \Psi(x,y) - K\circ \Psi(x',y)|
  \leq 3\nor{K} d (\Psi(x,y), \Psi(x',y))^\beta
  \\
  \leq 3 \nor{K} L^\beta d((x,y),(x',y))^\beta
  \leq 3 \nor{K} \max(1,L) |x-x'|^\beta.
  \end{multline*}
\eqref{KK_y} for $K\circ \Psi$ follows from \eqref{KK_y} for
$K$ and from the fact that $\Psi$ sends stable leaves to stable
leaves and is Lipschitz continuous.

We turn to \eqref{KK_xy}. We write $\Psi(x,y)=(x_1,y_1)$,
$\Psi(x,y')=(x_1, y_1')$, $\Psi(x',y)=(x_2,y_2)$ and
$\Psi(x',y')=(x_2,y_2')$.

Assume first $|y-y'|\leq |x-x'|$. Then
  \begin{align*}
  |K(x_1,y_1)-K(x_1,y_1')&-K(x_2,y_2)+K(x_2, y_2')|
  \\&
  \leq |K(x_1,y_1)-K(x_1, y_1')|+|K(x_2,y_2)-K(x_2,y_2')|
  \\&
  \leq \nor{K} |y_1-y_1'|^\alpha+ \nor{K} |y_2-y_2'|^\alpha
  \leq 2 \nor{K} L^\alpha |y-y'|^\alpha \, .
  \end{align*}
Since $L^\alpha\leq \max(1,L)$ and $|y-y'|^\alpha \leq
|x-x'|^\beta |y-y'|^{\alpha-\beta}$, this is the desired
conclusion. Assume now $|x-x'|\leq |y-y'|$. Then
  \begin{equation}
  \label{4termes}
  \begin{split}
  \raisetag{-45pt}
  |K(x_1,y_1&)-K(x_1,y_1')-K(x_2,y_2)+K(x_2, y_2')|
  \\&
  \leq |K(x_1,y_1)-K(x_1,y_1')-K(x_2,y_1)+K(x_2, y_1')|
  \\& \quad\quad
  +| K(x_2,y_2)-K(x_2,y_1)| + |K(x_2,y'_2) - K(x_2,y'_1)| \, .
  \\&
  \leq \nor{K} |x_1-x_2|^\beta |y_1-y'_1|^{\alpha-\beta}+ \nor{K} |y_2-y_1|^\alpha
  +\nor{K} |y'_2-y'_1|^\alpha.
  \end{split}
  \end{equation}
Since $\Psi$ is Lipschitz continuous, we have $|x_1-x_2|\leq L
|x-x'|$ and $|y_1-y_1'|\leq L|y-y'|$. Moreover,
  \begin{equation*}
  |y_2-y_1|\leq d( (x_1,y_1), (x_2,y_2))=d( \Psi(x,y), \Psi(x', y))
  \leq L d((x,y),(x',y))=L |x-x'| \, .
  \end{equation*}
Since $|x-x'|\leq |y-y'|$, we obtain $|y_2-y_1|^\alpha \leq
L^\alpha |x-x'|^\alpha \leq \max(1,L) |x-x'|^\beta
|y-y'|^{\alpha-\beta}$. Moreover, $|y'_2-y'_1|$ satisfies a
similar inequality. Finally, \eqref{4termes} is bounded by $3
\nor{K} \max(1,L) |x-x'|^\beta |y-y'|^{\alpha-\beta}$. This
concludes the proof.
\end{proof}

\begin{remark}
\label{domdef}
If $A_1$ and $A_2$ are convex subsets of, respectively,
$\R^{d_u}$ and $\R^{d_s}$, we can define analogously a space
$\KK(C, A_1\times A_2)$ of matrix-valued functions defined on
$A_1\times A_2$ and satisfying \eqref{KK_sup}--\eqref{KK_xy}.
The previous results also hold for this space, with the same
proofs, up to the following small modification: In Proposition
~\ref{KK_invcomp}, if $K$ is defined on $A_1 \times A_2$,  we
need to require that $\Psi$ be defined on $A'_1 \times A'_2$
with $\Psi(A'_1 \times A'_2) \subset A_1 \times A_2$.
Successive applications of the proposition
 in the proof of Lemma ~\ref{lemcomposeQ} will require stronger conditions.
The careful reader is invited to check that this does not cause
any problems in the proof of Lemma ~\ref{lemcomposeQ}.
\end{remark}

\section{Convex transversality}

\label{apB}

We prove the claims made after Definition~\ref{defconvtr}.
Consider the cone $\{(x,y)\st |x|\leq |Ay|\}$ for some nonzero
linear map $A$. We should prove that, for any vector space $E$
so that $C \cap E=\{0\}$, the set $C \cap (E+w)$ is convex for
all $w\in \real^d$.

\begin{proof}
Pick $z_1$, $z_2$ in $C \cap (E+w)$, we want to show that the
segment $[z_1,z_2]$ is included in $C \cap (E+w)$. The line
directed by $z_0 :=z_2-z_1$ is contained in $E$, so
$z_0=(x_0,y_0) \notin C$, i.e., $|Ay_0|^2 < |x_0|^2$.

Let $D=\{ (x_1 + tx_0, y_1+t y_0)\, , t\in [0,1]\}$ be the
segment between $z_1=(x_1,y_1)$ and $z_2$. The leading
coefficient of the polynomial $\Phi(t) := |x(t)|^2 - |Ay(t)|^2
= |x_1 + tx_0|^2 - |Ay_1+t Ay_0|^2$ is  $|x_0|^2 - |Ay_0|^2 >
0$. Therefore, the set $\{ t \st \Phi(t) \leq 0\}$ is convex,
i.e., $C \cap D$ is convex. Since $z_1$ and $z_2$ belong to $C
\cap D$,  we find $D \subset C \cap D$, as desired.
\end{proof}

\section{A more general setting}
\label{app_general}

For the sake of simplicity, we have formulated all our results
for the transfer operator associated to a map. However, it
turns out that the same proof applies to a wider class of
operators, which would formally correspond to the transfer
operators of multivalued maps. This kind of generalized
transfer operators has been studied in one dimension in
\cite{baru_sharp}.

In our main result, we also assumed that the continuity domains
of the stable and unstable cones coincide with the domains of definition
of the branches of the map. Although this assumption is quite
natural, it plays no role in the proof, and can therefore be
removed.

These remarks lead to the following general setting, which turns
out to be useful for many applications (see the comments after
the statement of Theorem~ \ref{main_extended}). We
consider finitely many subsets $(O_i)_{i\in I}$ of a manifold
$X$ (that may not be disjoint, and may not cover everything),
with compact closure, and maps $T_i : O_i \to X$ such that
$T_i$ admits a $C^{1+\alpha}$ extension to a neighborhood of
$\overline{O_i}$, for some $\alpha\in (0,1]$. Consider also
finitely many disjoint open subsets $(\Pi_e)_{e\in E}$,
covering almost all $X$, and assume that on each of these
subsets are given two convexly transverse cones
$\CC^{(u)}_e(q)$ and $\CC^{(s)}_e(q)$ in the tangent space
$\TT_q X$, depending continuously on $q\in \Pi_e$ and which
extend continuously up to the boundary of $\Pi_e$.

The following transversality conditions are needed. For the
domains $\Pi_e$, we require transversality with the stable
cones at time $0$: the boundary of each set $\Pi_e$ is a finite
union of hypersurfaces $P_{e,k}$ such that, for all $q\in
P_{e,k}$, the tangent space $\TT_q P_{e, k}$ is transverse to
$\CC_e^{(s)}(q)$. For the domains $O_i$, we only require
transversality at time $1$ (i.e., in the image): the boundary
of each set $O_i$ is a finite union of hypersurfaces $K_{i,k}$
such that, for all $q\in K_{i,k}$ and all $e$ such that
$T_i(q)\in \overline{\Pi_e}$, the cone $\CC_e^{(s)}(T_i (q))$
is transverse to $\TT_{T_i(q)} (T_i (K_{i,k}))$.

We will need hyperbolicity: for each  $q\in \overline{ O_i}\cap
\overline{\Pi_e} \cap T_i ^{-1}(\overline{\Pi_{e'}})$, then
$DT_i(q) \CC^{(u)}_e(q) \subset \CC^{(u)}_{e'}(T_i (q))$, and
there exists $\lambda_{i,u}(q)>1$ (independent of $e, e'$) such
that
  \begin{equation*}
  |DT_i (q) v|\ge \lambda_{i,u}(q) |v| \, ,\forall v\in \CC^{(u)}_e(q) \, .
  \end{equation*}
Moreover, for each $q\in \overline{ O_i}\cap \overline{\Pi_e}
\cap T_i ^{-1}(\overline{\Pi_{e'}})$, then $DT_i^{-1}(T_i(q))
\CC^{(s)}_{e'}(T_i(q)) \subset \CC^{(s)}_e(q)$, and there
exists $\lambda_{i,s}(q)\in (0,1)$ (independent of $e, e'$)
such that
  \begin{equation*}
  |DT_i^{-1} (T_i(q))v|\ge \lambda_{i,s}^{-1}(q) |v|\, ,
  \forall v\in \CC^{(s)}_{e'}(T_i(q)) \, .
  \end{equation*}

For $\ii \in I^n$, we define $O_\ii$ and $T_\ii$ as in
Paragraph \ref{setting}, and we also define the complexities
$D_n^b$ and $D_n^e$ at the beginning and at the end, and the
best expansion and contraction coefficients
$\lambda_{\ii,u}(q)$ and $\lambda_{\ii,s}(q)$. In this
generalized setting, we obtain the following variant of Theorem
\ref{MainTheorem}:

\begin{theorem}
\label{main_extended}
Let $T$ satisfy the piecewise hyperbolicity and transversality
conditions just given. Assume that the bunching
conditions \eqref{bunch} and \eqref{bunch2} are satisfied for
some parameters $\alpha,\beta$, and consider parameters $p,s,t$
satisfying \eqref{pstcond}. Then there exists a space $\HHH$ of
distributions on $X$ with the following properties.

Consider functions $(g_i)_{i\in I}$, defined on $O_i$ and
admitting a $C^\gamma$ extension to its closure for some
$\gamma > t+|s|$. Define an operator $(\LL_g \omega)(q)=
\sum_{T_i(q')=q} g_i(q') \omega(q')$. Then this operator acts
on $\HHH$. Moreover, its essential spectral radius on $\HHH$ is
at most the limit when $n$ tends to infinity of
  \begin{equation*}
  (D_{n}^b)^{\frac{1}{pn}} \cdot (D_{n}^e)^{\frac{1}{n}\left(1-\frac{1}{p}\right)}
  \cdot \sup_{\ii=(i_0,\dots,i_{n-1})} \norm{g_\ii^{(n)}|\det DT_\ii^{n}|^{\frac{1}{p}}
  \max(\lambda_{\ii,u}^{-t},\lambda_{\ii,s}^{-(t-|s|)})}{L^\infty(O_\ii)}^{\frac{1}{n}} \, ,
  \end{equation*}
where we set $g_\ii^{(n)}(q)= \prod_{k=0}^{n-1}
g_{i_k}(T^k_{(i_0,\dots,i_{k-1})}(q))$, for $n \ge 1$.
\end{theorem}
In the case of a single-valued map, and when the sets $\Pi_e$
and $O_i$ coincide, this theorem reduces to Theorem
\ref{MainTheorem}. However, this extension is useful is many
cases. For instance, if there is a single cone field (i.e.,
$\Pi_1 = X$), then the transversality condition is only on the
images $T(O_i)$, it is therefore weaker than the condition in
Definition \ref{transs} (we already mentioned this fact and
its relevance for Sinai billiards in
Remark \ref{forbilliards}). Another interest of Theorem
\ref{main_extended} is that the class of operators studied
there is closed under time reversal. Indeed, for all functions
$\omega_1, \omega_2$, we have
  \begin{align}\label{addjoint}
  \int \omega_1 \LL_g \omega_2 \dLeb
  &
  = \sum_i \int_{T_i(O_i)} \omega_1 \cdot (g_i \omega_2)\circ T_i^{-1} \dLeb
  \\
\nonumber &
  = \sum_i \int_{O_i} (|\det DT_i| g_i \cdot \omega_1 \circ T_i) \cdot \omega_2 \dLeb.
  \end{align}
Therefore, the adjoint of $\LL_g$ is the operator
$\omega\mapsto \sum_i 1_{O_i} \Jac(T_i) g_i\cdot  \omega \circ
T_i$, to which Theorem \ref{main_extended} also applies (if
transversality with the unstable cones is satisfied). It is
sometimes more convenient to apply the theorem in this
direction, since its statement is not completely symmetric with
respect to the stable and unstable directions. An important
particular case, which will appear in Proposition ~
\ref{prop_futur} and its Corollary ~ \ref{goodforLozi}, and
which is useful when studying e.g. ~ Lozi maps, is when $g_i =
|\det DT_i|^{-1}$ for all $i$, and  the $O_i$ form a partition
of $X_0$. In this case, the dual operator is just
$\MM(\omega)=\omega \circ T$.

\begin{proof}[Sketch of proof of Theorem \ref{main_extended}]
The proof of Theorem \ref{MainTheorem} applies almost directly
to yield Theorem \ref{main_extended}, we should only modify
slightly the charts and the norm to take into account the fact
that the sets $\Pi_e$ and $O_i$ do not coincide, by introducing
an additional dependency on $e$.

More precisely, for every $i, e$, we can consider as in
Subsection \ref{spaces} charts $\kappa_{i,e,j}$ (for $1\leq
j\leq N_{i,e}$) whose domains of definitions $U_{i,e,j,0}$
cover a neighborhood of $\overline{\Pi_e}\cap \overline{O_i}$,
and extended cones $\CC_{i,e,j}$ such that, wherever
$\kappa_{i',e',j'}\circ T_i \circ \kappa_{i,e,j}^{-1}$ is
defined, its differential sends $\CC_{i,e,j}$ to
$\CC_{i',e',j'}$ compactly.

Let $U_{i,e,j,1}$ be a subset with compact closure of
$U_{i,e,j,0}$ such that the sets $U_{i,e,j,1}$ ($1\leq j\leq
N_{i,e}$) still cover $\overline{\Pi_e}\cap \overline{O_i}$. We
can then define sets $\ZZ_{i,e,j}(R)$ and
$\ZZ(R)=\{(i,e,j,m)\st m\in \ZZ_{i,e,j}(R)\}$ as in
\eqref{fset} and \eqref{fset2}. For $\zeta=(i,e,j,r) \in
\ZZ(R)$, let $\Pi_\zeta = \Pi_e$. We can then follow line by
line the discussion in Subsection \ref{spaces}, define a norm
  \begin{equation*}
  \norm{\omega}{\Phi}=\left(\sum_{\bm\in \ZZ(R)}
  \norm
  {(\brho_\bm(R) \cdot 1_{\Pi_\bm} \omega)\circ \bphi_\bm}
  {H_p^{t,s}}^p\right)^{1/p}
  \end{equation*}
for any system of charts $\Phi$, and finally put
$\norm{\omega}{\HHH_{p}^{t,s}(R, C_0, C_1)}
=\sup_{\Phi}\norm{\omega}{\Phi}$.

The proof of Theorem \ref{MainTheorembis} still works in this
context, with trivial notational modifications (one should
replace $1_{O_\zeta}$ and $1_{O_{\zeta'}}$ by $1_{\Pi_\zeta}$
and $1_{\Pi_{\zeta'}}$, and insert a characteristic function
$1_{\Pi_{\zeta'}}$ in \eqref{klqjsdfml}). The transversality of
the boundary of $\Pi_e$ with the stable cone is used at the
beginning of the second step to show that the multiplication by
$1_{\Pi_\zeta}$ is bounded on $H_p^{t,s}$, while the
transversality of the boundary of the image of $O_i$ with this
cone is used to show the same multiplier property for
$1_{T_\ii^n O_\ii}$.

Finally, the result follows from the analogue of Theorem
\ref{MainTheorembis}, by the arguments of Subsection
\ref{reducr}.
\end{proof}

\section{Physical measures}
\label{app_physical}

In this appendix, we discuss the existence of physical
measures, combining our main result Theorem \ref{MainTheorem} (or
its extension Theorem \ref{main_extended}), with Theorem 33 of
\cite{baladi_gouezel_piecewise}. The discussion is essentially
straightforward once the above results are given, apart from a
more subtle point: one should check that the possible physical
measures would give no mass to the discontinuity set of $T$.

Let us first give a convenient definition:
\begin{definition}
Let $T$ be a measurable map on an open subset $X_0$ with
compact closure of a manifold. A \emph{physical description} of
$T$ is a finite number of probability measures
$\mu_1,\dots,\mu_l$ which are $T$-invariant and ergodic, and
disjoint sets $A_1,\dots,A_l$ such that $\mu_i(A_i)=1$,
$\Leb(A_i)>0$, $\Leb(X_0\backslash \bigcup_{i=1}^l A_i)=0$ and,
for every $x\in A_i$ and every function $f\in C^0(X_0)$, we
have $\frac{1}{n}\sum_{j=0}^{n-1}f(T^j x) \to \int f\dd\mu_i$.
Moreover, for every $i$, there exist an integer $k_i$ and a
decomposition $\mu_i=\mu_{i,1}+\dots+\mu_{i,k_i}$ such that $T$
sends $\mu_{i,j}$ to $\mu_{i,j+1}$ for $j\in \Z/k_i\Z$, and the
probability measures $k_{i,j} \mu_{i,j}$ are mixing for
$T^{k_i}$.
\end{definition}
We could strengthen the requirements by requiring that the
measures $\mu_{i,j}$ are exponentially mixing for $T^{k_i}$
and H\"{o}lder observables,
and that all kinds of statistical limit theorems (central limit
theorem, strong invariance principle, etc.) are satisfied.
These additional properties will also hold in the
examples below.

Consider now a piecewise hyperbolic map $T$.
We will deal with a true (i.e.,
single-valued) map $T$, but we will not
necessarily assume that the continuity domains of the cone
families coincide with the continuity domains of $T$, as in
Appendix \ref{app_general}.

We give two results, corresponding to the application of our
main theorems in forward or backward time.

\begin{proposition}
\label{prop_past}
Let $T$ be a piecewise $C^{1+\alpha}$ hyperbolic map on a
domain $X_0$ with compact closure in a manifold $X$, such that
\begin{itemize}
\item the boundaries of the continuity domains of the cone
families are transverse to the stable cones,
\item the images under $T$ of the boundaries of the continuity
domains of $T$ are transverse to the stable cones.
\end{itemize}
Assume, for some $\beta\in (0,\alpha)$, the bunching condition
  \begin{equation*}
  \sup_{\ii \in I^n,\, q\in \overline O_\ii} \frac{
  \lambda^{(n)}_{\ii,s}(q)^{\alpha-\beta} \Lambda^{(n)}_{\ii,u}(q)^{1+\beta} }{\lambda^{(n)}_{\ii,u}(q)  }  <1
  \, .
\end{equation*}

Assume also that, for some parameters $p\in (1,\infty)$ and $t,
s \in \real$ with
  \begin{equation*}
  1/p-1<s<0<t<1/p\, ,
  \quad -\beta<t-|s| <0\, , \quad \alpha t+|s| < \alpha\, ,
  \end{equation*}
we have for some $n$
  \begin{equation}
  \label{D1}
  (D_{n}^b)^{1/(pn)} \cdot (D_{n}^e)^{(1/n)(1-1/p)}
  \cdot \norm{|\det DT^{n}|^{1/p -1}
  \max(\lambda_{u,n}^{-t},\lambda_{s,n}^{-(t-|s|)})}{L^\infty}^{1/n}
  < 1.
  \end{equation}
Then $T$ admits a physical description.
\end{proposition}

\begin{proposition}
\label{prop_futur}
Let $T$ be a piecewise $C^{1+\alpha}$ hyperbolic map on a
domain $X_0$ with compact closure in a manifold $X$, such that
\begin{itemize}
\item the boundaries of the continuity domains of the cone
families are transverse to the unstable cones,
\item the preimages under $T$ of the boundaries of the continuity
domains of $T$ are transverse to the unstable cones.
\end{itemize}
Assume, for some $\beta\in (0,\alpha)$, the bunching condition
  \begin{equation*}
  \sup_{\ii \in I^n,\, q\in \overline O_\ii} \frac{
  \lambda^{(n)}_{\ii,u}(q)^{\alpha-\beta} \Lambda^{(n)}_{\ii,s}(q)^{1+\beta} }{\lambda^{(n)}_{\ii,s}(q)  }  > 1
  \, .
\end{equation*}

Assume also that, for some parameters $p\in (1,\infty)$ and $t,
s \in \real$ with
  \begin{equation*}
  1/p-1<s<0<t<1/p\, ,
  \quad -\beta<|s|-t <0\, , \quad \alpha |s| + t < \alpha\, ,
  \end{equation*}
we have for some $n$
  \begin{equation}
  \label{D2}
  (D_{n}^b)^{1/(pn)} \cdot (D_{n}^e)^{(1/n)(1-1/p)}
  \cdot \norm{|\det DT^{n}|^{1/p -1}
  \max(\lambda_{u,n}^{|s|-t},\lambda_{s,n}^{|s|})}{L^\infty}^{1/n}
  < 1.
  \end{equation}
Then $T$ admits a physical description.
\end{proposition}

In both propositions, if $D_n^b$ and $D_n^e$ grow
subexponentially fast and $\det DT \equiv 1$, the limit in
\eqref{D1} and \eqref{D2} is $<1$ for any valid choice of
parameters $p,s,t$. If $\det DT\not\equiv 1$, one should choose
the parameters more carefully, as in the next corollary.
\begin{corollary}\label{goodforLozi}
If $D_n^b$ and $D_n^e$ grow subexponentially fast, $d_s = 1$
and the transversality conditions of Proposition
\ref{prop_futur} are satisfied, then $T$ admits a physical
description.
\end{corollary}
\begin{proof}
Since $d_s=1$, we can fix $\beta>0$ such that the bunching
condition of Proposition \ref{prop_futur} is satisfied. Then
the limit in \eqref{D2} is $<1$ if $p$ is very close to $1$,
$t=\beta/2$ and $s=1/p-1+\epsilon$ for some very small
$\epsilon$: the easy computation is the same as in
\cite[Example 3]{baladi_gouezel_piecewise}. Therefore,
Proposition \ref{prop_futur} gives the result.
\end{proof}

We state here the slightly stronger version of \cite[Theorem
33]{baladi_gouezel_piecewise}  that we shall need to prove the
two propositions above:

\begin{theorem}
\label{thm_physical}
Let $T$ be a nonsingular measurable map on an open subset $X_0$
with compact closure in a manifold $X$. Let us define its
transfer operator $\LSRB$ by $\int_{X_0} \LSRB u\cdot
v\dLeb=\int_{X_0} u\cdot v\circ T \dLeb$ whenever $v$ is
bounded and measurable. It is given by $\LSRB u(x)=\sum_{Ty=x}
|\det DT(y)|^{-1}u(y)$.

Let $H_0$ be a vector subspace of $L^\infty(\Leb)$, endowed
with a (possibly non--complete) norm $\nor{\cdot}$. Assume that
\begin{enumerate}
\item There exist $\alpha>0$ and $C>0$ such that, for any $u\in H_0$
and $f\in C^\alpha(X)$, then $fu\in H_0$ and $\nor{fu}\leq
C\norm{f}{C^\alpha}\nor{u}$.
\item There exists $C>0$ such that, for any $u\in H_0$, $\left|\int u\dLeb \right| \leq
C\nor{u}$.
\item The transfer operator $\LSRB$ associated to $T$ sends
$H_0$ to itself, and satisfies $\nor{\LSRB u} \leq C
\nor{u}$. Therefore, $\LSRB$ admits a continuous extension
to the completion $H$ of $H_0$ (still denoted by $\LSRB$).
We assume that the essential spectral radius of this
extension is $<1$, and that the iterates of $\LSRB$ are
uniformly bounded.
\item There exist $f_0\in H_0$ taking its values in $[0,1]$ and $N_0>0$ such
that $f_0=1$ on $T^{N_0}(X_0)$.
\item For any $u\in H$ which is a limit of nonnegative functions $u_n\in
H_0$ and for which there exists a measure $\mu_u$ such that
$\langle u, g\dLeb \rangle=\int g\dd\mu_u$ for any
$C^\alpha$ function $g$, then the measure $\mu_u$ gives
zero mass to the discontinuity set of $T$.
\end{enumerate}
Then $T$ admits a physical description.
\end{theorem}
In the fifth point, $\langle u, g\dLeb \rangle$ is defined as
follows. A function $u\in H_0$ can be multiplied by $g$ and
then integrated against Lebesgue measure. Those operations are
continuous for the norm (by the first and second assumption),
and therefore extend to $H$.

Theorem
\ref{thm_physical} is stronger than \cite[Theorem
33]{baladi_gouezel_piecewise} for the following reasons:
\begin{itemize}
\item We do not assume that the space $H$ is a space of distributions,
i.e., there may be elements $u\in H$ with $\langle u,
g\dLeb\rangle = 0$ for any $C^\infty$ function $g$. The
space $H_0$ used in the proof of Proposition
\ref{prop_past} is a space of distributions, but this is
not clear for the space $H_0$ used in the proof of
Proposition \ref{prop_futur} (it would be true if $C^1$
were dense in $\HHH$, but we do not know if this holds).
This is why we had to abstain from using this assumption in
Theorem \ref{thm_physical}.
\item The conclusion ``$T$ admits a physical description''
gives the convergence of Birkhoff sums for all continuous
functions, while \cite[Theorem
33]{baladi_gouezel_piecewise} obtains such a convergence
only for functions in the closure of $H_0$ for the $C^0$
norm.
\end{itemize}

We next  show how to reduce
Theorem \ref{thm_physical} to  \cite[Theorem
33]{baladi_gouezel_piecewise}.
\begin{proof}[Proof of Theorem \ref{thm_physical}]
We first deal with the second issue, that \cite[Theorem
33]{baladi_gouezel_piecewise} proves the convergence of
Birkhoff sums only for functions in the closure of $H_0$ in the
$C^0$ norm. In fact, the proof in
\cite{baladi_gouezel_piecewise} gives this convergence for any
countable family of functions in $H_0$. Let $g_n$ be a family
of $C^\alpha$ functions, dense in $C^0$. We obtain the
convergence of Birkhoff sums for all the functions $g_n f_0$,
since they all belong to $H_0$ by assumption. Moreover, for all
$k\geq N_0$, $(g_n f_0) \circ T^k = g_n\circ T^k$. Therefore,
the convergence of Birkhoff sums also holds for all the
functions $g_n$. Since they are dense in $C^0$, this concludes
the proof.

Let us now deal with the first problem, that $H$ is not
necessarily a space of distributions. Let $G\subset H$ be the
problematic subspace, i.e., $G=\{ u\in H \st \langle u,
g\dLeb\rangle = 0 \text{ for all }g\in C^\alpha\}$. If
$G=\{0\}$, the results of \cite{baladi_gouezel_piecewise}
directly apply, otherwise we have to eliminate it. We can not
work directly with the quotient space $H/G$, since it is
possible that $G$ is not invariant under $\LSRB$. On the other
hand, for $|\lambda|= 1$, let $E_\lambda \subset H$ be the
eigenspace of $\LSRB$ for the eigenvalue $\lambda$, then
$F_\lambda=E_\lambda\cap G$ is invariant under $\LSRB$. All the
arguments in \cite{baladi_gouezel_piecewise} then apply on
$H/\bigoplus F_\lambda$ (modulo straightforward adjustments).
\end{proof}

\begin{proof}[Proof of Proposition \ref{prop_past}]
By Theorem \ref{MainTheorem}, under the assumptions of the
proposition, we may construct a Banach space $\HHH$ (of
distributions) on which the essential spectral radius of
$\LSRB$ (as defined in the statement of Theorem~
\ref{thm_physical}) is $<1$. To simplify notations, we will
pretend that $\HHH$ is the space $\HHH_p^{t,s}(R, C_0, C_1)$,
and not the more complicated space constructed using
\eqref{normeHnR}.

We wish to apply Theorem \ref{thm_physical} to $H_0 = \HHH\cap
L^\infty(\Leb)$, to obtain the conclusion of the proposition.
The first four assumptions of this theorem are trivial, but the
fifth one should be checked more carefully. The norm in $\HHH$
is a supremum of norms along admissible charts. Let us fix one
such chart, and consider $\tilde H$ the space obtained by using
only the norm in this chart. This space is not interesting from
the dynamical point of view (it is not invariant under
$\LSRB$), but $\HHH$ is continuously contained in $\tilde H$.
Moreover, \cite[Lemma 34]{baladi_gouezel_piecewise} shows that,
if an element $u\in \tilde H$ satisfies $\langle u, g\dLeb
\rangle = \int g\dd\mu_u$ for some nonnegative measure $\mu_u$,
then $\mu_u$ gives zero mass to the discontinuities of $T$.
Since $\HHH$ is smaller than $\tilde H$, this readily implies
the same result for $\HHH$.
\end{proof}

\begin{proof}[Proof of Proposition \ref{prop_futur}]
Consider the operator $\MM u =u \circ T$. This operator is
obtained locally by composing with hyperbolic maps, therefore
we may apply Theorem \ref{main_extended} to it (under suitable
transversality assumptions, that are exactly those of
Proposition \ref{prop_futur}) -- one should simply be careful
with notations, since stable and unstable directions are
exchanged. The assumption \eqref{D2} ensures that the essential
spectral radius of $\MM$ on the space $\HHH$ constructed in
Theorem \ref{main_extended} (for the parameters $p'=p/(p-1)$,
$s'=-t$ and $t'=-s$) is $<1$. Moreover, since $\HHH$ is a space
of distributions, one may prove as in the first step of the
proof of \cite[Theorem 33]{baladi_gouezel_piecewise} that there
is no eigenvalue of modulus $>1$ and no Jordan block for the
eigenvalues of modulus $1$, i.e., the iterates of $\MM$ on
$\HHH$ are uniformly bounded. As above, we will pretend that
$\HHH=\HHH_{p'}^{t',s'}(R, C_0, C_1)$ to simplify notations.

Define a (possibly infinite) norm $\nor{\cdot}$ (dual to the
$\HHH$--norm) on $L^\infty(\Leb)$ by
  \begin{equation}
  \nor{u} = \sup_{ v\in \HHH\cap L^\infty(\Leb),\ \norm{v}{\HHH}\leq 1} \left| \int uv\dLeb \right|,
  \end{equation}
and let $H_0$ be the set of elements of $L^\infty(\Leb)$ with
$\nor{u}<\infty$. Since the dual of $\MM$ is $\LSRB$
(as defined in the statement of
Theorem~ \ref{thm_physical}), it
follows that $\LSRB$ leaves $H_0$ invariant, that its essential
spectral radius on the completion of $H_0$ is $<1$, and that
the iterates of $\LSRB$ are uniformly bounded.

We wish to apply Theorem \ref{thm_physical} to this space
$H_0$, to conclude the proof. As above, the first four
conditions of this theorem are easily checked, but we should be
more careful for the last one.

For any hypersurface $Q$ bounding a domain $O_i$, consider a
decreasing sequence $K_n(Q)$ of neighborhoods of $Q$, with
sides parallel to $Q$ (in local coordinate charts), and
converging to $Q$. It follows from the argument in the second
step of the proof of Lemma \ref{maintechlemma} that the
$\HHH_{p'}^{t',s'}(R, C_0, C_1)$ norm of $1_{K_n(Q)}$ in any
admissible chart is uniformly bounded. Therefore,
$\norm{1_{K_n(Q)}}{\HHH} \leq C$ for some constant $C$
independent of $n$. The same argument even shows that
$1_{K_n(Q)}$ is uniformly bounded in the space
$\HHH_{p'}^{t'',s'}(R, C_0, C_1)$ if $t''\in (t', 1/p')$. By
Lemma \ref{embed}, this space is compactly included in $\HHH$,
therefore the sequence $1_{K_n(Q)}$ is compact in $\HHH$. Any
of its cluster values has to be $0$ as a distribution (since
$\Leb(K_n(Q)) \to 0$). Since $\HHH$ is a space of
distributions, it follows that all the cluster values of
$1_{K_n(Q)}$ are $0$, hence $1_{K_n(Q)}$ tends to $0$ in
$\HHH$.

Let $K_n$ be the (finite) union of the $K_n(Q)$ for all
boundary hypersurfaces of the sets $O_i$. Then $K_n$ contains
the discontinuity set of $T$ in its interior, and $1_{K_n}$
tends to $0$ in $\HHH$.

Consider $u$ in the completion $H$ of $H_0$, such that $u$ is a
limit of nonnegative functions $u_m$, and such that, for some
measure $\mu_u$, we have $\langle u, g\dLeb \rangle = \int
g\dd\mu_u$ for any $C^\alpha$ function $g$. Consider a
$C^\alpha$ function $g$ such that $0\leq g \leq 1_{K_n}$. We
have $\langle u_m, g \rangle \leq \langle u_m, 1_{K_n} \rangle$
since $u_m$ is a nonnegative function. Letting $m$ tend to
infinity, we get $\langle u, g\rangle \leq \langle u, 1_{K_n}
\rangle \leq \nor{u} \norm{1_{K_n}}{\HHH}$. Choosing $g$ equal
to $1$ on the discontinuity set of $T$, we get $\mu_u(\Disc T)
\leq \nor{u} \norm{1_{K_n}}{\HHH}$. Since this quantity tends
to $0$ when $n\to\infty$, this concludes the proof.
\end{proof}

\bibliography{biblio}

\providecommand{\bysame}{\leavevmode\hbox to3em{\hrulefill}\thinspace}
\providecommand{\MR}{\relax\ifhmode\unskip\space\fi MR }
\providecommand{\MRhref}[2]{%
  \href{http://www.ams.org/mathscinet-getitem?mr=#1}{#2}
}
\providecommand{\href}[2]{#2}
\begin{thebibliography}{BKL02}

\bibitem[Bal05]{baladi_Cinfty}
Viviane Baladi, \emph{Anisotropic {S}obolev spaces and dynamical transfer
  operators: {$C\sp \infty$} foliations}, Algebraic and topological dynamics,
  Contemp. Math., vol. 385, Amer. Math. Soc., Providence, RI, 2005,
  pp.~123--135. \MR{MR2180233}

\bibitem[BG09]{baladi_gouezel_piecewise}
Viviane Baladi and S{\'e}bastien Gou{\"e}zel, \emph{Good {B}anach spaces for
  piecewise hyperbolic maps via interpolation}, Ann. Inst. H. Poincar\'e Anal.
  Non Lin\'eaire \textbf{26} (2009), no.~4, 1453--1481. \MR{MR2542733}

\bibitem[BKL02]{bkl_spectre_anosov}
Michael Blank, Gerhard Keller, and Carlangelo Liverani,
  \emph{Ruelle-{P}erron-{F}robenius spectrum for {A}nosov maps}, Nonlinearity
  \textbf{15} (2002), 1905--1973. \MR{MR1938476 (2003m:37033)}

\bibitem[BR96]{baru_sharp}
Viviane Baladi and David Ruelle, \emph{Sharp determinants}, Invent. Math.
  \textbf{123} (1996), no.~3, 553--574. \MR{MR1383961 (97a:58146)}

\bibitem[BT07]{bt_aniso}
Viviane Baladi and Masato Tsujii, \emph{Anisotropic {H}\"older and {S}obolev
  spaces for hyperbolic diffeomorphisms}, Ann. Inst. Fourier (Grenoble)
  \textbf{57} (2007), no.~1, 127--154. \MR{MR2313087 (2008d:37034)}

\bibitem[BT08]{bt_zeta}
\bysame, \emph{Dynamical determinants and spectrum for hyperbolic
  diffeomorphisms}, Geometric and probabilistic structures in dynamics,
  Contemp. Math., vol. 469, Amer. Math. Soc., Providence, RI, 2008, pp.~29--68.
  \MR{MR2478465}

\bibitem[Che99]{chernov_decay}
Nikolai Chernov, \emph{Decay of correlations and dispersing billiards}, J.
  Statist. Phys. \textbf{94} (1999), 513--556. \MR{MR1675363 (2000j:37044)}

\bibitem[Che07]{chernov_stretched_flow}
\bysame, \emph{A stretched exponential bound on time correlations for billiard
  flows}, J. Stat. Phys. \textbf{127} (2007), no.~1, 21--50. \MR{MR2313061
  (2008b:37062)}

\bibitem[DL08]{demers_liverani}
Mark~F. Demers and Carlangelo Liverani, \emph{Stability of statistical
  properties in two-dimensional piecewise hyperbolic maps}, Trans. Amer. Math.
  Soc. \textbf{360} (2008), no.~9, 4777--4814. \MR{MR2403704 (2009f:37021)}

\bibitem[Dol98]{dolgopyat_decay}
Dmitry Dolgopyat, \emph{On decay of correlations in {A}nosov flows}, Ann. of
  Math. (2) \textbf{147} (1998), no.~2, 357--390. \MR{MR1626749 (99g:58073)}

\bibitem[GL06]{gouezel_liverani}
S{\'e}bastien Gou{\"e}zel and Carlangelo Liverani, \emph{Banach spaces adapted
  to {A}nosov systems}, Ergodic Theory Dynam. Systems \textbf{26} (2006),
  no.~1, 189--217. \MR{MR2201945 (2007h:37037)}

\bibitem[GL08]{GL_Anosov2}
\bysame, \emph{Compact locally maximal hyperbolic sets for smooth maps: fine
  statistical properties}, J. Differential Geom. \textbf{79} (2008), no.~3,
  433--477. \MR{MR2433929}

\bibitem[Hen93]{hennion}
Hubert Hennion, \emph{Sur un th\'eor\`eme spectral et son application aux
  noyaux lipchitziens}, Proc. Amer. Math. Soc. \textbf{118} (1993), no.~2,
  627--634. \MR{MR1129880 (93g:60141)}

\bibitem[HK95]{katok}
Boris Hasselblatt and Anatole Katok, \emph{Introduction to the modern theory of
  dynamical systems}, Encyclopedia of Mathematics and its Applications,
  vol.~54, Cambridge University Press, Cambridge, 1995, With a supplementary
  chapter by Katok and Leonardo Mendoza. \MR{MR1326374 (96c:58055)}

\bibitem[HPS77]{hirsch_pugh_shub}
Morris~W. Hirsch, Charles~C. Pugh, and Michael Shub, \emph{Invariant
  manifolds}, Lecture Notes in Mathematics, Vol. 583, Springer-Verlag, Berlin,
  1977. \MR{MR0501173 (58 \#18595)}

\bibitem[Liv04]{liverani_contact}
Carlangelo Liverani, \emph{On contact {A}nosov flows}, Ann. of Math. (2)
  \textbf{159} (2004), no.~3, 1275--1312. \MR{MR2113022 (2005k:37048)}

\bibitem[Str67]{strichartz}
Robert~S. Strichartz, \emph{Multipliers on fractional {S}obolev spaces}, J.
  Math. Mech. \textbf{16} (1967), 1031--1060. \MR{MR0215084 (35 \#5927)}

\bibitem[Tri77]{triebel_III}
Hans Triebel, \emph{General function spaces. {III}. {S}paces
  {$B\sb{p,q}\sp{g(x)}$} and {$F\sb{p,q}\sp{g(x)}$}, {$1<p<\infty $}: basic
  properties}, Anal. Math. \textbf{3} (1977), no.~3, 221--249. \MR{MR0628468
  (58 \#30155c)}

\bibitem[You98]{lsyoung_annals}
Lai-Sang Young, \emph{Statistical properties of dynamical systems with some
  hyperbolicity}, Ann. of Math. (2) \textbf{147} (1998), no.~3, 585--650.
  \MR{MR1637655 (99h:58140)}

\end{thebibliography}
\bibliographystyle{amsalpha}

\end{document}